\documentclass{amsart}

\usepackage{amsfonts, amssymb, amsmath, amsthm}                               % AMS symbols

\newtheorem{theorem}{Theorem} 	      	      	                              % Theorem environment
\newtheorem{corollary}[theorem]{Corollary}     	      	      	      	      % Corollary environment
\newtheorem{lemma}[theorem]{Lemma}     	       	      	      	      	      % Lemma environment
\newtheorem{proposition}[theorem]{Proposition} 	      	      	      	      % Proposition environment
\newtheorem*{remark}{Remark}                                                  % Remark environment

\newcommand{\ol}[1]{\overline{#1}{}}                                          % Overline
\newcommand{\ul}[1]{\underline{#1}{}}                                         % Underline
\newcommand{\mf}[1]{\mathfrak{#1}}                                            % Mathfrak
\newcommand{\mc}[1]{\mathcal{#1}}                                             % Mathcal

\newcommand{\paren}[1]{\left(#1\right)}                                       % Resized parentheses
\newcommand{\brak}[1]{\left[#1\right]}                                        % Resized brackets
\newcommand{\brac}[1]{\left\{#1\right\}}                                      % Resized braces
\newcommand{\brakparen}[1]{\left[#1\right)}                                   % Left bracket, right parenthesis
\newcommand{\abs}[1]{\left|#1\right|}                                         % Resized absolute value
\newcommand{\nabs}[1]{\left\|#1\right\|}                                      % Resized norm
\newcommand{\valat}[1]{\left.#1\right|}                                       % Value at

\newcommand{\R}{\mathbb{R}}                                                   % Real numbers
\newcommand{\Sph}{\mathbb{S}}                                                 % Unit sphere

\DeclareMathOperator{\trace}{tr}                                              % Trace

\DeclareMathOperator{\grad}{grad}                                             % Gradient
\newcommand{\lapl}{\Delta}                                                    % Laplacian
\newcommand{\nasla}{/\mspace{-11mu}\nabla}                                    % Slashed nabla
\newcommand{\onasla}{/\mspace{-11mu}\ol{\nabla}}                              % Slashed overlined nabla
\newcommand{\lasl}{/\mspace{-11mu}\lapl}                                      % Slashed Laplacian
\DeclareMathOperator{\ric}{Ric}                                               % Ricci curvature

\newcommand{\ass}[2]{({\bf #1})${}_{#2}$}                                     % Assumption condition
\newcommand{\sss}[2]{#1^{\scriptscriptstyle (#2)}{}}                          % System superscript 

\allowdisplaybreaks

\begin{document}

\title[Breakdown Criteria]{On Breakdown Criteria for Nonvacuum Einstein Equations}
\author{Arick Shao}
\address{Department of Mathematics, Princeton University, Princeton NJ 08544}
\email{aricks@math.princeton.edu}

\begin{abstract}
The recent ``breakdown criterion" result \cite{kl_rod:bdc} of S. Klainerman and I. Rodnianski stated roughly that an Einstein-vacuum spacetime, given as a CMC foliation, can be further extended in time if the second fundamental form and the derivative of the lapse of the foliation are uniformly bounded.
This theorem and its proof were extended to Einstein-scalar and Einstein-Maxwell spacetimes in the thesis \cite{shao:bdc_nv}.
In this paper, we state the main results of \cite{shao:bdc_nv}, and we summarize and discuss their proofs.
In particular, we will discuss the various issues resulting from nontrivial Ricci curvature and the coupling between the Einstein and the field equations.
\end{abstract}

\maketitle

\section{Introduction}

The general breakdown/continuation problem for PDE is the following:
\begin{itemize}
\item[] {\it Under what conditions can an existing local solution of an evolution equation on a finite interval $[T_0, T)$ be further continued past $T$?}
\end{itemize}
This can be equivalently posed as the breakdown of such a solution at a finite time implying the violation of such conditions.
The determination of such \emph{breakdown criteria} can be a potentially useful step toward characterizing the blowup of solutions.
Furthermore, in some instances, such breakdown conditions can be critical tools for proving global existence results.

In this paper, we will consider this breakdown problem for the Einstein-scalar (E-S) and Einstein-Maxwell (E-M) equations in the CMC gauge.
The results and proofs described in this paper are extensions of those in \cite{kl_rod:bdc}, which established the analogous breakdown criterion for the Einstein-vacuum (E-V) case.
The matter field present in these nonvacuum cases presents additional issues to be addressed.
The full details of this work can be found in the original report \cite{shao:bdc_nv}.

\subsection{Classical Results}

Consider, as a model example, the initial value problem for the following nonlinear wave equation on $\R^{1+3}$:
\begin{equation}\label{eq.bdc_model_wave} \square \phi = \paren{\partial_t \phi}^2 \text{,} \qquad \valat{\phi}_{t = 0} = \phi_0 \text{, } \valat{\partial_t \phi}_{t = 0} = \phi_1 \text{.} \end{equation}
From classical theory, cf. \cite[Thm. 22]{selb:wave_eq}, we have the following results:
\begin{itemize}
\item Given initial data $(\phi_0, \phi_1) \in Z_s = H^s(\R^3) \times H^{s-1}(\R^3)$, where $s > 5/2$, then a unique solution to \eqref{eq.bdc_model_wave} exists in the space
\[ X_{s, T^\prime} = C\paren{\brak{0, T^\prime}; H^s\paren{\R^3}} \cap C^1\paren{\brak{0, T^\prime}; H^{s-1}\paren{\R^3}} \]
for sufficiently small $T^\prime > 0$ depending on the $Z_s$-norm of the initial data.

\item The maximal time of existence $T$, i.e., the supremum of such $\tau$'s for which a solution exists to time $\tau$, depends on the $Z_s$-norm of the initial data.
\end{itemize}

A continuation result follows naturally as a companion to local well-posedness:
\begin{itemize}
\item[] {\it Suppose a solution $\phi$ on a finite time interval $[0, T)$ to \eqref{eq.bdc_model_wave} is in $X_{s, T^\prime}$ for every $0 < T^\prime < T$, where $s > 5/2$.
If $\phi$ also satisfies the criterion
\begin{equation}\label{eq.bdc_model_wave_bdc} \nabs{\partial \phi}_{L^\infty\paren{\brakparen{0, T} \times \R^3}} < \infty \text{,} \end{equation}
then $\phi$ can be extended past time $T$ as a solution of \eqref{eq.bdc_model_wave}.
Moreover, this extension is an element of $X_{s, T + \epsilon}$ for small $\epsilon > 0$.}
\end{itemize}
In terms of breakdown, this can be equivalently stated as follows:
\begin{itemize}
\item[] {\it Suppose a solution $\phi$ on a finite time interval $[0, T)$ to \eqref{eq.bdc_model_wave} is in $X_{s, T^\prime}$ for every $0 < T^\prime < T$, where $s > 5/2$.
Then, if $\phi$ breaks down at time $T$, i.e., if $\phi$ cannot be extended as above, then $\partial \phi \not\in L^\infty([0, T) \times \R^3)$.}
\end{itemize}

The main idea behind the result is the following observation: if \eqref{eq.bdc_model_wave_bdc} holds, then we can uniformly bound the $H^s(\R^3) \times H^{s-1}(\R^3)$-norms of $(\phi, \partial_t \phi)$ on each timeslice $\{\tau\} \times \R^3$, where $0 < \tau < T$.
Therefore, we can apply the previous local well-posedness result to generate local solutions existing for a fixed time $\epsilon > 0$ with each of the above cross-sections as the initial data.
By uniqueness, we can patch these local solutions into a solution which exists on the interval $[0, T + \epsilon)$.

\begin{remark}
The condition \eqref{eq.bdc_model_wave_bdc} is certainly not optimal, since slightly weaker iterated norm conditions are also known to be sufficient.
\end{remark}

A multitude of breakdown results have been established for other evolution equations.
For example, consider the incompressible $3$-dimensional Euler equations
\begin{equation}\label{eq.euler_eq} \partial_t u + u \cdot \nabla u + \nabla p \equiv 0 \text{,} \qquad \nabla \cdot u \equiv 0 \text{,} \end{equation}
where $u: \R^{1 + 3} \rightarrow \R^3$ represents the velocity and $p: \R^{1 + 3} \rightarrow \R$ represents the pressure.
In addition, define the \emph{vorticity} of $u$ to be the curl $\omega = \nabla \times u$ of $u$.

A well-known result of Beale, Kato, and Majda in \cite{be_ka_maj:bdc_eul} established that if a local finite-time solution has its vorticity bounded in the $L^1_t L^\infty_x$-norm, then the solution can be extended further in time.
\footnote{The $L^1_t L^\infty_x$-norm of $u$ is the $L^1$-norm of the function $t \mapsto \|u(t)\|_{L^\infty(\R^3)}$.}
An important point here is that unlike the condition \eqref{eq.bdc_model_wave_bdc}, we need not bound all components of the derivative of the solution.
The proof is in principle like that of the nonlinear wave equation; we use this $L^1_t L^\infty_x$-bound on the vorticity in order to derive uniform energy bounds related to the local well-posedness theory of the equations \eqref{eq.euler_eq}.

Another example of a breakdown condition lies in the paper \cite{ea_mo:g_ymh} of Eardley and Moncrief on the Yang-Mills equation in $\R^{1+3}$.
\footnote{For simplicity, we neglect the Higgs field, which was also discussed in \cite{ea_mo:g_ymh}.}
In this setting, a sufficient continuation criterion is an $L^\infty$-bound on the Yang-Mills curvature $F$.
Using the standard representation formula for the wave equation, however, one can demonstrate that such a uniform bound always holds.
The immediate consequence, then, is a global existence result.
Furthermore, Chru\'sciel and Shatah, in \cite{chru_sh:ym_curv}, generalized this result to globally hyperbolic $(1+3)$-dimensional Lorentzian manifolds using mostly the same principles, but applying instead the representation formula of \cite{fr:wv_eq}.

\subsection{Results in General Relativity}

In general relativity, a number of breakdown results have been established for the E-V equations.
For example, by solving the equations in the standard fashion by imposing favorable gauge conditions, such as wave coordinates, one can show that given a solution in which an $L^\infty$-bound holds for $\partial \mf{G}$, where $\mf{G}$ is the spacetime metric and $\partial$ refers to the gauge coordinate derivatives, then the solution can be further continued.
Examples of such results include \cite{cb:le_einst} and, more recently, \cite{and_mo:gr_cmc}.
Although such a condition is quite analogous to the model case of the nonlinear wave equation, it is also non-geometric, as it depends on the specific choice of coordinates.
Moreover, the condition requires bounds on all components of the derivative of the metric.

Another more geometric breakdown result for the E-V equations was given by M. Anderson in \cite{and:gr_geom}.
Here, the continuation criterion is an $L^\infty$-bound on the curvature of the spacetime.
While this is clearly geometric, it does have the added disadvantage of depending on essentially two derivatives of the spacetime metric.

The next point of discussion is the improved breakdown result of S. Klainerman and I. Rodnianski for E-V spacetimes, presented in \cite{kl_rod:bdc}.
Their main result, stated in \cite[Thm. 1.1]{kl_rod:bdc}, can be summarized as follows:

\begin{theorem}\label{thm.bdc_vacuum}
Suppose $(M, g)$ is an E-V spacetime, given as a CMC foliation
\[ M = \bigcup_{t_0 < \tau < t_1} \Sigma_\tau \text{,} \qquad t_0 < t_1 < 0 \text{,} \]
where each $\Sigma_\tau$ is a compact spacelike hypersurface of $M$ satisfying the constant mean curvature condition $\trace k \equiv \tau$, and where $k$ denotes the second fundamental form of $\Sigma_\tau$ in $M$.
In addition, let $n$ denote the lapse of the $\Sigma_\tau$'s, and assume the following breakdown criterion holds:
\begin{equation}\label{eq.bdc_vacuum} \nabs{k}_{L^\infty\paren{M}} + \nabs{\nabla \paren{\log n}}_{L^\infty\paren{M}} < \infty \text{.} \end{equation}
Then, the spacetime $(M, g)$ can be extended as an E-V spacetime in the CMC gauge to some time $t_1 + \epsilon$ for some $\epsilon > 0$.
\end{theorem}

We make note of the following important features of Theorem \ref{thm.bdc_vacuum}:
\begin{itemize}
\item This result is more geometric than the previous coordinate breakdown condition, whose statement required the choice of an entire coordinate system.
Theorem \ref{thm.bdc_vacuum}, on the other hand, relies only on the (constant mean curvature) time foliation to state the breakdown criterion.

\item Both $k$ and $\nabla (\log n)$ in \eqref{eq.bdc_vacuum} reside at the level of one derivative of $g$.
Also, these quantities do not represent all components of one derivative of $g$.
\end{itemize}

The proof of Theorem \ref{thm.bdc_vacuum} is a complicated affair which in its totality spans several papers: \cite{kl_rod:cg, kl_rod:glp, kl_rod:stt, kl_rod:ksp, kl_rod:rin, kl_rod:bdc, wang:cg, wang:cgp}.
One important point of the proof is that although the associated local well-posedness problem applies to objects on the timeslices, much of the technical work revolves around \emph{spacetime} objects, in particular the spacetime curvature.
The relevant spacetime and timeslice objects can then be related using standard elliptic estimates.

Using mostly the same set of ideas as in \cite{kl_rod:bdc}, D. Parlongue, in \cite{parl:bdc}, proved an analogous breakdown criterion in the setting of a maximal time foliation with asymptotically flat timeslices.
In addition, the breakdown condition was weakened from an $L^\infty$-bound to an iterated $L^2_t L^\infty_x$-bound:
\[ \nabs{k}_{L^2_t L^\infty_x \paren{M}} + \nabs{\nabla \paren{\log n}}_{L^2_t L^\infty_x \paren{M}} < \infty \text{.} \]
Furthermore, there have been additional results by Q. Wang toward an improved $L^1_t L^\infty_x$-criterion in the original CMC setting; see \cite{wang:ibdc, wang:tbdc}.

\subsection{Nonvacuum Spacetimes}

The main question posed here is the following:
\begin{itemize}
\item[] {\it Do breakdown results analogous to that of Theorem \ref{thm.bdc_vacuum} hold for Einstein-scalar and Einstein-Maxwell spacetimes?}
\end{itemize}
In this paper, we will answer this question affirmatively.

The breakdown criteria in these nonvacuum cases remain largely the same as in the vacuum case.
We retain the uniform bounds on the time foliation quantities $k$ and $\nabla (\log n)$.
However, we must also impose a uniform bound on the now nontrivial matter field: the scalar field in the E-S setting, or the spacetime Maxwell $2$-form in the E-M setting.
This will be the only additional condition.

In the vacuum case of \cite{kl_rod:bdc}, the breakdown criterion in conjunction with the CMC gauge imply a wide range of priori controls with respect to the time foliation, all of which are essential for proving Theorem \ref{thm.bdc_vacuum}.
These include, for example, uniform $L^2$-bounds for curvature along the timeslices, as well as uniform bounds for Sobolev constants on the timeslices.
In the E-S and E-M cases, we will also need uniform bounds on the matter fields in order to achieve the a priori controls mentioned above.
This will justify the modified breakdown criterion which we shall adopt.

We note that many of the elements which made the vacuum case so difficult will also make our nonvacuum cases similarly demanding.
For instance, both \cite{kl_rod:bdc} and the main result here are essentially ``large data" results.
Therefore, the only a priori ``energy estimates" we will have are lower-order $L^2$-bounds on the curvature and the matter field.
This greatly complicates the process behind controlling the local null geometry, since both the null injectivity radius and the Ricci coefficients will need to be controlled by these $L^2$-quantities.
In particular, this will necessitate the use of the geometric Littlewood-Paley theory and the associated Besov estimates developed in \cite{kl_rod:glp} and further discussed in \cite[Sec. 2.2]{shao:bdc_nv}.

The addition of matter fields also introduces a number of new difficulties:
\begin{itemize}
\item The spacetime Ricci curvature is now nontrivial.

\item We must also deal with the coupling between the curvature and the matter field.
An unavoidable consequence of this is that both the curvature and the matter field must be estimated concurrently.

\item In the E-M case, there exist first-order terms in the wave equations satisfied by the curvature and Maxwell field which were not present in \cite{kl_rod:bdc}.
These terms cannot be fully treated using only the techniques of \cite{kl_rod:bdc} and its subsidiary papers.
In particular, we will require a generalized representation formula for tensor wave equations; see \cite[Ch. 5]{shao:bdc_nv} and \cite{shao:ksp}.
\end{itemize}

Another unfortunate reality is that the assumption of an E-V spacetime was pervasive throughout the entire proof of Theorem \ref{thm.bdc_vacuum}.
\footnote{Notable exceptions include the geometric Littlewood-Paley theory of \cite{kl_rod:glp} and the Kirchhoff-Sobolev parametrix of \cite{kl_rod:ksp} (and its generalization in \cite[Ch. 5]{shao:bdc_nv} and \cite{shao:ksp}).}
As a consequence, most of the elements of the proof in the E-V case must in principle be redone.
In particular, this includes controlling the local null geometry.
In both the vacuum case of \cite{kl_rod:bdc} and the nonvacuum cases presented here, the work on the null geometry comprises the most lengthy and technically involved portion of the overall proof.
We will omit this portion of the proof from this paper, since it would more than double the length of this text.
For details in this area, see \cite{shao:bdc_nv}, as well as its predecessors in the vacuum setting: \cite{kl_rod:cg, kl_rod:stt, kl_rod:rin, parl:bdc, wang:cg, wang:cgp}.

The following statement summarizes the main theorem of this paper:

\begin{theorem}\label{thm.bdc_nv_pre}
Suppose $(M, g, \Phi)$ is an E-S or E-M spacetime, given by
\[ M = \bigcup_{t_0 < \tau < t_1} \Sigma_\tau \text{,} \qquad t_0 < t_1 < 0 \text{,} \]
where each $\Sigma_\tau$ is a compact spacelike hypersurface of $M$ satisfying the constant mean curvature condition $\trace k \equiv \tau$, and where $\Phi$ denotes the matter field:
\begin{itemize}
\item In the E-S case, $\Phi$ denotes the scalar field $\phi$.

\item In the E-M case, $\Phi$ denotes the Maxwell $2$-form $F$.

\item Let $k$ and $n$ be as before in the statement of Theorem \ref{thm.bdc_vacuum}.
\end{itemize}
In addition, let the quantity $\mf{F}$ denote the following:
\begin{itemize}
\item The spacetime covariant differential $D \phi$ of $\phi$, in the E-S case.

\item The Maxwell $2$-form $F$, in the E-M case.
\end{itemize}
Assume the following breakdown criterion holds:
\begin{equation}\label{eq.bdc_nv_pre} \nabs{k}_{L^\infty\paren{M}} + \nabs{\nabla \paren{\log n}}_{L^\infty\paren{M}} + \nabs{\mf{F}}_{L^\infty\paren{M}} < \infty \text{.} \end{equation}
Then, $(M, g, \Phi)$ can be extended as an E-S or E-M (resp.) spacetime in the CMC gauge to some time $t_1 + \epsilon$ for some $\epsilon > 0$.
\end{theorem}

A more technically precise version is stated as Theorem \ref{thm.bdc}, after a sufficient amount of background and notations have been developed.

\subsection{Extensions and Open Problems}

Finally, we discuss some possible extensions of Theorem \ref{thm.bdc_nv_pre}, as well as some related open problems.
The most immediate question is whether Theorem \ref{thm.bdc_nv_pre} can be similarly adapted to other related classes of nonvacuum spacetimes, such as Einstein-Klein-Gordon and Einstein-Yang-Mills spacetimes.
The answer is affirmative for the above two spacetimes in this same CMC gauge framework.
The class of techniques presented in this text can be applied directly to the Einstein-Klein-Gordon case.
This is also mostly true for the Einstein-Yang-Mills setting; however, in this case, we will also require vector bundle generalizations of many of the tools used here to handle the Yang-Mills curvature both in a covariant and in a gauge-invariant fashion.
\footnote{In particular, this includes the vector bundle extension of the generalized Kirchhoff-Sobolev parametrix; see \cite[Sec. 5.3]{shao:bdc_nv} and \cite[Sec. 5]{shao:ksp}.}

Whether analogous breakdown criteria can be stated and proved for less similar classes of nonvacuum spacetimes, such as the Einstein-Euler and the Einstein-Vlasov models, remains an open question.
For such matter fields, one would likely require other tools besides the representation formula for tensor wave equations in order to derive higher-order energy inequalities.

Another variation of Theorem \ref{thm.bdc_nv_pre} which should be possible is the case of a maximal foliation with asymptotically flat timeslices for various nonvacuum spacetimes.
The adaptations to the theorem statement and proof should then be analogous to \cite{parl:bdc}.
Note that one must also amend the energy norms for the curvature and the matter fields in this setting along the lines of \cite{parl:bdc}.

One more potential direction of investigation involves the possible weakening of the $L^\infty$-bounds in the breakdown criterion to slightly weaker iterated bounds, such as the $L^2_t L^\infty_x$ bounds found in \cite{parl:bdc}.
A further improvement to $L^1_t L^\infty_x$-bounds along the lines of \cite{wang:ibdc, wang:tbdc} is also expected to be achievable.

A more difficult open question related to the topic at hand is the well-known CMC conjecture.
In this case, the desired conclusion is that the spacetimes of Theorems \ref{thm.bdc_vacuum} and \ref{thm.bdc_nv_pre} can in fact be continued up to time $0$.
Finally, another related conjecture is the bounded $L^2$-curvature conjecture, which in principle states that a local well-posed theory exists based on control of the $L^2$-norms of the curvature.
There has been significant recent progress in resolving this question.

\subsection*{Acknowledgements}

The author wishes to thank Professor Sergiu Klainerman for suggesting this problem and for hours of discussions throughout the preparation of this work.
Thanks also extends to Professor Igor Rodnianski for his insights.
In addition, the author thanks Qian Wang for helpful technical discussions.

\section{Tensorial Notations}

In this section, we construct some notations for the various types of tensor fields we will encounter.
The main objective of these notations is to highlight the covariant structures present in our setting.
\footnote{By covariant structures, we mean vector bundles with metrics and compatible connections.}
Although such structures were also pertinent to the analogous vacuum problem of \cite{kl_rod:bdc}, as well as other related works (for example, \cite{chr_kl:stb_mink, kl_rod:cg, parl:bdc}), they were described more implicitly in those texts.
We develop the new and more explicit notations here in the hopes of promoting a more conceptual outlook on the objects of our analysis.
The conventions used here will be an abridged version of that of \cite{shao:bdc_nv}.

\subsection{Tensor Fields}

Let $M$ denote an arbitrary smooth manifold.
First, we set the notations for standard tensor bundles over $M$:
\begin{itemize}
\item Let $T^r_s M$ denote the rank-$(r, s)$ tensor bundle of $M$.

\item Let $S^s M$ denote the bundle of all fully symmetric elements of $T^0_s M$.

\item Let $\Lambda^s M$ denote the bundle of all fully antisymmetric elements of $T^0_s M$, i.e., the exterior bundle of degree $s$ on $M$.
\end{itemize}

In general, for a smooth vector bundle $\mc{V}$ over $M$, we define the following:
\begin{itemize}
\item Given $p \in M$, let $\mc{V}_p$ denote the fiber of $\mc{V}$ at $p$.
The lone exception is for tangent spaces, for which we use the more standard notation
\[ T_p M = \paren{T^1_0 M}_p \text{.} \]

\item Let $\Gamma \mc{V}$ denote the smooth sections of $\mc{V}$.
For example, $\Gamma T^r_s M$ and $\Gamma \Lambda^s$ are the spaces of rank-$(r, s)$ tensor fields and $s$-forms on $M$, respectively.
\end{itemize}
For convenience, we also define the abbreviations
\begin{equation}\label{eq.tensor_field_sp} \mf{X}\paren{M} = \Gamma T^1_0 M \text{,} \qquad \mf{X}^\ast\paren{M} = \Gamma T^0_1 M \text{,} \qquad \Gamma \mc{T} M = \bigcup_{r, s \geq 0} \Gamma T^r_s M \text{,} \end{equation}
i.e., the space of vector fields on $M$, the space of $1$-forms on $M$, and the space of all tensor fields on $M$, respectively.

When convenient, we will adopt standard index notation to represent tensor fields with respect to local frames and dual coframes.
In many instances, these frames satisfy additional conditions, e.g., coordinate, orthonormal, or null frames.
In accordance with Einstein summation notation, indices repeated in both superscript and subscript represent a summation of components, i.e., a contraction.
In addition, we will use capital letters to denote collections of indices.
Moreover, repeated capital letters signify summations over all represented indices.

Furthermore, when writing in index notation, we let $\mf{A}_{\alpha\beta}[ \cdot ]$ denote an unnormalized antisymmetrization of the indices $\alpha$ and $\beta$ of the indexed quantity within the bracket.
For example, if $S \in \Gamma T^1_1 M$ and $T \in \Gamma T^0_2 M$, then
\begin{equation}\label{eq.index_antisym} \mf{A}_{\beta\gamma}\brak{S^\alpha{}_\beta T_{\gamma\delta}} = S^\alpha{}_\beta T_{\gamma\delta} - S^\alpha{}_\gamma T_{\beta\delta} = 2 S^\alpha{}_{[\beta} T_{\gamma]\delta} \text{.} \end{equation}
Although the standard bracket notation on the right-hand side of \eqref{eq.index_antisym} achieves the same effect, the $\mf{A}$-notation will be useful for larger expressions with multiple antisymmetries, in which the bracket notation would be confusing or ambiguous.

In the case that $M$ is Riemannian or Lorentzian, with metric $g \in \Gamma S^2 M$, then contravariant and covariant components are equivalent, and we will often write $T^{r+s} M$ for $T^r_s M$.
\footnote{For example, the Riemann curvature $R$ of $(M, g)$ is an element of $\Gamma T^4 M$.}
Moreover, given a field $T \in \Gamma T^r M$, we define $R[T] \in \Gamma T^{r + 2} M$ to be the covariant derivative commutator
\begin{equation}\label{eq.curv_gen} R_{\alpha\beta}\brak{T}_I = D_{\alpha\beta} T_I - D_{\beta\alpha} T_I \text{.} \end{equation}
We will also follow these standard traditions within index notation:
\begin{itemize}
\item We denote both the Ricci and scalar curvatures of $(M, g)$ by $R$, but with the appropriate number of indices.

\item If $M$ is oriented, then the volume form is denoted $\epsilon$.
\end{itemize}

Finally, if $(M, g)$ is Riemannian, we can define natural tensor norms by
\begin{equation}\label{eq.tensor_norm} \abs{\cdot}: \Gamma \mc{T} M \rightarrow C^\infty\paren{M} \text{,} \qquad \abs{\Psi} = \paren{\Psi^I \Psi_I}^\frac{1}{2} \text{.} \end{equation}
We can then define the obvious $L^p$-norms, $1 \leq p \leq \infty$, for such tensor fields.

\subsection{Foliations}\label{sec.fol}

Next, we briefly discuss a rather general situation of a $1$-parameter foliation of Riemannian manifolds.
Let $(M, g)$ denote a Lorentzian manifold, and let $N$ denote a smooth submanifold of $M$.
Let $f \in C^\infty(N)$, mapping onto a possibly infinite open interval $\mc{I}$, with $df$ nonvanishing.
For each $x \in \mc{I}$, we assume
\[ \Sigma_x = \brac{p \in M \mid f\paren{p} = x} \]
is a Riemannian submanifold of $M$.
In other words, we can write
\[ N = \bigcup_{x \in \mc{I}} \Sigma_x \]
as a $1$-parameter foliation of Riemannian submanifolds.

In future sections, this situation will arise in the following instances:
\begin{itemize}
\item \emph{Time foliations of spacetime:} In this case, $N = M$, and $f$ is a time function on $M$ which assigns to each a point a time value; see Sec. \ref{sec.tfol}.

\item \emph{Spherical foliations of regular null cones:} In this case, $N$ is a smooth portion of a past null cone in $M$, and $f$ foliates $N$ into spherical cross-sections; see Sec. \ref{sec.nfol}.
In particular, $N$ is null and hence is not pseudo-Riemannian.
\end{itemize}
Since it will be convenient to have some common notations for both cases, we commit to this abstract development here.

We begin by defining the following bundles over $N$:
\begin{itemize}
\item A tensor $w \in (T^r M)_q$ at $q \in \Sigma_x$ is \emph{horizontal} iff $w$ is tangent to $\Sigma_x$.

\item We denote by $\ul{T}^k N$ the \emph{horizontal bundle} over $N$ of all horizontal tensors of total rank $k$ at every $q \in N$.

\item We denote by $\ol{T}^l N$ the \emph{extrinsic bundle} over $N$ of all tensors in $M$ of total rank $l$ at every $q \in N$, i.e., the restriction of $T^l M$ to $N$.

\item The \emph{mixed bundle} $\ol{T}^l \ul{T}^k N$ over $N$ is defined to be the tensor product bundle
\begin{equation}\label{eq.mixed_bundle} \ul{T}^k \ol{T}^l N = \ul{T}^k N \otimes \ol{T}^l N \text{.} \end{equation}
\end{itemize}
For convenience, we adopt the notations
\begin{equation}\label{eq.fol_bundle_sp} \Gamma \ul{\mc{T}} N = \bigcup_{k \geq 0} \Gamma \ul{T}^k N \text{,} \qquad \Gamma \ol{\mc{T}} N = \bigcup_{l \geq 0} \Gamma \ol{T}^l N \text{,} \qquad \Gamma \ul{\mc{T}} \ol{\mc{T}} N = \bigcup_{k, l \geq 0} \Gamma \ul{T}^k \ol{T}^l N \text{,} \end{equation}
i.e., the spaces of horizontal, extrinsic, and mixed tensor fields on $N$.
We also define
\begin{equation}\label{eq.fol_bundle_vf} \ul{\mf{X}}\paren{N} = \Gamma \ul{T}^1 N \text{,} \qquad \ol{\mf{X}}\paren{N} = \Gamma \ol{T}^1 N \text{.} \end{equation}

Moreover, we adopt the following general indexing conventions:
\begin{itemize}
\item Horizontal indices will be denoted using Latin letters.

\item Extrinsic indices will be denoted using Greek letters.

\item Collections of extrinsic indices will be denoted using capital Latin letters.
\end{itemize}

For example, the pullback to $N$ of $g$ is an element of $\Gamma \ol{T}^2 N$, while the induced metrics $\gamma$ on the $\Sigma_x$'s is in $\Gamma \ul{T}^2 N$.
Note that $\gamma$ induces positive-definite bundle metrics $\langle \cdot, \cdot \rangle$ on the horizontal bundles $\ul{T}^k N$ via full metric contraction.
Similarly, the restriction of $g$ to $N$ induces bundle metrics $\langle \cdot, \cdot \rangle$ on the extrinsic bundles $\ol{T}^l N$.
These in turn naturally induce ``product" bundle metrics on $\ul{T}^k \ol{T}^l N$.

Next, let $D$ denote the Levi-Civita connection on $(M, g)$.
Recall from standard theory that $D$ induces a connection $\ol{D}$ on the extrinsic bundles $\ol{T}^l N$.
Moreover, $\ol{D}_X g \equiv 0$ by definition, hence $\ol{D}$ is compatible with the extrinsic bundle metrics.
\footnote{Technically, by $\ol{D}_X g$, we mean $\ol{D}_X$ acting on the restriction of $g$ to $N$.}

We can also naturally define a connection $\nabla$ on the horizontal bundles via projections.
In particular, given $X \in \mf{X}(N)$, then $\nabla$ satisfies the following properties:
\begin{itemize}
\item If $f \in C^\infty(N)$, then $\nabla_X f = Xf$, as usual.

\item If $Y \in \ul{\mf{X}}(N)$, then $\nabla_X Y$ is the projection onto the $\Sigma_x$'s of $\ol{D}_X Y$.

\item If $A \in \Gamma \ul{T}^k N$ is fully covariant, and $Y_1, \ldots, Y_k \in \ul{\mf{X}}(N)$, then
\begin{align}
\label{eq.hor_conn_gen} \nabla_X A\paren{Y_1, \ldots, Y_k} &= X\brak{A\paren{Y_1, \ldots, Y_k}} - A\paren{\nabla_X Y_1, Y_2, \ldots, Y_k} \\
\notag &\qquad - \ldots - A\paren{Y_1, \ldots, Y_{k-1}, \nabla_X Y_k} \text{.}
\end{align}
\end{itemize}
This definition of $\nabla$ generalizes the usual Levi-Civita connections on the $\Sigma_x$'s to also include non-horizontal derivatives tangent to $N$.
Since $\nabla_X \gamma \equiv 0$ for any $X \in \mf{X}(N)$, then $\nabla$ remains compatible with the horizontal bundle metrics.

We can also define mixed connections $\ol{\nabla}$ on the mixed bundles using the above connections $\nabla$ and $\ol{D}$.
On $\Gamma \ul{T}^k \ol{T}^l N$, we define $\ol{\nabla}$ to be the unique connection satisfying the following Leibniz identity for decomposable fields:
\[ \ol{\nabla}_X \paren{A \otimes B} = \nabla_X A \otimes B + A \otimes \ol{D}_X B \text{,} \qquad X \in \mf{X}\paren{N} \text{, } A \in \Gamma \ul{T}^k N \text{, } B \in \Gamma \ol{T}^l N \text{.} \]
In other words, $\ol{\nabla}$ behaves like $\nabla$ and $\ol{D}$ on horizontal and extrinsic components, respectively.
Clearly, $\ol{\nabla}$ is compatible with the above mixed bundle metrics.

\begin{remark}
More explicitly, if $A \in \Gamma \ul{T}^k \ol{T}^l N$; $X \in \mf{X}(N)$; $Y_1, \ldots, Y_k \in \ul{\mf{X}}(N)$; and $Z_1, \ldots, Z_l \in \ol{\mf{X}}(N)$; then $\ol{\nabla}_X A(Y_1, \ldots, Y_k; Z_1, \ldots, Z_l)$ is assigned the value
\begin{align}
\label{eq.mixed_conn_gen} &X\brak{A\paren{Y_1, \ldots, Y_k; Z_1, \ldots, Z_l}} - A\paren{\nabla_X Y_1, Y_2, \ldots, Y_k; Z_1, \ldots Z_l} - \ldots \\
\notag &\qquad - A\paren{Y_1, \ldots, Y_{k-1}, \nabla_X Y_k; Z_1, \ldots, Z_l} - A\paren{Y_1, \ldots, Y_k; \ol{D}_X Z_1, Z_2, \ldots Z_l} \\
\notag &\qquad - \ldots - A\paren{Y_1, \ldots, Y_k; Z_1, \ldots, Z_{l-1}, \ol{D}_X Z_l} \text{.}
\end{align}
\end{remark}

We also set the following notations:
\begin{itemize}
\item For $A \in \Gamma \ul{T}^k N$, define $\nabla A \in \Gamma \ul{T}^{k+1} N$ to map $X \in \ul{\mf{X}}(N)$ to $\nabla_X A$.

\item For $A \in \Gamma \ul{T}^k \ol{T}^l N$, define $\ol{\nabla} A \in \Gamma \ul{T}^{k+1} \ol{T}^l N$ to map $X \in \ul{\mf{X}}(N)$ to $\ol{\nabla}_X A$.

\item Horizontal and mixed Laplacians are defined in the usual fashion:
\begin{equation}\label{eq.lapl} \lapl = \gamma^{ab} \nabla_{ab} \text{,} \qquad \ol{\lapl} = \gamma^{ab} \ol{\nabla}_{ab} \text{.} \end{equation}
\end{itemize}
For further details involving the above constructions, see \cite[Sec. 1.2]{shao:bdc_nv}.

\begin{remark}
In contrast, the wave operator $g^{\alpha\beta} D_{\alpha\beta}$ on $(M, g)$ is denoted $\Box$.
\end{remark}

The fundamental properties are the Leibniz rules satisfied by the connections and the compatibility between the connections and metrics.
These justify integration by parts operations involving $\ol{\nabla}$-derivatives, seen in \cite{kl_rod:ksp, kl_rod:bdc, shao:bdc_nv, shao:ksp} and later in this paper.
This was used implicitly in \cite{kl_rod:ksp, kl_rod:bdc} and was discussed in detail in \cite{shao:bdc_nv, shao:ksp}.

\subsection{Normal Transport}\label{sec.nt}

Assume the same foliation setting as before.
For each point $p \in \Sigma_x \subseteq N$, there is a unique direction in $N$ which is normal to $\Sigma_x$.
Thus, we can define the ``normalized normal" vector field $Z \in \mf{X}(N)$ such that $Z$ points in the direction in $N$ normal to the $\Sigma_x$'s and satisfies $Z f \equiv 1$.

We can naturally define quantities via transport along the integral curves of $Z$.
For instance, given a locally defined function on some $\Sigma_y$, we can define a corresponding local function on another $\Sigma_x$ via the diffeomorphisms induced by the above transport.
In particular, given a coordinate system $(y^1, y^2)$ in $\Sigma_y$, we can define a \emph{transported coordinate system} on the other $\Sigma_x$'s.

Similarly, we can transport horizontal tensors along the integral curves of $Z$.
This induces the standard definition of ``Lie derivatives" of horizontal tensor fields.
More explicitly, for any $A \in \Gamma \ul{\mc{T}} N$, its ``normal Lie derivative" $\mc{L}_f A$ is defined
\begin{equation}\label{eq.lie_deriv_normal} \valat{\mc{L}_f A}_p = \lim_{\delta \rightarrow 0} \frac{f^\ast_{x + \delta, x}\paren{\valat{A}_{f_{x, x + \delta}\paren{p}}} - \valat{A}_p}{\delta} \text{,} \qquad p \in \Sigma_x \subseteq N \text{,} \end{equation}
where $f_{x + \delta, x}$ is the flow from $\Sigma_{x + \delta}$ to $\Sigma_x$ via the integral curves of $Z$, and where $f_{x + \delta, x}^\ast$ denotes the ``push-forward" of tensors through $f_{x + \delta, x}$.

\begin{remark}
The development described here differs from that of \cite{kl_rod:bdc, shao:bdc_nv}, which expressed similar notions entirely in terms of transported coordinate systems.
Notice that the normal Lie derivatives of \eqref{eq.lie_deriv_normal}, which are tensorial and invariant, coincide with the $f$-coordinate derivatives for such transported coordinate systems.
\end{remark}

Finally, we define the analogous ``normalized normal" covariant derivatives: for fields $A \in \Gamma \ul{\mc{T}} N$ and $B \in \Gamma \ul{\mc{T}} \ol{\mc{T}} N$, we define
\begin{equation}\label{eq.cov_deriv_normal} \nabla_f A = \nabla_Z A \text{,} \qquad \ol{\nabla}_f B = \ol{\nabla}_Z B \text{.} \end{equation}

\subsection{The Einstein Equations}

Assume now that $(M, g)$ is a $(1+3)$-dimensional connected, oriented, and time-oriented Lorentzian manifold, with Levi-Civita connection $D$ and Riemann curvature $R$.
In addition, let $\mf{f}$ denote a collection of prescribed ``matter fields" on $M$, satisfying the equations of their respective theories.
The \emph{Einstein equations} on $M$ are given in index notation by
\begin{equation}\label{eq.einstein} G_{\alpha\beta} = R_{\alpha\beta} - \frac{1}{2} \cdot R \cdot g_{\alpha\beta} = Q_{\alpha\beta} \text{,} \end{equation}
where $Q \in \Gamma S^2 M$ denotes the \emph{energy-momentum tensor} of the matter fields.
Recall that the Ricci curvature can be expressed in terms of the matter fields:
\begin{equation}\label{eq.einstein_ric} R_{\alpha\beta} = Q_{\alpha\beta} - \frac{1}{2} g_{\alpha\beta} g^{\mu\nu} Q_{\mu\nu} \text{.} \end{equation}
Moreover, recall that $Q$ is both symmetric and divergence-free.

In this text, we will consider the following settings for matter fields:
\begin{itemize}
\item {\it Vacuum:} There are no matter fields, so both $Q$ and $\ric$ vanish.

\item {\it Scalar field}: In this model case, $\mf{f}$ is given as a \emph{scalar field} $\phi \in C^\infty(M)$, which satisfies the linear homogeneous covariant wave equation
\begin{equation}\label{eq.scalar_field} \square_g \phi = 0 \text{.} \end{equation}
The energy-momentum tensor and Ricci curvature are given by
\begin{equation}\label{eq.es_emt} Q_{\alpha\beta} = D_\alpha \phi D_\beta \phi - \frac{1}{2} g_{\alpha\beta} D^\mu \phi D_\mu \phi \text{,} \qquad R_{\alpha\beta} = D_\alpha \phi D_\beta \phi \text{.} \end{equation}

\item {\it Maxwell field}: In the model electromagnetic case, $\mf{f}$ is given as a \emph{Maxwell field} $F \in \Gamma \Lambda^2 M$, which satisfies the linear homogeneous Maxwell equations 
\begin{equation}\label{eq.maxwell} D^\beta F_{\alpha\beta} \equiv 0 \text{,} \qquad D_\alpha F_{\beta\gamma} + D_\beta F_{\gamma\alpha} + D_\gamma F_{\alpha\beta} \equiv 0 \text{.} \end{equation}
The energy-momentum tensor and Ricci curvature are given by
\begin{equation}\label{eq.em_emt} Q_{\alpha\beta} = R_{\alpha\beta} = F_{\alpha\mu} F^\beta{}^\mu - \frac{1}{4} g_{\alpha\beta} F^{\mu\nu} F_{\mu\nu} \text{.} \end{equation}
\end{itemize}

We remark that in the above settings, $Q$ satisfies both the \emph{positive energy condition} and the \emph{strong energy condition}, i.e., $Q(X, Y) \geq 0$ and $\ric(X, Y) \geq 0$ for all future causal $X, Y \in \mf{X}(M)$.
Moreover, by direct calculations, we can compute
\begin{align}
\label{eq.curv_divg_ev} D^\alpha R_{\alpha\beta\gamma\delta} &\equiv 0 \text{,} \\
\label{eq.curv_divg_es} D^\alpha R_{\alpha\beta\gamma\delta} &= - \mf{A}_{\gamma\delta} \brak{D_\gamma \phi D_{\delta\beta} \phi} \text{,} \\
\label{eq.curv_divg_em} D^\alpha R_{\alpha\beta\gamma\delta} &= - F_\beta{}^\mu D_\mu F_{\gamma\delta} - \mf{A}_{\gamma\delta} \brak{F_\gamma{}^\mu D_\delta F_{\beta\mu} + \frac{1}{2} g_{\beta\gamma} F^{\mu\lambda} D_\delta F_{\mu\lambda}}
\end{align}
in the E-V, E-S, E-M settings, respectively.
\footnote{See \eqref{eq.index_antisym}.}

\section{The CMC Breakdown and Cauchy Problems}

For the breakdown problem at hand, we are given an existing solution of the Einstein equations, i.e., an E-S or E-M spacetime ``existing on a finite time interval".
Our goal will be to ``extend the solution further in time" given that certain criteria hold on this existing solution.
In this section, we will make precise the above informal expressions.
This essentially requires smoothly assigning to each point of the spacetime a ``time value" and then foliating the spacetime into ``timeslices", that is, hypersurfaces of constant time value.

This objective will be achieved by defining global time functions and the various fundamental objects associated with its resulting foliation.
In addition, for the breakdown problem, we will adopt the CMC gauge condition as well.

Throughout the remainder of this paper, we will always let $(M, g)$ denote a $(1+3)$-dimensional connected, oriented, and time-oriented Lorentzian manifold, with Levi-Civita connection $D$ and curvature $R$.

\subsection{Time Foliations}\label{sec.tfol}

We define a \emph{time function} on $M$ to be a map $t \in C^\infty(M)$ satisfying the following conditions:
\begin{itemize}
\item The spacetime gradient of $t$ is everywhere past timelike.

\item Every nonempty level set of $t$ is a Cauchy hypersurface of $M$, i.e., every inextendible causal curve in $M$ intersects each level set of $t$ exactly once. 
\end{itemize}
From the first property, we see $t$ is strictly increasing along all future directions.
Moreover, the second property implies that $(M, g)$ is globally hyperbolic.

For $\tau \in \R$ and an interval $I \subseteq \R$, define
\[ \Sigma_\tau = \brac{z \in M \mid t\paren{z} = \tau} \text{,} \qquad \Sigma_I = \brac{z \in M \mid t\paren{z} \in I} \text{.} \]
The family $\{\Sigma_\tau\}$ of spacelike Cauchy hypersurfaces defines a \emph{time foliation} of $M$,
\begin{equation}\label{eq.tfol} M = \bigcup_{\tau \in \mc{I}} \Sigma_\tau \text{,} \end{equation}
where $\mc{I}$ is an interval in $\R$, and where $t$ maps onto $\mc{I}$.
We will also impose the following assumptions: the interval $\mc{I}$ is finite, and the $\Sigma_\tau$'s are compact.

\begin{remark}
Any two slices $\Sigma_{\tau_1}$, $\Sigma_{\tau_2}$, where $\tau_1, \tau_2 \in \mc{I}$, are in fact diffeomorphic, since any $p_1 \in \Sigma_{\tau_1}$ can be canonically identified with the point $p_2 \in \Sigma_{t_2}$ on the integral curve of $\grad t$ through $p_1$.
\end{remark}

We adopt the notations developed in the Section \ref{sec.fol} to describe the foliation \eqref{eq.tfol}, with $N = M$ and $f = t$.
In addition, we adopt the following indexing conventions:
\begin{itemize}
\item Greek letters refer to all components in $M$, ranging from $0$ to $3$.

\item Latin letters refer to horizontal components, ranging from $1$ to $3$.

\item For implicit index summations, repeating Greek indices are summed from $0$ to $3$, while repeating Latin indices are summed from $1$ to $3$.
\end{itemize}
Also, let $\gamma$, $\nabla$, and $\mc{R}$ denote the horizontal metrics, connections, and curvatures.

Next, we define the following basic objects of interest:
\begin{itemize}
\item Let $n \in C^\infty(M)$ denote the lapse function, given by
\begin{equation}\label{eq.lapse_def} n = \abs{g\paren{\grad t, \grad t}}^{-\frac{1}{2}} > 0 \text{.} \end{equation}

\item Let $T \in \mf{X}(M)$ denote the future unit normal to the $\Sigma_\tau$'s, i.e.,
\begin{equation}\label{eq.T} T = -n \cdot \grad t \text{.} \end{equation}

\item Let $k \in \Gamma \ul{T}^2 M$ be the future second fundamental form of the $\Sigma_\tau$'s:
\begin{equation}\label{eq.sff} k\paren{X, Y} = -g\paren{D_X T, Y} \text{,} \qquad X, Y \in \ul{\mf{X}}\paren{M} \text{.} \end{equation}
\end{itemize}
Recall that $k$ is symmetric and can be decomposed into trace and traceless parts:
\begin{equation}\label{eq.sff_tr} \trace k = \gamma^{ij} k_{ij} \in C^\infty\paren{M} \text{,} \qquad \hat{k} = k - \frac{1}{3} \paren{\trace k} \gamma \in \Gamma \ul{T}^2 M \text{.} \end{equation}
Recall also that $\trace k$ corresponds to the mean curvature of the $\Sigma_\tau$'s in $M$.

We can relate spatial derivatives of $\mc{R}$ and $k$ to spacetime derivatives of $R$.
First, the Gauss equation, expressed in index notation in terms of $\mc{R}$ and $k$, becomes
\begin{equation}\label{eq.gauss_tf} R_{ijlm} = \mc{R}_{ijlm} - k_{jl} k_{im} + k_{il} k_{jm} \text{.} \end{equation}
Similarly, the Codazzi equations and a direct calculation imply
\begin{equation}\label{eq.codazzi_tf} \nabla^j k_{ij} = R_{\alpha i} T^\alpha + \nabla_i \paren{\trace k} \text{,} \qquad \nabla_{[i} k_{j]l} = R_{\alpha l ij} T^\alpha \text{.} \end{equation}

In this setting, the ``normalized normal" vector field $Z$ is given by $Z = n T$.
\footnote{See Section \ref{sec.nt}.}
In addition, the quantities $\gamma$ and $k$ satisfy the evolution equations
\begin{align}
\label{eq.dt_met} \mc{L}_t \gamma &= -2 n k \text{,} \\
\label{eq.dt_sff} \mc{L}_t k_{ij} &= -\nabla_{ij} n + n R_{i \alpha j \beta} T^\alpha T^\beta - n k_{il} k_j{}^l \text{,}
\end{align}
where $\mc{L}_t$ denotes the normal Lie derivative of \eqref{eq.lie_deriv_normal}.

We can naturally define an associated Riemannian metric $h$ on $M$ by
\begin{equation}\label{eq.tf_rmet} h\paren{X, Y} = g\paren{X, Y} + 2 g\paren{T, X} g\paren{T, Y} \text{,} \qquad X, Y \in \mf{X}\paren{M} \text{.} \end{equation}
In other words, we define $h$ to act like $g$ on the horizontal components, and we simply invert the sign of the normal timelike component.
It is easy to see that $h$ is Riemannian, and that $h$ agrees with $\gamma$ for horizontal vector fields.
From now on, tensor norms on $M$ will be defined with respect to $h$.

The following proposition states that the Cauchy-Schwarz inequality continues to hold with respect to contractions by $g$:

\begin{proposition}\label{thm.norm_prod_tf}
Let $\Phi, \Psi \in \Gamma \mc{T} M$, and let $\Phi \cdot \Psi \in \Gamma \mc{T} M$ denote a tensor field obtained by taking zero or more contractions and $g$-contractions of $\Phi \otimes \Psi$.
Then, $|\Phi \cdot \Psi| \lesssim |\Phi| |\Psi|$, where the constant depends on the number of contractions.
\footnote{By $A \lesssim B$, we mean $A \leq CB$ for some positive constant $C$.}
\end{proposition}

\begin{proof}
See \cite[Prop. 4.1]{shao:bdc_nv}.
\end{proof}

Let $\pi = {}^{(T)} \pi \in \Gamma S^2 M$ denote the deformation tensor of $T$:
\begin{equation}\label{eq.dft} \pi\paren{X, Y} = \mc{L}_T g\paren{X, Y} = g\paren{D_X T, Y} + g\paren{D_Y T, X} \text{,} \qquad X, Y \in \mf{X}\paren{M} \text{.} \end{equation}
By direct computation, we can relate $\pi$ to $n$ and $k$: for any $X, Y \in \ul{\mf{X}}(M)$,
\begin{equation}\label{eq.dft_cp} \pi\paren{T, T} \equiv 0 \text{,} \qquad \pi\paren{X, Y} = -2 k\paren{X, Y} \text{,} \qquad \pi\paren{T, X} = X\paren{\log n} \text{.} \end{equation}
In other words, $\pi$ can be thought of as a spacetime object which contains precisely the same information as the horizontal objects $k$ and $\nabla (\log n)$.

One reason for the importance of the deformation tensor is that it bounds a number of essential quantities related to the time foliation, the most important of which are demonstrated by \eqref{eq.dft_cp}.
We also have from \cite[Prop. 4.2]{shao:bdc_nv} the bounds
\begin{equation}\label{eq.dft_bound} \abs{D T} \lesssim \abs{\pi} \text{,} \qquad \abs{D h} \lesssim \abs{\pi} \text{.} \end{equation}
For instance, the following calculus estimate is a direct consequence of \eqref{eq.dft_bound}:

\begin{proposition}\label{thm.norm_power_rule_tf}
The following calculus estimate holds:
\[ \ol{\nabla} \abs{\Psi}^p \lesssim \abs{\Psi}^{p-1} \abs{\ol{\nabla} \Psi} + \abs{\pi} \abs{\Psi}^p \text{,} \qquad p \geq 1 \text{, } \Psi \in \Gamma \ul{\mc{T}} \ol{\mc{T}} \Sigma_\tau \text{.} \]
\end{proposition}

\begin{proof}
Recalling the Leibniz and metric compatibility properties of the mixed covariant differential $\ol{\nabla}$, along with Proposition \ref{thm.norm_prod_tf}, we compute
\begin{align*}
\ol{\nabla} \abs{\Psi}^p &= \frac{p}{2} \brak{h\paren{\Psi, \Psi}}^\frac{p-2}{2} \ol{\nabla} \brak{h\paren{\Psi, \Psi}} \\
&\leq p \abs{\Psi}^{p-2} \cdot \abs{\Psi} \abs{\ol{\nabla} \Psi} + \frac{p}{2} \abs{\Psi}^{p-2} \abs{D h} \abs{\Psi}^2 \text{.}
\end{align*}
The proof is completed by applying \eqref{eq.dft_bound}.
\end{proof}

Lastly, we note the following coarea formula:

\begin{proposition}\label{thm.coarea_tf}
If $\Omega$ is an open subset of $M$ and $\mc{J}$ is a subinterval of $\mc{I}$, then
\[ \int_{\Sigma_{\mc{J}} \cap \Omega} \phi = \int_{\mc{J}} \paren{\int_{\Sigma_\tau \cap \Omega} n \cdot \phi} d\tau \]
for any integrable $\phi \in C^\infty(\Omega)$.
\end{proposition}

\begin{proof}
See \cite[Prop. 4.4]{shao:bdc_nv}.
\end{proof}

\subsection{The Main Theorem}

In terms of describing the Einstein equations as a system of partial differential equations, the above prescription of a time function and foliation leaves us with still another degree of freedom: a gauge condition satisfied by the foliation.
In this text, we will adopt the \emph{constant mean curvature}, or \emph{CMC}, gauge, as was done in \cite{kl_rod:bdc}.

To be much more precise, we say that the spacetime $(M, g)$, along with a time function $t \in C^\infty(M)$ and negative real numbers $t_0 < t_1 < 0$, satisfy the condition \ass{CMC}{t, t_0, t_1} iff the following conditions hold:
\begin{itemize}
\item The total timespan $\mc{I} = t(M)$ is precisely the interval $(t_0, t_1)$.

\item The timeslices $\Sigma_\tau$ are compact.

\item The CMC condition holds, that is, $\trace k = t$ everywhere on $M$.
\end{itemize}

We are now ready to state the precise main theorem of this text:

\begin{theorem}\label{thm.bdc}
Suppose $(M, g, \Phi)$ is an Einstein-scalar or an Einstein-Maxwell spacetime, where $\Phi$ denotes the matter field $\phi \in C^\infty(M)$ or $F \in \Gamma \Lambda^2 M$, corresponding to the E-S and E-M cases, respectively.
In addition, suppose the following hold:
\begin{itemize}
\item The condition \ass{CMC}{t, t_0, t_1} holds for $(M, g)$.

\item The following ``breakdown criterion" holds for some constant $C_0 > 0$,
\begin{equation}\label{eq.bdc} \nabs{k}_{L^\infty\paren{M}} + \nabs{\nabla \paren{\log n}}_{L^\infty\paren{M}} + \nabs{\mf{F}}_{L^\infty\paren{M}} \leq C_0 \text{,} \end{equation}
where $\mf{F}$ denotes either $D \phi$ or $F$ in the E-S and E-M cases, respectively.
\end{itemize}
Then, $(M, g, \Phi)$ can be extended past time $t_1$ as a CMC foliation.
In other words, there is an Einstein-scalar or Einstein-Maxwell (resp.) spacetime $(M_\star, g_\star, \Phi_\star)$ satisfying \ass{CMC}{t_\star, t_0, t_1 + \epsilon}, where $t_\star$ is a time function on $M_\star$, $\epsilon > 0$, and $t_1 + \epsilon < 0$, such that the following statements hold:
\begin{itemize}
\item There exists an isometric imbedding $\mf{i}$ from $(M, g)$ into $(M_\star, g_\star)$.

\item The maps $t$ and $t_\star$ correspond with respect to $\mf{i}$, i.e., for every $t_0 < \tau < t_1$,
\[ \mf{i}\paren{\Sigma_\tau} = \brac{q \in M_\star \mid t_\star\paren{q} = \tau} \text{,} \qquad \mf{i}\paren{M} = \brac{q \in M_\star \mid t_0 < t_\star\paren{q} < t_1} \text{.} \]
In particular, we have $t = t_\star \circ \mf{i}$.

\item The matter fields $\Phi$ and $\Phi_\star$ also correspond with respect to $\mf{i}$, i.e., $\Phi = \mf{i}^\ast \Phi_\star$.
\end{itemize}
\end{theorem}

The goal of the remainder of the paper will be to prove Theorem \ref{thm.bdc}.

\subsection{The Cauchy Problem}

We now give an explicit formulation of the Cauchy problems for the E-S and E-M equations in the CMC gauge.
We will then state a basic local well-posedness result for this problem.

A $4$-tuple $(\Sigma_0, \gamma_0, k_0, \mf{f}_0)$ will be called an ``admissible initial data set" for the Einstein-scalar or Einstein-Maxwell equations iff the following hold:
\begin{itemize}
\item $(\Sigma_0, \gamma_0)$ is a $3$-dimensional compact oriented Riemannian manifold.

\item The field $k_0$ is an element of $\Gamma S^2 \Sigma_0$.
\footnote{$k_0$ is to be interpreted as a ``second fundamental form" for $\Sigma_0$.}

\item The ``mean curvature" $\trace k_0$ has a constant value $\tau_0 < 0$ on all of $\Sigma_0$.

\item The field $\mf{f}_0$ corresponds to initial data for the given matter field.
In the E-S case, this is expressed as a pair $\phi_0, \phi_1 \in C^\infty(\Sigma_0)$, while in the E-M case, this is given as a pair $E_0, H_0 \in \mf{X}^\ast(\Sigma_0)$.

\item Letting $\nabla$ and $\mc{R}_0$ denote the Levi-Civita connection and the curvature for $\Sigma_0$, respectively, then in the E-S case, the following constraints hold,
\begin{equation}\label{eq.constraint_init_es} -\phi_1 \nabla_i \phi_0 = \nabla^j \paren{k_0}_{ij} \text{,} \qquad \abs{\nabla \phi_0}^2 + \phi_1^2 = \mc{R}_0 - \abs{k_0}^2 + \tau_0^2 \text{,} \end{equation}
while in the E-M case, the following constraints hold,
\begin{equation}\label{eq.constraint_init_em} -\epsilon_i{}^{jl} \paren{E_0}_j \paren{H_0}_l = \nabla^j \paren{k_0}_{ij} \text{,} \qquad \abs{E_0}^2 + \abs{H_0}^2 = \mc{R}_0 - \abs{k_0}^2 + \tau_0^2 \text{.} \end{equation}

\item In the E-M case, the ``matter field" $\mf{f}_0$ solves the additional constraints
\begin{equation}\label{eq.constraint_init_maxwell} \nabla^i \paren{E_0}_i \equiv 0 \text{,} \qquad \nabla^i \paren{H_0}_i \equiv 0 \text{.} \end{equation}
\end{itemize}

\begin{remark}
We make the following remarks about the above definition:
\begin{itemize}
\item The symbols $|\cdot|$ in \eqref{eq.constraint_init_es}, \eqref{eq.constraint_init_em} denote the $\gamma_0$-tensor norm.
The volume form $\epsilon_{ijk}$ in \eqref{eq.constraint_init_em} is that of $(\Sigma_0, \gamma_0)$, with respect to a chosen orientation of $\Sigma_0$.

\item Note that \eqref{eq.constraint_init_es} and \eqref{eq.constraint_init_em} correspond to the E-S and E-M (resp.) constraint equations in the CMC gauge.
Similarly, \eqref{eq.constraint_init_maxwell} corresponds to the standard constraints for the Maxwell equations.

\item If $\phi$ and $F$ are the desired ``spacetime fields" for which we wish to solve, then $\phi_0$ and $\phi_1$ correspond to the values $\phi$ and $\mc{L}_t \phi$ on $\Sigma_0$, while $E_0$ and $H_0$ correspond to the electromagnetic decomposition of $F$ on $\Sigma_0$.
\end{itemize}
\end{remark}

Now, if we are given an admissible initial data set $(\Sigma_0, \gamma_0, k_0, \mf{f}_0)$ in either the E-S or the E-M setting, then the goal will be to solve for a triple $(M, g, \Phi)$, along with maps $\mf{i}: \Sigma_0 \rightarrow M$ and $t \in C^\infty(M)$, where the following hold:
\begin{itemize}
\item $(M, g)$ is a $(1 + 3)$-dimensional globally hyperbolic Lorentzian manifold.

\item There exists a time function $t$ on $M$, along with constants $t_0, t_1 \in \R$, where $t_0 < \tau_0 < t_1 < 0$, such that $(M, g)$ satisfies the condition \ass{CMC}{t, t_0, t_1}.

\item The map $\mf{i}$ is an isometric imbedding of $\Sigma_0$ into $M$, and $\mf{i}(\Sigma_0)$ is precisely the level set $\Sigma_{\tau_0}$ of $t$.

\item The element $\Phi$ represents the matter field on $M$: in the E-S setting, then $\Phi = \phi \in C^\infty(M)$, while in the E-M setting, we have $\Phi = F \in \Gamma \Lambda^2 M$.

\item Both the Einstein equations \eqref{eq.einstein} and the appropriate field equations for $\Phi$ (either \eqref{eq.scalar_field} or \eqref{eq.maxwell}) are satisfied on $(M, g)$.

\item The field $k_0$ corresponds to the future second fundamental form of $\Sigma_{\tau_0}$.
In other words, $k_0$ coincides with the pullback $\mf{i}^\ast k$.

\item The field $\mf{f}_0$ corresponds with the restriction of $\Phi$ to $\Sigma_0$.
To be more precise, in the E-S setting, this means
\[ \phi_0 = \mf{i}^\ast \phi \text{,} \qquad \phi_1 = \mf{i}^\ast \mc{L}_t \phi \text{.} \]
In the E-M setting, if $E, H$ is the electromagnetic decomposition of $F$, then
\[ E_0 = \mf{i}^\ast E \text{,} \qquad H_0 = \mf{i}^\ast H \text{.} \]
\end{itemize}
If all the above conditions hold, then $(M, g, \Phi)$, along with $\mf{i}$ and $t$, will be called a solution of the (Einstein-scalar or Einstein-Maxwell) CMC Cauchy problem corresponding to the initial data set $(\Sigma_0, \gamma_0, k_0, \mf{f}_0)$.

Our next task is to state a local well-posedness theorem for both of the above CMC Cauchy problems.
This will be a straightforward modification of the vacuum analogue stated in \cite[Prop. 6.1]{kl_rod:bdc}.

\begin{theorem}\label{thm.lwp_cmc}
Let $(\Sigma_0, \gamma_0, k_0, \mf{f}_0)$ be an admissible initial data set for the E-S or E-M equations, and let $\tau_0 = \trace k_0 < 0$.
Then, there exists a solution $(M, g, \Phi)$ of the E-S/E-M (resp.) CMC Cauchy problem corresponding to the above initial data set, with $(M, g)$ satisfying \ass{CMC}{t, t_0, t_1} for some time function $t$ on $M$ and $t_0 < \tau_0 < t_1 < 0$.
Furthermore, the solution is unique up to isometric imbedding, and the time of existence $t_1 - \tau_0$ depends continuously on the following parameters:
\begin{itemize}
\item The initial mean curvature $\tau_0$.

\item The diameter and injectivity radius of $\Sigma_0$.

\item The following Sobolev norm for $k_0$:
\[ \mf{K}_0 = \nabs{k_0}_{L^4\paren{\Sigma_0}} + \nabs{\nabla k_0}_{L^2\paren{\Sigma_0}} + \nabs{\nabla^2 k_0}_{L^2\paren{\Sigma_0}} + \nabs{\nabla^3 k_0}_{L^2\paren{\Sigma_0}} \text{.} \]

\item The following Sobolev norm for the curvature $\mc{R}_0$ of $\Sigma_0$:
\[ \mf{R}_0 = \nabs{\mc{R}_0}_{L^2\paren{\Sigma_0}} + \nabs{\nabla \mc{R}_0}_{L^2\paren{\Sigma_0}} + \nabs{\nabla^2 \mc{R}_0}_{L^2\paren{\Sigma_0}} \text{.} \]

\item In the E-S case, the following Sobolev norm for $\mf{f}_0 = (\phi_0, \phi_1)$:
\begin{align*}
\mf{F}_0 &= \nabs{\nabla \phi_0}_{L^4\paren{\Sigma_0}} + \nabs{\nabla^2 \phi_0}_{L^2\paren{\Sigma_0}} + \nabs{\nabla^3 \phi_0}_{L^2\paren{\Sigma_0}} + \nabs{\nabla^4 \phi_0}_{L^2\paren{\Sigma}} \\
&\qquad + \nabs{\phi_1}_{L^4\paren{\Sigma_0}} + \nabs{\nabla \phi_1}_{L^2\paren{\Sigma_0}} + \nabs{\nabla^2 \phi_1}_{L^2\paren{\Sigma_0}} + \nabs{\nabla^3 \phi_1}_{L^2\paren{\Sigma_0}} \text{,}
\end{align*}

\item In the E-M case, the following Sobolev norm for $\mf{f}_0 = (E_0, H_0)$:
\begin{align*}
\mf{F}_0 &= \nabs{E_0}_{L^4\paren{\Sigma_0}} + \nabs{\nabla E_0}_{L^2\paren{\Sigma_0}} + \nabs{\nabla^2 E_0}_{L^2\paren{\Sigma_0}} + \nabs{\nabla^3 E_0}_{L^2\paren{\Sigma_0}} \\
&\qquad + \nabs{H_0}_{L^4\paren{\Sigma_0}} + \nabs{\nabla H_0}_{L^2\paren{\Sigma_0}} + \nabs{\nabla^2 H_0}_{L^2\paren{\Sigma_0}} + \nabs{\nabla^3 H_0}_{L^2\paren{\Sigma_0}} \text{.}
\end{align*}
\end{itemize}
\end{theorem}

The ideas behind the proof of Theorem \ref{thm.lwp_cmc} are standard.
The main points of the proof are summarized in further detail in \cite[Sec. 6.2]{shao:bdc_nv}.

\begin{remark}
Both Theorem \ref{thm.lwp_cmc} and \cite[Prop. 6.1]{kl_rod:bdc} are derived by solving the Einstein equations in the CMC gauge along with transported coordinate systems.
This, however, is by no means an optimal result.
For a local well-posedness result requiring less differentiability in the CMC gauge in the vacuum setting, see \cite{and_mo:gr_cmc}.
\end{remark}

\subsection{Outline of the Proof of Theorem \ref{thm.bdc}}\label{sec.outline}

Like the other breakdown results discussed in the introduction, the proof of Theorem \ref{thm.bdc} is at its highest level intimately tied to the corresponding local well-posedness result of Theorem \ref{thm.lwp_cmc}.
Our strategy for proving Theorem \ref{thm.bdc} will be analogous to that of the model breakdown problem for the nonlinear wave equation \eqref{eq.bdc_model_wave}.
Indeed, we aim to bound the parameters on the timeslices in our spacetime which control the time of existence in Theorem \ref{thm.lwp_cmc}.

Assume now the hypotheses of Theorem \ref{thm.bdc}.
For each $t_0 < \tau < t_1$, we define:
\begin{align}
\label{eq.bdc_energy_KR} \mf{K}\paren{\tau} &= \nabs{k}_{L^4\paren{\Sigma_\tau}} + \nabs{\nabla k}_{L^2\paren{\Sigma_\tau}} + \nabs{\nabla^2 k}_{L^2\paren{\Sigma_\tau}} + \nabs{\nabla^3 k}_{L^2\paren{\Sigma_\tau}} \text{,} \\
\notag \mf{R}\paren{\tau} &= \nabs{\mc{R}}_{L^2\paren{\Sigma_\tau}} + \nabs{\nabla \mc{R}}_{L^2\paren{\Sigma_\tau}} + \nabs{\nabla^2 \mc{R}}_{L^2\paren{\Sigma_\tau}} \text{,}
\end{align}
as well as a corresponding matter field energy $\mf{f}(\tau)$.
In the E-S case, we define
\begin{align}
\label{eq.bdc_energy_fes} \mf{f}\paren{\tau} &= \nabs{\nabla \phi}_{L^4\paren{\Sigma_\tau}} + \nabs{\nabla^2 \phi}_{L^2\paren{\Sigma_\tau}} + \nabs{\nabla^3 \phi}_{L^2\paren{\Sigma_\tau}} + \nabs{\nabla^4 \phi}_{L^2\paren{\Sigma_\tau}} \\
\notag &\qquad + \nabs{\mc{L}_t \phi}_{L^4\paren{\Sigma_\tau}} + \nabs{\nabla \paren{\mc{L}_t \phi}}_{L^2\paren{\Sigma_\tau}} \\
\notag &\qquad + \nabs{\nabla^2 \paren{\mc{L}_t \phi}}_{L^2\paren{\Sigma_\tau}} + \nabs{\nabla^3 \paren{\mc{L}_t \phi}}_{L^2\paren{\Sigma_\tau}} \text{,}
\end{align}
while in the E-M case, we define
\begin{align}
\label{eq.bdc_energy_fem} \mf{f}\paren{\tau} &= \nabs{E}_{L^4\paren{\Sigma_\tau}} + \nabs{\nabla E}_{L^2\paren{\Sigma_\tau}} + \nabs{\nabla^2 E}_{L^2\paren{\Sigma_\tau}} + \nabs{\nabla^3 E}_{L^2\paren{\Sigma_\tau}} \\
\notag &\qquad + \nabs{H}_{L^4\paren{\Sigma_\tau}} + \nabs{\nabla H}_{L^2\paren{\Sigma_\tau}} + \nabs{\nabla^2 H}_{L^2\paren{\Sigma_\tau}} + \nabs{\nabla^3 H}_{L^2\paren{\Sigma_\tau}} \text{,}
\end{align}
where $E, H \in \ul{\mf{X}}^\ast(M)$ is the electromagnetic decomposition of $F$.

For convenience, we also fix an ``initial time"
\[ \max\paren{t_0, 2 t_1} < \tau_0 < t_1 \text{,} \]
and we treat $\Sigma_{\tau_0}$ as the ``initial timeslice" in $M$.
From now on, we will only be concerned with the timespan $[\tau_0, t_1)$ in this proof.
For convenience, we also define
\begin{equation}\label{eq.future} M_+ = \Sigma_{\brakparen{\tau_0, t_1}} = \brac{q \in M \mid \tau_0 \leq t\paren{q} < t_1} \text{.} \end{equation}

The mean curvatures of the $\Sigma_\tau$'s, $\tau_0 \leq \tau < t_1$, are trivially comparable to $|t_1|$.
Suppose we can also uniformly control the following:
\begin{itemize}
\item The diameters and the injectivty radii of the $\Sigma_\tau$'s, $\tau_0 \leq \tau < t_1$.

\item The quantities $\mf{K}(\tau)$, $\mf{R}(\tau)$, and $\mf{f}(\tau)$ for all $\tau_0 \leq \tau < t_1$.
\end{itemize}
Then, applying Theorem \ref{thm.lwp_cmc} to each of the timeslices of $M$, we obtain roughly that there exists some sufficiently small $0 < \epsilon < |t_1|$ such that for every $\tau_0 \leq \tau < t_1$:
\begin{itemize}
\item A solution to the Cauchy problem exists with initial data given by $\Sigma_\tau$.

\item The solution exists on a time interval including $[\tau, \tau + \epsilon)$.

\item The solution is unique up to isometric imbedding.
\end{itemize}
By combining the above solutions, we obtain the desired continuation of $(M, g, \Phi)$ to the time interval $(t_0, t_1 + \epsilon)$, which completes the proof of Theorem \ref{thm.bdc}.

As a result, the main objectives will be that of controlling the diameters, the injectivity radii, and the energy quantities $\mf{K}(\tau)$, $\mf{R}(\tau)$, $\mf{f}(\tau)$.
The remainder of this paper will be dedicated to establishing these estimates.
Of these, the energy bounds will be the primary task.
The diameter and injectivity radius bounds, on the other hand, will be trivial consequences by the end of the proof.

Rather than directly controlling the quantities associated with $\mf{K}(\tau)$, $\mf{R}(\tau)$, and $\mf{f}(\tau)$, which are horizontal fields on the timeslices, we instead control norms of corresponding \emph{spacetime} quantities on these timeslices.
These include the spacetime Riemann curvature $R$ and the spacetime matter field, i.e., either $\phi$ or $F$.
The necessary estimates for $\mf{K}(\tau)$, $\mf{R}(\tau)$, and $\mf{f}(\tau)$ can then be derived using standard elliptic estimates.
For example, $R$ can be related to $\mc{R}$ and $k$ using the Gauss equations \eqref{eq.gauss_tf} and the Codazzi equations \eqref{eq.codazzi_tf}.

To control $R$ and the matter field $\Phi$, we apply variations of standard energy-momentum tensor techniques.
In particular, we take advantage of the observations that both $R$ and $\Phi$ satisfy covariant tensorial wave and Maxwell-type equations.
We then construct ``generalized energy-momentum tensors" based on these relations, and we apply these in the standard fanshion to derive basic energy inequalities.

The above suffices for ``lower-order" a priori energy estimates.
In order to derive ``higher-order" energy estimates, however, we will also need uniform bounds for $R$ and for quantities derived from $\Phi$.
\footnote{More specifically, $D^2 \phi$ in the E-S case, and $D F$ in the E-M case.}
Since these quantities satisfy a system of covariant tensorial wave equations, we can then apply the representation formula presented in \cite[Thm. 7]{shao:ksp} (and also in \cite[Thm. 5.1]{shao:bdc_nv}).

The representation formula mentioned above is only valid on the ``regular" portion of past null cones, i.e., prior to the null injectivity radius.
Thus, in order to obtain satisfactory estimates using this formula, we must also control the geometry of past null cones.
In particular, we must control the null injectivity radius and various connection quantities on these cones by other quantities which can be controlled a priori.
In our case, these include $L^2$ ``flux" quantities for $R$ and $\Phi$ on these cones, and other a priori bounds relating to the time foliation.

Unfortunately, this task is exceedingly difficult and is responsible for a vast portion of the technical work behind Theorem \ref{thm.bdc}.
For instance, this involves constructing a geometric tensorial Littlewood-Paley theory, cf. \cite{kl_rod:glp} and \cite[Sec. 2.2]{shao:bdc_nv}, and applying it in a massive bootstrap argument.
Within this argument are various Besov estimates as well as an elaborate sharp trace estimate for regular null cones.
As a result, we omit a majority of this development from this paper.
For details regarding this portion of the argument, see \cite{shao:bdc_nv}.
\footnote{Earlier work in this area for the vacuum case were done in \cite{kl_rod:cg, kl_rod:glp, kl_rod:stt, kl_rod:rin, parl:bdc, wang:cg, wang:cgp}.}

Using the above control on the local null geometry along with the representation formula of \cite{shao:ksp}, we can derive all the necessary estimates for $R$ and $\Phi$.
As a result, we can then estimate $\mf{K}(\tau)$, $\mf{R}(\tau)$, and $\mf{f}(\tau)$, as mentioned before, and hence Theorem \ref{thm.bdc} is proved.
In the remaining sections of the paper, we will provide more detailed discussions on the steps described in this outline.

\section{Regular Past Null Cones}

As mentioned before, both the local energy estimates in this paper and the representation formula for covariant tensor wave equations in \cite{shao:ksp} are essential components of the proof of Theorem \ref{thm.bdc}.
Moreover, both depend heavily on the local null geometry of $(M, g)$.
More specifically, both are applicable to our setting only on ``regular" past null cones, where the null exponential map remains a diffeomorphism and its image retains a smooth structure.

We shall provide in this section some preliminaries on such regular past null cones.
We mainly follow the development given in \cite[Ch. 3]{shao:bdc_nv} and \cite{shao:ksp}, but we restrict ourselves in this paper to the special case of time foliated null cones.
\footnote{This case was discussed in \cite[Ch. 4]{shao:bdc_nv}.}
This eliminates much of the minor technical irritations present in \cite{shao:bdc_nv, shao:ksp} that arose from dealing with more general foliating functions.

\subsection{Regular Past Null Cones}

Assume the spacetime $(M, g, \Phi)$, along with the time foliation of $M$ given by the time function $t$, as expressed in the statement of Theorem \ref{thm.bdc}.
Fix $p \in M$, and consider the \emph{null exponential map} $\ul{\exp}_p$ about $p$, i.e., the exponential map about $p$ restricted to the past null cone $\mf{N}$ of the tangent space $T_p M$.
\footnote{For convenience, we assume $\mf{N}$ does not include the origin of $T_p M$.}
Define the past null cone $N^-(p)$ of $p$ to be the image of $\ul{\exp}_p$.

By definition, $N^-(p)$ is ruled by the past inextendible null geodesics beginning at $p$.
Recall that $\ul{\exp}_p$ is a diffeomorphism between a sufficiently small neighborhood of $0$ in $\mf{N}$ and its image, which is then a smooth null hypersurface of $M$.
We refer to such regions of $N^-(p)$ as ``regular".
A loss of such regularity can occur at a \emph{terminal point} $z$ of $N^-(p)$, where one of the following scenarios hold:
\begin{itemize}
\item The point $z$ is a \emph{cut locus point}, that is, distinct past null geodesics from $p$ intersect at $z$.
In other words, the map $\ul{\exp}_p$ fails to be one-to-one at $z$.

\item The pair $p$ and $z$ are \emph{past null conjugate points}.
In other words, the map $\ul{\exp}_p$ fails to be nonsingular at $z$.
\end{itemize}

Define the function $t_p$ on $N^-(p)$ by $t_p(q) = t(p) - t(q)$, i.e., the difference in time between the vertex $p$ and the given point $q$.
We also define the following:
\begin{itemize}
\item Let $\mf{s}(p)$ be the infimum of all values $t_p(q)$ for which $q \in N^-(p)$ is a past null conjugate point.
We call $\mf{s}(p)$ the \emph{past null conjugacy radius} of $p$ (with respect to the $t$-foliation).
In the case that $N^-(p)$ has no conjugate points, we define $\mf{s}(p) = t(p) - t_0$.
Note that $\mf{s}(p)$ indicates the largest $v > 0$ such that $N^-(p)$ encounters no conjugate points before $t_p$-value $v$.

\item Let $\mf{l}(p)$ be the infimum of all values $t_p(q)$ for which $q \in N^-(p)$ is a cut locus point.
On the other hand, if $N^-(p)$ has no cut locus points, we define $\mf{l}(p) = t(p) - t_0$.
Then, $\mf{l}(p)$ indicates the largest $v > 0$ such that $N^-(p)$ encounters no cut locus points before $t_p$-value $v$.

\item Define $\mf{i}(p) = \min(\mf{s}(p), \mf{l}(p))$.
We call this the \emph{past null injectivity radius} of $p$ (with respect to the $t$-foliation).
\end{itemize}

We will refer to the region
\[ \mc{N}^-\paren{p} = \brac{q \in N^-\paren{p} \mid t_p\paren{q} < \mf{i}\paren{p}} \]
as the \emph{regular past null cone} of $p$.
Then, the null exponential map $\ul{\exp}_p$ is a diffeomorphism between a neighborhood in $\mf{N}$ and $\mc{N}^-(p)$, and $\mc{N}^-(p)$ is indeed a smooth null hypersurface of $M$.
From now on, we will only refer to the regular null cone $\mc{N}^-(p)$, as $N^-(p)$ is in general too irregular for our use.

In addition, for any $0 < v \leq \mf{i}(p)$, we define the null cone segment
\[ \mc{N}^-\paren{p; v} = \brac{q \in \mc{N}^-\paren{p} \mid t_p\paren{q} < v} \text{,} \]

\subsection{Normalization and Foliation}\label{sec.nfol}

Since tangent null vectors in $\mc{N}^-(p)$ have vanishing Lorentzian ``length" and are orthogonal to $\mc{N}^-(p)$, they cannot be normalized without introducing vectors transversal to $\mc{N}^-(p)$.
Consequently, in our treatment, we will require an additional choice of a future timelike unit vector $\mf{t} \in T_pM$ at $p$.
For this, we use the most natural choice based on the problem at hand: the value of the future unit normal $T$ to the $\Sigma_\tau$'s at $p$.

We define the \emph{null generators} of $\mc{N}^-(p)$ (or of $N^-(p)$) to be the inextendible past null geodesics $\gamma$ on $M$ which satisfy $\gamma(0) = p$ and $g(\gamma^\prime(0), T|_p) = 1$.
We can smoothly parametrize these generators by $\Sph^2$ using the following process.
If we choose an orthonormal basis $e_0, \ldots, e_3$ of $T_p M$, with $e_0 = T|_p$, then we can identify each $\omega \in \Sph^2$ with the null generator $\gamma_\omega$ satisfying $\gamma_\omega^\prime(0) = -e_0 + \omega^k e_k$.
For convenience, we assume such a parametrization of the null generators of $\mc{N}^-(p)$, and we denote by $\gamma_\omega$ the null generator corresponding to $\omega$.

\begin{remark}
The objects on $\mc{N}^-(p)$ that we will discuss are of course defined independently of any parametrization of the null generators.
However, for ease of notation, we will work explicitly with $\Sph^2$.
\end{remark}

In addition, we define $L \in \mf{X}(\mc{N}^-(p))$ to be the tangent vector fields of the null generators of $\mc{N}^-(p)$, i.e., we define $L|_{\gamma_\omega(v)} = \gamma_\omega^\prime(v)$ for any $\omega \in \Sph^2$ and $0 < v < \mf{i}(p)$.
We note in particular that $L$ is a geodesic vector field.

As in previous works, e.g., \cite{kl_rod:cg, parl:bdc, shao:bdc_nv, shao:ksp, wang:cg}, we would like to express $\mc{N}^-(p)$ as a foliation of spherical cross-sections.
This was done abstractly in \cite[Sec. 3.1]{shao:bdc_nv} and \cite[Sec. 2]{shao:ksp} for the sake of the generalized representation formula for wave equations.
Here, we can simplify our presentation by considering only the special case of the foliating function $t_p$, the most natural choice for the current setting.

Define the \emph{null lapse} function $\vartheta$ by the formula
\begin{equation}\label{eq.null_lapse} \vartheta = \paren{L t_p}^{-1} \in C^\infty\paren{\mc{N}^-\paren{p}} \text{.} \end{equation}
Note that $\vartheta$ is everywhere strictly positive, and $\vartheta$ satisfies the initial limits
\begin{equation}\label{eq.null_lapse_init} \lim_{v \searrow 0} \valat{\vartheta}_{\gamma_\omega\paren{v}} = n\paren{p} \text{.} \end{equation}
The positivity of $\vartheta$ implies $dt_p$ is nonvanishing, so the level sets of $t_p$, denoted
\[ \mc{S}_v = \brac{q \in \mc{N}^-\paren{p} \mid t_p\paren{q} = v} \text{,} \qquad v > 0 \text{,} \]
form a family of hypersurfaces of $\mc{N}^-(p)$.
Since $L$ represents the unique null direction tangent to $\mc{N}^-(p)$, and the positivity of $\vartheta$ implies $L$ is transverse to each $\mc{S}_v$, we can conclude that each $\mc{S}_v$ is spacelike, i.e., Riemannian.
Furthermore, the definition of $\mf{i}(p)$ implies $\mc{S}_v$ is diffeomorphic to $\Sph^2$ for every $0 < v < \mf{i}(p)$.

We adopt the conventions of Section \ref{sec.fol} to discuss this foliation of $\mc{N}^-(p)$, with $N = \mc{N}^-(p)$ and $f = t_p$.
In order to distinguish this from the time foliation of $M$, we denote the induced horizontal metrics and connections on the $\mc{S}_v$'s by $\lambda$ and $\nasla$, respectively.
Mixed connections on $\mc{N}^-(p)$ are denoted by $\onasla$, and the Gauss curvatures of the $\mc{S}_v$'s are denoted by $\mc{K} \in C^\infty(\mc{N}^-(p))$.
In addition, the ``normalized normal" vector field in this case is given by $Z = \vartheta L$, i.e., $\onasla_{t_p} = \vartheta \onasla_L$.

Since $\mc{N}^-(p)$ is null, we have no volume form on $\mc{N}^-(p)$ with respect to which we can integrate scalar functions.
However, we can still provide a canonical definition for integrals of functions over $\mc{N}^-(p)$.
Indeed, we define this integral by
\begin{equation}\label{eq.pnc_int} \int_{\mc{N}^-(p)} \phi = \int_0^\infty \paren{\int_{\mc{S}_v} \vartheta \cdot \phi} dv \end{equation}
for any $\phi \in C^\infty(\mc{N}^-(p))$ for which the right-hand side is well-defined.
We can similarly define integrals over any open subset of $\mc{N}^-(p)$, in particular for $\mc{N}^-(p; v)$.

\begin{remark}
We can show using the change of variables formula that \eqref{eq.pnc_int} is in fact independent of the foliating function of $\mc{N}^-(p)$; see \cite[Prop. 3.4]{shao:bdc_nv}.
Later, we will further justify \eqref{eq.pnc_int} in terms of general local energy estimates.
\end{remark}

\subsection{Parametrizations and Null Frames}

We can parametrize $\mc{N}^-(p)$ using $t_p$ and a spherical value.
For any $0 < v < \mf{i}(p)$ and $\omega \in \Sph^2$, we can identify the pair $(v, \omega)$ with the unique point $q$ on both the null generator $\gamma_\omega$ and $\mc{S}_v$.
As a result, we can naturally treat any $\phi \in C^\infty(\mc{N}^-(p))$ as a smooth function on the cylinder $(0, \mf{i}(p)) \times \Sph^2$.
For any such $\phi$, we denote by $\phi|_{(v, \omega)}$ the value of $\phi$ at the point $q$ corresponding to the parameters $(v, \omega)$.
We will freely use this $(v, \omega)$-notation throughout future sections without further elaboration.

In general, null frames are local frames $\hat{l}, \hat{m}, e_1, e_2$ which satisfy
\begin{align*}
g\paren{\hat{l}, \hat{l}} = g\paren{\hat{m}, \hat{m}} \equiv 0 \text{, } &\qquad g\paren{\hat{l}, \hat{m}} \equiv -2 \text{,} \\
g\paren{\hat{l}, e_a} = g\paren{\hat{m}, e_a} \equiv 0 \text{, } &\qquad g\paren{e_a, e_b} = \delta_{ab} \text{,}
\end{align*}
Here, we construct null frames which are adapted to our $t_p$-foliation of $\mc{N}^-(p)$.

Each point of $\mc{S}_v$ is normal to exactly two null directions, one of which is represented by $L$.
We define $\ul{L} \in \ol{\mf{X}}(\mc{N}^-(p))$, called the \emph{conjugate null vector field}, to be the vector field in the other normal null direction, subject to the normalization $g(L, \ul{L}) \equiv -2$.
A direct calculation yields the following explicit formula for $\ul{L}$:
\begin{equation}\label{eq.nvc_tfol} \ul{L} = - n^{-2} \vartheta^2 L - 2 n^{-1} \vartheta T \text{.} \end{equation}
Next, we append to $L$ and $\ul{L}$ a local orthonormal frame $e_1, e_2$ on the $\mc{S}_v$'s.
Then, $\{L, \ul{L}, e_1, e_2\}$ defines a natural null frame for $\mc{N}^-(p)$.
We adopt the following indexing conventions for adapted null frames.
\begin{itemize}
\item Horizontal indices $1, 2$ correspond to the directions $e_1$ and $e_2$.

\item $\ul{L}$ corresponds to the index $3$, while $L$ corresponds to the index $4$.
\end{itemize}

Next, if we define the vector field
\[ N = \frac{1}{2} \paren{n^{-1} \vartheta L - n \vartheta^{-1} \ul{L}} \in \ol{\mf{X}}\paren{\mc{N}^-\paren{p}} \text{,} \]
then \eqref{eq.nvc_tfol} implies $g(N, T) \equiv 0$ and $g(N, N) \equiv 1$.
In other words, $N$ is the outer unit normal to each $\mc{S}_v$ in $\Sigma_{t\paren{p} - v}$.
Using \eqref{eq.nvc_tfol} again, we can derive the following:
\begin{equation}\label{eq.nv_tllb} L = - n \vartheta^{-1} \paren{T - N} \text{,} \qquad \ul{L} = - n^{-1} \vartheta \paren{T + N} \text{.} \end{equation}
With this, we can compute the ($h$-)norms of $L$ and $\ul{L}$,
\begin{equation}\label{eq.nv_norm} \abs{L}^2 = 2 n^2 \vartheta^{-2} \text{,} \qquad \abs{\ul{L}}^2 = 2 n^{-2} \vartheta^2 \text{.} \end{equation}

\subsection{Ricci Coefficients}\label{sec.rc}

We will make use of the following connection quantities:
\begin{itemize}
\item Define the null second fundamental forms $\chi, \ul{\chi} \in \Gamma \ul{T}^2 \mc{N}^-(p)$ by
\[ \chi\paren{X, Y} = g\paren{\ol{D}_X L, Y} \text{,} \qquad \ul{\chi}\paren{X, Y} = g\paren{\ol{D}_X \ul{L}, Y} \text{,} \qquad X, Y \in \ul{\mf{X}}\paren{\mc{N}^-\paren{p}} \text{.} \]
Both $\chi$ and $\ul{\chi}$ are symmetric, since both $L$ and $\ul{L}$ are normal to the $\mc{S}_v$'s.
We often decompose $\chi$ into its trace and traceless parts:
\[ \trace \chi = \lambda^{ab} \chi_{ab} \text{,} \qquad \hat{\chi} = \chi - \frac{1}{2} \paren{\trace \chi} \lambda \text{.} \]
We also use an analogous decomposition for $\ul{\chi}$.

\item Define $\zeta, \ul{\eta} \in \Gamma \ul{T}^1 \mc{N}^-(p)$ by
\[ \zeta\paren{X} = \frac{1}{2} g\paren{\ol{D}_X L, \ul{L}} \text{,} \qquad \ul{\eta}\paren{X} = \frac{1}{2} g\paren{X, \ol{D}_L \ul{L}} \text{,} \qquad X \in \ul{\mf{X}}\paren{\mc{N}^-\paren{p}} \text{.} \]
From \cite[Prop. 2.7]{kl_rod:cg}, we have the following relation between $\zeta$ and $\ul{\eta}$:
\begin{equation}\label{eq.null_torsion} \ul{\eta} = -\zeta + \nasla \paren{\log \vartheta} \text{.} \end{equation}
\end{itemize}
The quantities $\trace \chi$, $\hat{\chi}$, and $\zeta$ are called the expansion, shear, and torsion of $\mc{N}^-(p)$.
We refer to the collection of fields $\trace \chi$, $\hat{\chi}$, $\trace \ul{\chi}$, $\hat{\ul{\chi}}$, $\zeta$, $\ul{\eta}$ as the \emph{Ricci coefficients} of $\mc{N}^-(p)$ (with respect to the $t_p$-foliation).

From \eqref{eq.nvc_tfol} and direct calculations, we see that $\ul{\chi}$ and $\ul{\eta}$ are intimately tied to the time foliation of $M$ via the following formulas:
\begin{align}
\label{eq.eta_tfol} \ul{\eta}_a &= k_{ia} N^i + \nasla_a \paren{\log n} \text{,} \\ 
\label{eq.chib_tfol} \ul{\chi}_{ab} &= -n^{-2} \vartheta^2 \cdot \chi_{ab} + 2 n^{-1} \vartheta \cdot k_{ab} \text{.}
\end{align}

Lastly, we define the \emph{mass aspect function} $\mu \in C^\infty(\mc{N}^-(p))$ by
\begin{equation}\label{eq.maf} \mu = \nasla^a \zeta_a - \frac{1}{2} \hat{\chi}^{ab} \hat{\ul{\chi}}_{ab} + \abs{\zeta}^2 + \frac{1}{4} R_{4343} - \frac{1}{2} R_{43} \text{.} \end{equation}
This quantity is present in the representation formula for wave equations, and it plays a crucial role in controlling the local past null geometry.

\subsection{Initial Values}

We now briefly examine the initial values of various quantities on $\mc{N}^-(p)$, that is, we look at the limits of these quantities at $p$ along null generators.
This is of importance in the problem of controlling the local null geometry, since our goal will be to show that these values differ little from the initial value in some specific sense.
Moreover, these initial value computations were applied in the derivations of the representation formulas of \cite{kl_rod:ksp, shao:ksp}.

We opt in this paper to skip the technical details on this topic, since most of them are rather distant from the heart of Theorem \ref{thm.bdc}.
For instance, some of the derivations require applications of convex geometry and bitensor fields.
A complete exposition in the case of general foliating functions can be found in \cite[Sec. 3.3]{shao:bdc_nv}; the case of the $t_p$-foliation is covered in \cite[Sec. 4.2]{shao:bdc_nv}.
An earlier account for the geodesic foliation case, which is needed for the general case, is presented in \cite{wang:cg}.

The first task is to examine the horizontal metrics $\lambda$ at the initial limit.
Note that a coordinate system $(U, \varphi)$ in $\Sph^2$ generates a transported coordinate system on each $\mc{S}_v$ by mapping the point with parameters $(v, \omega)$ to $\varphi(\omega)$.

\begin{proposition}\label{thm.met_pnc_init}
Let $\lambda_0$ denote the Euclidean metric on $\Sph^2$, and fix a coordinate system $(U, \varphi)$ on $\Sph^2$.
Index $\lambda_0$ using $\varphi$-coordinates, and index $\lambda$ on each $\mc{S}_v$ using the associated transported coordinate system.
Then, for any $\omega \in U$,
\[ \lim_{v \searrow 0} v^{-2} \valat{\lambda_{ij}}_{\paren{v, \omega}} = \valat{n^2}_p \cdot \valat{\paren{\lambda_0}_{ij}}_\omega \text{,} \qquad 1 \leq i, j \leq 2 \text{.} \]
\end{proposition}

\begin{proof}
The proof involves relating the ``$(v, \omega)$-parametrization" of $\mc{N}^-(p)$ with normal coordinates; see \cite[Prop. 3.15]{shao:bdc_nv}.
\end{proof}

Proposition \ref{thm.met_pnc_init} implies the following integral limit:

\begin{corollary}\label{thm.int_pnc_init}
Let $\phi \in C^\infty(\mc{N}^-(p))$, let $\phi_0 \in C^\infty(\Sph^2)$, and suppose
\[ \lim_{v \searrow 0} \valat{\phi}_{\paren{v, \omega}} = \phi_0\paren{\omega} \text{,} \qquad \omega \in \Sph^2 \text{.} \]
Then, the following integral limit holds:
\[ \lim_{v \searrow 0} v^{-2} \int_{\mc{S}_v} \phi = \valat{n^2}_p \cdot \int_{\Sph^2} \phi_0 \text{.} \]
\end{corollary}

The next proposition deals with the effects of derivatives on initial limits:

\begin{proposition}\label{thm.pnc_init_D}
Let $A \in \Gamma \ul{\mc{T}} \mc{N}^-(p)$, and suppose
\[ \lim_{v \searrow 0} \valat{\abs{A}}_{\paren{v, \omega}} = 0 \text{,} \qquad \omega \in \Sph^2 \text{.} \]
Then, for any integer $k > 0$, we also have
\[ \lim_{v \searrow 0} v^k \valat{\abs{\nasla^k A}}_{\paren{v, \omega}} = 0 \text{,} \qquad \omega \in \Sph^2 \text{.} \]
\end{proposition}

\begin{proof}
See \cite[Prop. 3.17]{shao:bdc_nv}.
\end{proof}

Finally, we consider the Ricci coefficients:

\begin{proposition}\label{thm.pnc_init}
The following limits hold for each $\omega \in \Sph^2$.
\begin{itemize}
\item We have the following limits for $\chi$:
\begin{equation}\label{eq.pnc_init_chi} \lim_{v \searrow 0} \valat{\abs{\vartheta \paren{\trace \chi} - 2 t_p^{-1}}}_{\paren{v, \omega}} = 0 \text{,} \qquad \lim_{v \searrow 0} \valat{\abs{\hat{\chi}}}_{\paren{v, \omega}} = 0 \text{.} \end{equation}

\item We have the following limits for $\zeta$ and $\ul{\eta}$:
\begin{equation}\label{eq.pnc_init_zeta} \lim_{v \searrow 0} \valat{\abs{\zeta}}_{\paren{v, \omega}} = \lim_{v \searrow 0} \valat{\abs{\ul{\eta}}}_{\paren{v, \omega}} \lesssim \valat{\abs{\pi}}_p \text{.} \end{equation}
\end{itemize}
\end{proposition}

\begin{proof}
Equation \eqref{eq.pnc_init_chi} is an immediate consequence of \cite[Prop. 3.18]{shao:bdc_nv}.
The limit for $\ul{\eta}$ in \eqref{eq.pnc_init_zeta} is an immediate consequence of \eqref{eq.eta_tfol}, while the derivation for $\zeta$ in \eqref{eq.pnc_init_zeta} can be found in \cite[Prop. 4.9]{shao:bdc_nv}.
\end{proof}

\begin{remark}
We can also obtain analogous initial value properties for both $\ul{\chi}$ and $\mu$.
For details, see \cite[Sec. 3.3, Sec. 4.2]{shao:bdc_nv}.
\end{remark}

\section{A Priori Estimates}\label{sec.a_priori}

In this section, we derive a priori estimates which are consequences of the breakdown criterion \eqref{eq.bdc}.
These are fundamental to the proof of Theorem \ref{thm.bdc}, since they determine the nature of the null geometry estimates and the higher-order estimates which we must prove.
In other words, we must control the null geometry and the higher-order energy norms by quantities which can be controlled a priori.

We now introduce the notion of ``fundamental constants", i.e., values on which all of our ``universal" constants will depend.
More explicitly, whenever we write $A \lesssim B$, we mean $A \leq C B$ for some constant $C$ depending only on these fundamental constants.
Similarly, if we write $A \simeq B$, then we mean  $C^{-1} B \leq A \leq C B$ for some constant $C$ depending only on these fundamental constants.

The complete list of such fundamental constants is given below:
\begin{itemize}
\item The ``breakdown time" $t_1$.

\item The ``breakdown criterion" constant $C_0$ in \eqref{eq.bdc}.

\item The intrinsic and extrinsic geometries of the ``initial timeslice" $\Sigma_{\tau_0}$, and the values of the fields defined on $\Sigma_{\tau_0}$.
These include derivatives of $R$, $\Phi$, $k$, $n$, etc., restricted to $\Sigma_{\tau_0}$, as well as the volume $V(\Sigma_{\tau_0})$ of $\tau_0$.
\end{itemize}

\begin{remark}
We can be more explicit on the exact properties of $\Sigma_{\tau_0}$ and values on $\Sigma_{\tau_0}$ for this dependence; see \cite[Sec. 6.3]{shao:bdc_nv}.
However, the main idea is that we can uniformly control a sufficient amount of quantities (as dictated in Section \ref{sec.outline}) on all the $\Sigma_\tau$'s, $\tau_0 \leq \tau < t_1$, by properties of only $\Sigma_{\tau_0}$.
\end{remark}

\subsection{Regularity of the Time Foliation}

Our first task will be to establish some basic control for our time foliation.
From the breakdown criterion \eqref{eq.bdc} along with \eqref{eq.einstein_ric}, \eqref{eq.es_emt}, \eqref{eq.em_emt}, \eqref{eq.dft_cp}, we can immediately obtain the following:
\begin{equation}\label{eq.bdc_lemma} \nabs{\pi}_{L^\infty\paren{M}} \lesssim 1 \text{,} \qquad \nabs{Q}_{L^\infty\paren{M}} + \nabs{\ric}_{L^\infty\paren{M}} \lesssim 1 \text{.} \end{equation}
Here, $\pi$ is the deformation tensor of $T$, while $Q$ and $\ric$ denote the energy-momentum tensor of the matter field and the Ricci curvature of $(M, g)$.

Contracting \eqref{eq.dt_sff} and recalling the CMC gauge condition yields the \emph{lapse equation}
\begin{equation}\label{eq.lapse} \lapl n = n \brak{\abs{k}^2 + \ric\paren{T, T}} - 1 \text{.} \end{equation}
Moreover, the strong energy condition implies that $\ric(T, T) \geq 0$ in both the E-S and E-M settings.
As a result, we can derive the following bounds:

\begin{proposition}\label{thm.cmc_bound}
On each $\Sigma_\tau$, $\tau_0 \leq \tau < t_1$, we have the comparisons
\begin{equation}\label{eq.cmc_bound_ex} \abs{n} \simeq 1 \text{,} \qquad V\paren{\Sigma_\tau} \simeq 1 \text{.} \end{equation}
\end{proposition}

\begin{proof}
At a minimum point $p$ of $n$ on $\Sigma_\tau$, we have from \eqref{eq.lapse} that
\[ \valat{n \brak{\ric\paren{T, T} + \abs{k}^2}}_p \geq 1 \text{.} \]
By \eqref{eq.bdc_lemma}, we obtain a lower bound for $n$ depending only on the fundamental constants.
Next, at a maximum point $p$ of $n$ on $\Sigma_\tau$, we have from \eqref{eq.lapse} that
\[ n\paren{p} \leq \brak{\ric\paren{T, T} + \abs{\hat{k}}^2 + \frac{\tau^2}{3}}^{-1} \leq \frac{3}{\tau^2} \lesssim 1 \text{,} \]
where we also applied the strong energy condition on $\ric(T, T)$.

For a volume form $V_\tau$ on $\Sigma_\tau$, we have
\[ \mc{L}_t V_\tau = -n \paren{\trace k} V_\tau = -n t \cdot V_\tau \geq 0 \text{.} \]
By the above explicit bound $n \leq 3 t^{-2}$, then
\[ 0 \leq \frac{d}{d\tau} V\paren{\Sigma_\tau} = - \tau \int_{\Sigma_\tau} n \leq -3 \tau^{-1} V\paren{\tau} \text{.} \]
From Gr\"onwall's inequality, we obtain $V(\tau) \simeq \tau^{-3} \simeq 1$, as desired.
\end{proof}

Furthermore, from \eqref{eq.bdc} and Proposition \ref{thm.cmc_bound}, we have proved the uniform bounds
\begin{equation}\label{eq.cond_T0} \nabs{n}_{L^\infty\paren{M_+}} + \nabs{n^{-1}}_{L^\infty\paren{M_+}} + \nabs{\nabla n}_{L^\infty\paren{M_+}} + \nabs{k}_{L^\infty\paren{M_+}} \lesssim 1 \text{,} \end{equation}
where $M_+$ is as defined in \eqref{eq.future}.

Next, we examine some coordinate regularity properties satisfied by the timeslices.
To begin with, since $\Sigma_{\tau_0}$ is compact, we can find a constant $C > 1$ and a finite covering of $\Sigma_{\tau_0}$ by local coordinate systems $(U_1, \varphi_1), \ldots, (U_m, \varphi_m)$ such that for each $1 \leq i \leq m$, we have the uniform ellipticity bound
\begin{equation}\label{eq.unif_ell} C^{-1} \abs{\xi}^2 \leq \valat{\gamma_{jl}}_p \xi^j \xi^l \leq C \abs{\xi}^2 \text{,} \qquad p \in U_i \text{, } \xi \in \R^3 \text{,} \end{equation}
where $\gamma$ is indexed with respect to the $\varphi_i$-coordinates.
We next show that this uniform ellipticity condition holds uniformly for all $\Sigma_\tau$'s, $\tau_0 \leq \tau < \tau_1$.

\begin{proposition}\label{thm.unif_ell_prop}
Let $(U_i, \varphi_i)$, $1 \leq i \leq m$, be as given above, and, for each $\tau$, we let $(U^\tau_i, \varphi^\tau_i)$ denote the coordinate system on $\Sigma_\tau$ obtained by transporting $(U_i, \varphi_i)$.
Then, there exists a constant $C_\flat > 1$, depending on the fundamental constants, such that for any $\tau_0 \leq \tau < t_1$ and $1 \leq i \leq m$, we have the estimates
\begin{equation}\label{eq.unif_ell_prop} C_\flat^{-1} \abs{\xi}^2 \leq \valat{\gamma_{jl}}_p \xi^j \xi^l \leq C_\flat \abs{\xi}^2 \text{,} \qquad p \in U^\tau_i \text{, } \xi \in \R^3 \text{,} \end{equation}
where $\gamma$ is indexed with respect to the $\varphi^\tau_i$-coordinates.
In other words, the uniform ellipticity condition \eqref{eq.unif_ell} holds uniformly on all the $\Sigma_\tau$'s.
\end{proposition}

We sketch the proof of Proposition \ref{thm.unif_ell_prop} below.
For further details, we refer the reader to \cite[Prop. 4.6]{shao:bdc_nv}, as well as \cite[Prop. 4.1]{kl_rod:rin} and \cite[Prop. 2.4]{kl_rod:bdc}.

\begin{proof}
Fix a coordinate system $(U, \varphi) = (U_i, \varphi_i)$ of $\Sigma_{\tau_0}$, and consider its transported systems $(U^\tau, \varphi^\tau) = (U^\tau_i, \varphi^\tau_i)$; we will index with respect to these coordinates.
Fix $\xi \in \R^3$, and let $X = \xi^i \partial_i$. 
By \eqref{eq.dt_met}, we have $\mc{L}_t |X|^2 = -2 n \cdot k(X, X)$, so
\begin{equation}\label{eql.unif_ell_prop_1} -2 \nabs{n}_{L^\infty\paren{\Sigma_\tau}} \nabs{k}_{L^\infty\paren{\Sigma_\tau}} \abs{X}^2 \leq \mc{L}_t \abs{X}^2 \leq 2 \nabs{n}_{L^\infty\paren{\Sigma_\tau}} \nabs{k}_{L^\infty\paren{\Sigma_\tau}} \abs{X}^2 \text{.} \end{equation}
Integrating \eqref{eql.unif_ell_prop_1} along the integral curves of $Z = nT$ (or equivalently, of $T$) and applying Gr\"onwall's inequality, we obtain for each $p \in U$ the comparisons
\[ \bar{C}^{-1} \valat{\abs{X}^2}_p \leq \valat{\abs{X}^2}_{p^\tau} \leq \bar{C} \valat{\abs{X}^2}_p \text{,} \qquad \tau \in \mc{I} \text{,} \]
where $p^\tau$ is the normal transport of $p$ along $Z$ to $\Sigma^\tau$, and where
\[ \bar{C} = \exp\brak{2 \paren{t_1 - \tau_0} \nabs{n}_{L^\infty\paren{M_+}} \nabs{\pi}_{L^\infty\paren{M_+}}} \lesssim 1 \text{.} \]
Defining $C_\flat = C \bar{C}$, where $C$ is as in \eqref{eq.unif_ell}, varying over all $\xi \in \R^3$, $1 \leq i \leq m$, and $p \in \Sigma_{\tau_0}$, and recalling the bound \eqref{eq.unif_ell}, then we obtain \eqref{eq.unif_ell_prop}, as desired.
\end{proof}

By considering the family of transported coordinate systems from Proposition \ref{thm.unif_ell_prop}, we can deduce some rudimentary Riemannian geometric estimates on the $\Sigma_\tau$'s.
For a point $x \in \R^3$ and $r > 0$, we let $\mc{B}(x, r)$ denote the Euclidean ball in $\R^3$ about $x$ of radius $r$.
The property given in the subsequent proposition is an immediate consequence of Proposition \ref{thm.unif_ell_prop} and the compactness of the $\Sigma_\tau$'s.

\begin{proposition}\label{thm.min_radius}
Let $(U^\tau_i, \varphi^\tau_i)$ be defined as in Proposition \ref{thm.unif_ell_prop}, for all $\tau_0 \leq \tau < t_1$ and $1 \leq i \leq m$.
There exists $r_0 > 0$, depending on the fundamental constants, such that for any $\tau_0 \leq \tau < t_1$ and $p \in \Sigma_\tau$, there exists $1 \leq i \leq m$ such that $p \in U^\tau_i$, and
\[ \varphi^\tau_i\paren{U^\tau_i} \supseteq \mc{B}\paren{\varphi^\tau_i\paren{p}, r_0} \text{.} \]
\end{proposition}

\begin{proof}
See \cite[Lemma 2.2]{kl_rod:rin} and \cite[Prop. 4.6]{shao:bdc_nv}.
\end{proof}

Next, for any $\tau_0 \leq \tau < t_1$, $p \in \Sigma_\tau$, and $\rho > 0$, we let $B_\tau(p, \rho)$ denote the geodesic ball in $\Sigma_\tau$ of radius $\rho$ about $p$.
In addition, for $\tau$ and $p$ as above, we define
\[ r_\tau\paren{p} = \inf_{\sigma \leq 1} \frac{V\paren{B_\tau\paren{p, \sigma}}}{\sigma^3} \text{,} \]
i.e., the \emph{volume radius} at $p$ (in $\Sigma_\tau$) with scale $1$.
By considering the transported coordinate systems in the statement of Proposition \ref{thm.unif_ell_prop}, we can derive the following uniform lower bound for the above volume radii.

\begin{proposition}\label{thm.vol_radius}
There exists $\rho > 0$, depending on the fundamental constants, such that $r_\tau(p) \geq \rho$ for every $\tau_0 \leq \tau < t_1$ and $p \in \Sigma_\tau$.
\end{proposition}

\begin{proof}
See the proof of \cite[Cor. 2.1]{shao:bdc_nv}.
The main idea is that the uniform ellipticity condition \eqref{eq.unif_ell} implies that geodesic balls are comparable to Euclidean balls.
\end{proof}

\subsection{Sobolev Inequalities}

Proposition \ref{thm.unif_ell_prop} also plays an instrumental role in deriving uniform Sobolev estimates on the $\Sigma_\tau$'s.
Indeed, we can derive such first-order Sobolev inequalities on each $\Sigma_\tau$ by using a partition of unity argument and applying the corresponding Euclidean Sobolev inequality to each coordinate system $(U^\tau_i, \varphi^\tau_i)$.
As a result of this, we obtain the following scalar Sobolev inequalities:

\begin{lemma}\label{thm.sob_scalar}
The following hold for any $f \in C^\infty(M)$, $\tau_0 \leq \tau < t_1$, and $q > 3$:
\begin{align}
\label{eq.sob_scalar} \nabs{f}_{L^\frac{3}{2}\paren{\Sigma_\tau}} \lesssim \nabs{\nabla f}_{L^1\paren{\Sigma_\tau}} + \nabs{f}_{L^1\paren{\Sigma_\tau}} \text{,} \\
\label{eq.sobs_scalar} \nabs{f}_{L^\infty\paren{\Sigma_\tau}} \lesssim \nabs{\nabla f}_{L^q\paren{\Sigma_\tau}} + \nabs{f}_{L^q\paren{\Sigma_\tau}} \text{.}
\end{align}
In \eqref{eq.sobs_scalar}, the constant of the inequality depends also on $q$.
\end{lemma}

\begin{proof}
See \cite[Lemma 2.1]{shao:bdc_nv}.
\footnote{Also essential to the proof of \eqref{eq.sobs_scalar} is the ``minimal radius" property of Proposition \ref{thm.min_radius}.}
\end{proof}

\begin{remark}
The significance of Lemma \ref{thm.sob_scalar} is that the constants of \eqref{eq.sob_scalar} and \eqref{eq.sobs_scalar}, i.e., the ``Sobolev constants", are uniformly bounded over all the $\Sigma_\tau$'s, $\tau_0 \leq \tau < t_1$.
In particular, they are controlled by the fundamental constants.
\end{remark}

In addition, we can apply Lemma \ref{thm.sob_scalar} along with the preceding remark in order to derive tensorial versions of Sobolev inequalities.

\begin{proposition}\label{thm.sobolev_tf}
Let $\tau_0 \leq \tau < t_1$.
\begin{itemize}
\item For any $2 \leq p \leq 6$, $q > 3$, and $\Psi \in \Gamma \ul{\mc{T}} \ol{\mc{T}} M$,
\begin{align}
\label{eq.sob} \nabs{\Psi}_{L^p\paren{\Sigma_\tau}} &\lesssim \nabs{\ol{\nabla} \Psi}_{L^2\paren{\Sigma_\tau}} + \nabs{\Psi}_{L^2\paren{\Sigma_\tau}} \text{,} \\
\label{eq.sobs} \nabs{\Psi}_{L^\infty\paren{\Sigma_\tau}} &\lesssim \nabs{\ol{\nabla} \Psi}_{L^q\paren{\Sigma_\tau}} + \nabs{\Psi}_{L^q\paren{\Sigma_\tau}} \text{.}
\end{align}
The constants of the inequalities depend also on $p$ and $q$, respectively.

\item In addition, for any $\Psi \in \Gamma \ul{\mc{T}} \ol{\mc{T}} M$ and $\Phi \in \Gamma \mc{T} M$,
\begin{align}
\label{eq.sob_2h} \nabs{\Psi}_{L^\infty\paren{\Sigma_\tau}} &\lesssim \nabs{\ol{\nabla}^2 \Psi}_{L^2\paren{\Sigma_\tau}} + \nabs{\Psi}_{L^2\paren{\Sigma_\tau}} \text{,} \\
\label{eq.sob_2e} \nabs{\Phi}_{L^\infty\paren{\Sigma_\tau}} &\lesssim \nabs{D^2 \Phi}_{L^2\paren{\Sigma_\tau}} + \nabs{D \Phi}_{L^2\paren{\Sigma_\tau}} + \nabs{\Phi}_{L^2\paren{\Sigma_\tau}} \text{.}
\end{align}
\end{itemize}
\end{proposition}

\begin{proof}
The inequalities \eqref{eq.sob} and \eqref{eq.sobs} are direct adaptations of the proof of the first part of \cite[Prop. 4.7]{shao:bdc_nv}; see also \cite[Prop. 2.3]{shao:bdc_nv}, \cite[Prop. 2.4]{shao:bdc_nv}, and \cite[Cor. 2.7]{kl_rod:bdc}.
The main idea is to apply Lemma \ref{thm.sob_scalar} to the scalar quantities $|\Psi|^q$, $q > 1$.
\footnote{We norm mixed tensor fields with respect to $h$ and $\gamma$.}

Next, \eqref{eq.sob_2h} is proved by applying \eqref{eq.sobs} and \eqref{eq.sob} in succession.
We can also obtain \eqref{eq.sob_2e} by applying \eqref{eq.sobs} and \eqref{eq.sob} in the following manner:
\begin{align*}
\nabs{\Psi}_{L^\infty\paren{\Sigma_\tau}} &\lesssim \nabs{D \Psi}_{L^4\paren{\Sigma_\tau}} + \nabs{\Psi}_{L^4\paren{\Sigma_\tau}} \\
&\lesssim \nabs{D^2 \Psi}_{L^2\paren{\Sigma_\tau}} + \nabs{D \Psi}_{L^2\paren{\Sigma_\tau}} + \nabs{\Psi}_{L^2\paren{\Sigma_\tau}} \text{.}
\end{align*}
Note that for $\Phi \in \Gamma \mc{T} M = \Gamma \ol{\mc{T}} M$, we clearly have $|\ol{\nabla} \Phi| \lesssim |D \Phi|$.
\end{proof}

\subsection{Generalized Energy-Momentum Tensors}\label{sec.emt}

We now describe some general methods for obtaining global and local energy bounds for various spacetime quantities.
Our focus will be on two general classes of ``energy-momentum" tensor fields, abbreviated as \emph{EMTs}, based on covariant wave-type and Maxwell-type equations.

The ``wave equation" EMT was previously applied in \cite{chru_sh:ym_curv, kl_rod:bdc} in order to obtain higher-order estimates.
In particular, in \cite{kl_rod:bdc}, this was used to bound the $L^2$-norms of covariant derivatives of the curvature in a vacuum spacetime.
We will apply these tensor fields with analogous intentions.
The ``Maxwell" EMT is an adaptation of this process to solutions of ``Maxwell-type" equations.
We will use this tensor field to derive a priori energy estimates for the curvature $R$, thereby avoiding the more precise but more computationally complex Bel-Robinson tensor field.

Throughout this section, we will assume the following fields:
\begin{itemize}
\item Let $U, V \in \Gamma T^r M$ satisfy the covariant tensor wave equation
\begin{equation}\label{eq.cov_wave_eq} \Box_g U = V \text{,} \end{equation}
where $\Box_g$ is the tensorial wave operator $\Box_g = g^{\alpha\beta} D_{\alpha\beta}$.

\item Let $A \in \Gamma T^{r + 2} M$, $B \in \Gamma T^{r + 1} M$, and $C \in \Gamma T^{r + 3} M$, with $A$ antisymmetric in its first two components and satisfing the Maxwell-type equations
\begin{equation}\label{eq.gen_maxwell_eq} D^\alpha A_{\alpha\beta I} = B_{\beta I} \text{,} \qquad D_\gamma A_{\alpha\beta I} + D_\alpha A_{\beta\gamma I} + D_\beta A_{\gamma\alpha I} = C_{\gamma\alpha\beta I} \text{.} \end{equation}
\end{itemize}

We begin by describing the ``wave equation" case as in \cite{kl_rod:bdc}.
Define the EMT $Q_w[U] \in \Gamma S^2 M$ and its associated current $P_w[U] \in \mf{X}^\ast(M)$ by
\begin{align*}
Q_w\brak{U}_{\alpha\beta} &= h\paren{D_\alpha U, D_\beta U} - \frac{1}{2} g_{\alpha\beta} g^{\mu\nu} \cdot h\paren{D_\mu U, D_\nu U} \text{,} \\
P_w\brak{U}\paren{X} &= Q_w\brak{U}\paren{X, T} \text{,}
\end{align*}
where $h(\cdot, \cdot)$ denotes a full $h$-metric contraction of the unindexed components.
Note the resemblance to the energy-momentum tensor \eqref{eq.es_emt} for the scalar field.
In addition, we define the corresponding divergence quantities $D \cdot Q_w[U] \in \mf{X}^\ast(M)$ and $D \cdot P_w[U] \in C^\infty(M)$ in the expected manner:
\[ D \cdot Q_w\brak{U}_\alpha = D^\beta Q_w\brak{U}_{\alpha\beta} \text{,} \qquad D \cdot P_w\brak{U} = D^\alpha P_w\brak{U}_\alpha \text{.} \]
The main properties of this wave EMT are listed below.

\begin{proposition}\label{thm.emt_wave}
The EMT $Q_w[U]$ satisfies the following properties:
\begin{itemize}
\item The following estimates hold:
\begin{align}
\label{eq.emt_wave_norm} \abs{Q_w\brak{U}} &\lesssim \abs{D U}^2 \text{,} \\
\label{eq.emt_wave_divg} \abs{D \cdot Q_w\brak{U}} &\lesssim \abs{D U} \paren{\abs{V} + \abs{R\brak{U}}} + \abs{\pi} \abs{D U}^2 \text{,} \\
\notag \abs{D \cdot P_w\brak{U}} &\lesssim \abs{D U} \paren{\abs{V} + \abs{R\brak{U}}} + \abs{\pi} \abs{D U}^2 \text{.}
\end{align}

\item For future causal $X, Y \in \mf{X}(M)$,
\begin{equation}\label{eq.emt_wave_pos} Q_w\brak{U}\paren{X, Y} \geq 0 \text{,} \end{equation}
In particular,
\begin{equation}\label{eq.emt_wave_tt} Q_w\brak{U}\paren{T, T} = \frac{1}{2} \abs{D U}^2 \text{.} \end{equation}
\end{itemize}
\end{proposition}

\begin{proof}
These follow immediately from a series of direct computations.
For further details, consult \cite[Prop. 4.13]{shao:bdc_nv} and \cite{kl_rod:bdc}.
\end{proof}

The ``Maxwell-type" EMT is constructed analogously.
Indeed, we define the EMT $Q_m[A] \in \Gamma S^2 M$ and its associated current $P_m[A] \in \mf{X}^\ast(M)$ by
\begin{align*}
Q_m\brak{A}_{\alpha\beta} &= h\paren{A_{\alpha\mu}, A_\beta{}^\mu} - \frac{1}{4} g_{\alpha\beta} h\paren{A_{\mu\nu}, A^{\mu\nu}} \text{,} \\
P_m\brak{A}\paren{X} &= Q_m\brak{A}\paren{X, T} \text{,}
\end{align*}
and define $D \cdot Q_m[A] \in \mf{X}^\ast(M)$, $D \cdot P_m[A] \in C^\infty(M)$ by
\label{pg.emt_maxwell_divg}
\[ D \cdot Q_m\brak{A}_\alpha = D^\beta Q_m\brak{A}_{\alpha\beta} \text{,} \qquad D \cdot P_m\brak{A} = D^\alpha P_m\brak{A}_\alpha \text{.} \]

\begin{proposition}\label{thm.emt_maxwell}
The EMT $Q_m[A]$ satisfies the following properties:
\begin{itemize}
\item The following estimates hold:
\begin{align}
\label{eq.emt_maxwell_norm} \abs{Q_m\brak{A}} &\lesssim \abs{A}^2 \text{,} \\
\label{eq.emt_maxwell_divg} \abs{D \cdot Q_m\brak{A}} &\lesssim \abs{A} \paren{\abs{B} + \abs{C}} + \abs{\pi} \abs{A}^2 \text{,} \\
\notag \abs{D \cdot P_m\brak{A}} &\lesssim \abs{A} \paren{\abs{B} + \abs{C}} + \abs{\pi} \abs{A}^2 \text{.}
\end{align}

\item For future causal $X, Y \in \mf{X}(M)$,
\begin{equation}\label{eq.emt_maxwell_pos} Q_m\brak{A}\paren{X, Y} \geq 0 \text{.} \end{equation}
In particular,
\begin{equation}\label{eq.emt_maxwell_tt} Q_m\brak{A}\paren{T, T} = \frac{1}{4} \abs{A}^2 \text{.} \end{equation}
\end{itemize}
\end{proposition}

\begin{proof}
The proofs are analogous to those of Proposition \ref{thm.emt_wave}; see \cite[Prop. 4.14]{shao:bdc_nv}.
\end{proof}

\subsection{Global Energy Estimates}

Next, we apply these generalized EMTs in order to establish both global and local energy estimates.
By ``global energy estimates", we mean $L^2$-estimates on entire timeslices, while by ``local energy estimates", we refer to $L^2$-estimates on past null cones, i.e., ``flux estimates".

We begin with general global estimates.
Recall that Killing vector fields can be associated via Noether's theorem with conservation laws for quantities derived from (standard) energy-momentum tensors.
Since we have no Killing fields in our setting, we have no hope for achieving such energy conservation.
We can, however, think of $T$ as an ``almost Killing" vector field, since its deformation tensor $\pi$ is uniformly bounded.
In fact, by slightly modifying the process behind Noether's theorem, we can derive energy \emph{inequalities} by taking advantage of this ``almost Killing" condition for $T$ along with the generalized EMTs $Q_w[U]$ and $Q_m[A]$.

\begin{lemma}\label{thm.emt_est_pre}
Assume $U, V, A, B, C$ as before.
If $\tau_0 \leq \tau_1 < \tau_2 < t_1$, then
\begin{align}
\label{eq.emt_wave_est_pre} \nabs{D U}_{L^2\paren{\Sigma_{\tau_2}}}^2 &\lesssim \nabs{D U}_{L^2\paren{\Sigma_{\tau_1}}}^2 + \int_{\tau_1}^{\tau_2} \int_{\Sigma_\tau} \paren{\abs{D U}^2 + \abs{V}^2 + \abs{R\brak{U}}^2} d\tau \text{,} \\
\label{eq.emt_maxwell_est_pre} \nabs{A}_{L^2\paren{\Sigma_{\tau_2}}}^2 &\lesssim \nabs{A}_{L^2\paren{\Sigma_{\tau_1}}}^2 + \int_{\tau_1}^{\tau_2} \int_{\Sigma_\tau} \paren{\abs{A}^2 + \abs{B}^2 + \abs{C}^2} d\tau \text{.}
\end{align}
\end{lemma}

\begin{proof}
First, for \eqref{eq.emt_wave_est_pre}, we integrate $D \cdot P_w[U]$ over the slab $\Sigma_{[\tau_1, \tau_2]}$.
\footnote{Recall that $\Sigma_{[\tau_1, \tau_2]} = \{p \in M \mid \tau_1 \leq t(p) \leq \tau_2\}$.}
Applying the divergence theorem along with Proposition \ref{thm.coarea_tf}, we obtain the identity
\[ \int_{\Sigma_{\tau_1}} Q_w\brak{U}\paren{T, T} - \int_{\Sigma_{\tau_2}} Q_w\brak{U}\paren{T, T} = \int_{\tau_1}^{\tau_2} \brak{\int_{\Sigma_\tau} n \cdot \paren{D \cdot P_w\brak{U}}} d\tau \text{.} \]
By \eqref{eq.cond_T0}, \eqref{eq.emt_wave_divg}, \eqref{eq.emt_wave_tt}, then
\begin{align*}
\nabs{D U}_{L^2\paren{\Sigma_{\tau_2}}}^2 &\lesssim \nabs{D U}_{L^2\paren{\Sigma_{\tau_1}}}^2 + \int_{\tau_1}^{\tau_2} \int_{\Sigma_\tau} \brak{\abs{D U} \paren{\abs{V} + \abs{R\brak{U}}} + \abs{D U}^2} d\tau \\
&\lesssim \nabs{D U}_{L^2\paren{\Sigma_{\tau_1}}}^2 + \int_{\tau_1}^{\tau_2} \int_{\Sigma_\tau} \paren{\abs{D U}^2 + \abs{V}^2 + \abs{R\brak{U}}^2} d\tau \text{.}
\end{align*}
For \eqref{eq.emt_maxwell_est_pre}, we repeat the above steps with $D \cdot P_m[A]$ in the place of $D \cdot P_w[U]$.
\end{proof}

To obtain the specific a priori estimates we desire, we will apply Lemma \ref{thm.emt_est_pre} to the spacetime curvature $R$ and matter field: $D^2 \phi$ in the E-S case, or $D F$ in the E-M case.
For this, we will need to take advantage of the following facts:
\begin{itemize}
\item The curvature $R$ satisfies Maxwell-type equations.
This includes
\begin{equation}\label{eq.bianchi} D_{[\mu} R_{\alpha\beta]\gamma\delta} \equiv 0 \text{,} \end{equation}
i.e., the Bianchi idenitities, along with the divergence relations \eqref{eq.curv_divg_es} and \eqref{eq.curv_divg_em}, which hold in the E-S and E-M settings, respectively.

\item In the E-S case, $D \phi$ satisfies the following wave equation:
\begin{equation}\label{eq.es_wave0} \square_g D_\alpha \phi = R \cdot D_\alpha \phi \text{.} \end{equation}
This can be derived by differentiating \eqref{eq.scalar_field} and commuting derivatives.

\item In the E-M case, $F$ satisfies the following wave equation:
\begin{equation}\label{eq.em_wave0} \square_g F_{\alpha\beta} = 2 R_\alpha{}^{\lambda\mu}{}_\beta F_{\lambda\mu} + R^\lambda{}_\alpha F_{\beta\lambda} - R^\lambda{}_\beta F_{\alpha\lambda} \text{.} \end{equation}
This is derived by differentiating the Bianchi identity for $F$ in \eqref{eq.maxwell}.
\end{itemize}

For every $\tau_0 \leq \tau < t_1$, we define the global spacetime energy quantities
\[ \mc{E}^1\paren{\tau} = \nabs{R}_{L^2\paren{\Sigma_\tau}} + \nabs{D^2 \phi}_{L^2\paren{\Sigma_\tau}} \text{,} \qquad \mc{E}^1\paren{\tau} = \nabs{R}_{L^2\paren{\Sigma_\tau}} + \nabs{D F}_{L^2\paren{\Sigma_\tau}} \]
in the E-S and E-M settings, respectively.
\footnote{In later sections, we will also define and estimate the higher-order analogues $\mc{E}^2(\tau)$ and $\mc{E}^3(\tau)$.  We adopt the abbreviation $\mc{E}^1(\tau)$ here in order to maintain consistency of notation.}

\begin{proposition}\label{thm.e1_est}
In both the E-S and E-M cases, we have the energy inequality
\[ \mc{E}^1\paren{\tau} \lesssim 1 \text{,} \qquad \tau_0 \leq \tau < t_1 \text{.} \]
\end{proposition}

\begin{proof}
We begin with the E-S case.
To handle the curvature, we apply \eqref{eq.emt_maxwell_est_pre}, with $A = R$, and with $B$ and $C$ given by \eqref{eq.curv_divg_es} and \eqref{eq.bianchi}, in order to obtain
\[ \nabs{R}_{L^2\paren{\Sigma_{\tau}}}^2 \lesssim \nabs{R}_{L^2\paren{\Sigma_{t_0}}}^2 + \int_{t_0}^\tau \int_{\Sigma_{\tau^\prime}} \paren{\abs{R}^2 + \abs{D \phi}^2 \abs{D^2 \phi}^2} d\tau^\prime \text{.} \]
By \eqref{eq.emt_wave_est_pre}, with $U = D \phi$, along with \eqref{eq.es_wave0}, we obtain the bound
\[ \nabs{D^2 \phi}_{L^2\paren{\Sigma_{\tau}}}^2 \lesssim \nabs{D^2 \phi}_{L^2\paren{\Sigma_{t_0}}}^2 + \int_{t_0}^\tau \int_{\Sigma_{\tau^\prime}} \paren{\abs{D^2 \phi}^2 + \abs{D \phi}^2 \abs{R}^2} d\tau^\prime \text{.} \]
Summing the above equations, recalling the uniform bound \eqref{eq.bdc} for $D \phi$, and applying Gr\"onwall's inequality to the result yields the desired bound.

The E-M case is derived in a completely analogous manner.
We apply \eqref{eq.emt_wave_est_pre} and \eqref{eq.emt_maxwell_est_pre} as before, along with \eqref{eq.curv_divg_em}, \eqref{eq.bianchi}, \eqref{eq.em_wave0}, and we obtain the bounds
\begin{align*}
\nabs{R}_{L^2\paren{\Sigma_{\tau}}}^2 \lesssim \nabs{R}_{L^2\paren{\Sigma_{t_0}}}^2 + \int_{t_0}^\tau \int_{\Sigma_{\tau^\prime}} \paren{\abs{R}^2 + \abs{F}^2 \abs{D F}^2} d\tau^\prime \text{,} \\
\nabs{D F}_{L^2\paren{\Sigma_{\tau}}}^2 \lesssim \nabs{D F}_{L^2\paren{\Sigma_{t_0}}}^2 + \int_{t_0}^\tau \int_{\Sigma_{\tau^\prime}} \paren{\abs{D F}^2 + \abs{F}^2 \abs{R}^2} d\tau^\prime \text{.}
\end{align*}
Summing and applying Gr\"onwall's inequality yields the desired inequality.
\end{proof}

\subsection{Local Energy Estimates}\label{sec.lest}

Local energy estimates are derived in a manner similar to the global estimates of Lemma \ref{thm.emt_est_pre}.
The main difference is that we integrate over the interior of a past null cone rather than over a time slab.

Fix $p \in M_+$; we normalize and foliate $\mc{N}^-(p)$ as before using $T|_p$ and $t_p$.
Fix also a time value $\tau_0 \leq \tau < t(p)$ such that $t(p) - \tau < \min(\mf{i}(p), 1)$.
\footnote{The stipulation $t(p) - \tau < 1$ is purely a matter of convenience, in order to eliminate the dependence of various inequalities on the timespan $t(p) - \tau$.}
We define
\[ \mc{N}_\tau = \mc{N}^-\paren{p; t\paren{p} - \tau} \text{,} \qquad \mc{I}_{\tau} = \brac{q \in I^-\paren{p} \mid t\paren{q} > \tau} \text{,} \qquad \Sigma^p_\tau = I^-\paren{p} \cap \Sigma_\tau \text{,} \]
where $I^-(p)$ is the chronological past of $p$, i.e., the set of points $q \in M$ reachable from $p$ via past timelike paths.
Note that $\mc{I}_\tau$ is the region bounded by $\mc{N}_\tau$ and $\Sigma_\tau$, while $\Sigma^p_\tau$ is the part of $\Sigma_\tau$ in the interior of $\mc{N}^-(p)$.

The following lemma provides the main technical calculation for deriving local energy inequalities and justifies our previous definition of integration on $\mc{N}^-(p)$.

\begin{lemma}\label{thm.int_parts_loc_energy}
If $\omega \in \mf{X}^\ast(M)$, then
\[ \int_{\mc{I}_\tau} D^\alpha \omega_\alpha = \int_{\Sigma^p_\tau} \omega\paren{T} + \int_{\mc{N}_\tau} \omega\paren{L} \text{.} \]
\end{lemma}

\begin{proof}
See \cite[Lemma 4.1]{shao:bdc_nv} or \cite{kl_rod:rin}.
\end{proof}

We are now prepared to derive the general local estimate.

\begin{lemma}\label{thm.emt_lest_pre}
Assume $U, V, A, B, C$ as before.
Then,
\begin{align}
\label{eq.emt_wave_lest_pre} -\int_{\mc{N}_\tau} Q_w\brak{U}\paren{T, L} &\lesssim \nabs{D U}_{L^2\paren{\Sigma^p_\tau}}^2 \\
\notag &\qquad + \int_\tau^{t\paren{p}} \int_{\Sigma^p_\sigma} \paren{\abs{D U}^2 + \abs{V}^2 + \abs{R\brak{U}}^2} d\sigma \text{,} \\
\label{eq.emt_maxwell_lest_pre} -\int_{\mc{N}_\tau} Q_m\brak{A}\paren{T, L} &\lesssim \nabs{A}_{L^2\paren{\Sigma^p_\tau}}^2 \\
\notag &\qquad + \int_\tau^{t\paren{p}} \int_{\Sigma^p_\sigma} \paren{\abs{A}^2 + \abs{B}^2 + \abs{C}^2} d\sigma \text{.}
\end{align}
\end{lemma}

\begin{proof}
Integrate $D \cdot P_w[U]$ over $\mc{I}_\tau$ and apply Lemma \ref{thm.int_parts_loc_energy} and Proposition \ref{thm.coarea_tf} to obtain
\[ \int_{\Sigma^p_\tau} Q_w\brak{U}\paren{T, T} + \int_{\mc{N}_\tau} Q_w\brak{U}\paren{T, L} = \int_\tau^{t\paren{p}} \paren{\int_{\Sigma^p_\sigma} n \cdot D \cdot P_w\brak{U}} d\sigma \text{.} \]
By \eqref{eq.cond_T0}, \eqref{eq.emt_wave_divg}, and \eqref{eq.emt_wave_tt}, we obtain \eqref{eq.emt_wave_lest_pre}.
Repeating the above steps, but with $P_m[A]$ in place of $P_w[U]$ and \eqref{eq.emt_maxwell_divg}, \eqref{eq.emt_maxwell_tt} in place of \eqref{eq.emt_wave_divg}, \eqref{eq.emt_wave_tt}, we obtain \eqref{eq.emt_maxwell_lest_pre}.
\end{proof}

\begin{remark}
Note that both $-Q_w[U](T, L)$ and $-Q_m[A](T, L)$ are nonnegative.
\end{remark}

Next, we define the lower-order flux densities
\begin{align*}
\rho^1\paren{p, \tau} &= -Q_m\brak{R}\paren{T, L} - Q_w\brak{D \phi}\paren{T, L} \text{,} \\
\rho^1\paren{p, \tau} &= -Q_m\brak{R}\paren{T, L} - Q_w\brak{F}\paren{T, L}
\end{align*}
on $\mc{N}_\tau$ in the E-S and E-M cases, respectively, and we define the corresponding flux
\[ \mc{F}^1\paren{p, \tau} = \int_{\mc{N}_\tau} \rho^1\paren{p, \tau} \text{.} \]
We can now apply Lemma \ref{thm.emt_lest_pre} directly to obtain the desired local energy estimates.

\begin{proposition}\label{thm.e1_lest}
In both the E-S and E-M cases, we have the local energy inequality
\[ \mc{F}^1\paren{p, \tau} \lesssim 1 \text{.} \]
\end{proposition}

\begin{proof}
For the E-S case, we first apply \eqref{eq.emt_maxwell_lest_pre} with $A = R$ to obtain
\[ -\int_{\mc{N}_\tau} Q_m\brak{R}\paren{T, L} \lesssim \nabs{R}_{L^2\paren{\Sigma_\tau}} + \int_{\tau}^{t\paren{p}} \paren{\nabs{R}_{L^2\paren{\Sigma_\sigma}} + \nabs{D^2 \phi}_{L^2\paren{\Sigma_\sigma}}} d\sigma \lesssim 1 \text{,} \]
where we also used Proposition \ref{thm.e1_est}.
Similarly, applying \eqref{eq.emt_wave_lest_pre} with $U = D \phi$ yields
\[ -\int_{\mc{N}_\tau} Q_w\brak{D \phi}\paren{T, L} \lesssim 1 \text{.} \]
The E-M case is handled analogously.
\end{proof}

\begin{remark}
Estimates analogous to Proposition \ref{thm.e1_lest} hold for similar flux quantities associated with future regular null cones on the time interval $[\tau_0, t_1)$.
\end{remark}

As of now, we have uniform control on the integrals over $\mc{N}_\tau$ of some quadratic quantities in $R$ and either $D^2 \phi$ or $D F$, given explicitly by $-Q_m[R](T, L)$, $-Q_w[D \phi](T, L)$, and $-Q_w[F](T, L)$.
The final step here is to determine exactly which components of $R$ and the matter field are controlled by Proposition \ref{thm.e1_lest}.

\begin{proposition}\label{thm.null_flux}
On $\mc{N}_\tau$, we have the comparisons
\begin{align*}
-Q_w\brak{U}\paren{T, L} &\simeq \abs{\onasla_4 U}^2 + \abs{\onasla U}^2 \text{,} \\
-Q_m\brak{A}\paren{T, L} &\simeq \abs{A_{43}}^2 + \sum_{a = 1}^2 \abs{A_{4a}}^2 + \sum_{a, b = 1}^2 \abs{A_{ab}}^2 \text{,}
\end{align*}
where we have indexed with respect to adapted null frames on $\mc{N}^-(p)$, and where $A_{\alpha\beta} \in \Gamma T^r M$ is defined $(A_{\alpha\beta})_I = A_{\alpha\beta I}$.
\end{proposition}

In other words, $-Q_w[U](T, L)$ controls all the components of $D U$ except for those of the form $D_3 U_I$, i.e., derivatives transverse to $\mc{N}_\tau$.
Similarly, $-Q_m[A](T, L)$ controls all the components of $A$ except those of the form $A_{3aI}$.
This coincides with standard results for scalar and Maxwell fields in Minkowski space.

\begin{proof}
We begin by defining the quantity $\varphi = g(T, L)$, which satisfies
\[ \lim_{v \searrow 0} \valat{\varphi}_{\paren{v, \omega}} = 1 \text{,} \qquad \abs{\mc{L}_{t_p} \varphi} = n \varphi^{-1} \cdot \abs{L \varphi} \lesssim n \varphi^{-1} \abs{\pi} \abs{L}^2 \lesssim \varphi \text{,} \]
where we have applied \eqref{eq.nv_norm}.
Integrating the above along the null generators of $\mc{N}^-(p)$ and applying Gr\"onwall's inequality yields the comparison $\varphi \simeq 1$.
\footnote{Here, we also take advantage of the assumption $t(p) - \tau < 1$.}

Now, the proof is simply a matter of expanding $-Q_w[U](T, L)$ and $-Q_m[A](T, L)$ in terms of adapted null frames.
From these computations, we obtain
\begin{align*}
-Q_w\brak{U}\paren{T, L} &= \frac{1}{2} \varphi^{-1} \abs{\onasla_4 U}^2 + \frac{1}{2} \varphi \abs{\onasla U}^2 \text{,} \\
-Q_m\brak{A}\paren{T, L} &= \frac{1}{8} \varphi \abs{A_{43}}^2 + \frac{1}{2} \varphi^{-1} \sum_{a = 1}^2 \abs{A_{4a}}^2 + \frac{1}{4} \varphi \sum_{a, b = 1}^2 \abs{A_{ab}}^2 \text{.}
\end{align*}
For some more details on these computations, see \cite[Prop. 4.17]{shao:bdc_nv}.
The desired results now follow from the comparison $\varphi \simeq 1$.
\end{proof}

We also note the following simple consequence of the proof of Proposition \ref{thm.null_flux}, which will be essential in later estimates involving past null cones.

\begin{corollary}\label{thm.nl_bound}
On $\mc{N}_\tau$, the null lapse $\vartheta$ satisfies $\vartheta \simeq 1$.
\end{corollary}

\begin{proof}
Since $\vartheta = n \varphi^{-1}$, where $\varphi$ is defined as in the proof of Proposition \ref{thm.null_flux}, then the desired estimate follows from \eqref{eq.cond_T0} and the estimate $\varphi \simeq 1$.
\end{proof}

\begin{remark}
As a result, the quantity $\vartheta$ will be negligible in estimates.
In particular, the quantity $\onasla_4 U$ in Proposition \ref{thm.null_flux} can be replaced by $\onasla_{t_p} U$.
\end{remark}

Lastly, we use Proposition \ref{thm.null_flux} to see which null frame components of $R$ and the matter field are controlled using Proposition \ref{thm.e1_lest}.
First, for $R$, Proposition \ref{thm.null_flux} implies that the components which are not controlled are those of the form $R_{3a \alpha\beta}$.
However, using the symmetry $R_{3a \alpha\beta} = R_{\alpha\beta 3a}$, we see in fact that the uncontrolled components are those of the form $R_{3a3b}$.
More explicitly, if $\psi$ denotes a null frame component of $R$ not of the form $R_{3a3b}$, then
\[ \nabs{\psi}_{L^2\paren{\mc{N}_\tau}} \lesssim 1 \text{.} \]

\begin{remark}
The same result could also be obtained for $R$ by using instead the Bel-Robinson tensor field.
This was done in \cite{kl_rod:rin, kl_rod:bdc} in the E-V setting.
\end{remark}

Next, in the E-S case, Proposition \ref{thm.null_flux} implies that only the components $D_{3 \alpha} \phi$ of $D^2 \phi$ are not controlled.
Since $D^2 \phi$ is symmetric, then in fact, only $D_{33} \phi$ cannot be controlled; for any other null frame component $\psi$ of $D^2 \phi$, then
\[ \nabs{\psi}_{L^2\paren{\mc{N}_\tau}} \lesssim 1 \text{.} \]
Finally, in the E-M case, Proposition \ref{thm.null_flux} implies that only the components $D_3 F_{\alpha\beta}$ are not controlled.
By rearranging components via the Maxwell equations, we can easily establish that the only truly uncontrollable components of $D F$ are those of the form $D_3 F_{3a}$; for any other null frame component $\psi$, then
\[ \nabs{\psi}_{L^2\paren{\mc{N}_\tau}} \lesssim 1 \text{.} \]

\begin{remark}
The second fundamental form $k$ also satisfies a tensorial covariant wave equation.
One can also use this equation to derive analogous global and local energy estimates for $k$.
Such global energy estimates are equivalent to those obtained from the elliptic estimates of \eqref{eq.ell_est_1}, while the local estimates are a new result.
\end{remark}

\subsection{Preliminaries for Elliptic Estimates}\label{sec.ell_pre}

The last class of a priori estimates we will establish are elliptic estimates on the $\Sigma_\tau$'s.
Before beginning this task, we first focus on some general concepts which will be used in these derivations.

A primary idea behind these elliptic estimates is the relations that exist between corresponding horizontal and spacetime objects, e.g., the curvatures $\mc{R}$ of the $\Sigma_\tau$'s and $R$ of $M$.
Our next step is to briefly describe such relations in an informal schematic manner.
A more detailed account can be found throughout \cite{shao:bdc_nv}.

For any $A \in \Gamma \mc{T} M$, we let $\Pi A \in \Gamma \ul{\mc{T}} M$ denote its projection to the $\Sigma_\tau$'s.
In particular, if $A$ is a vector field, then $\Pi A$ is its orthogonal projection onto the tangent bundles of the $\Sigma_\tau$'s.
If $A$ is fully covariant, then $\Pi A$ is given trivially by
\[ \Pi A\paren{X_1, \ldots, X_n} = A\paren{X_1, \ldots, X_n} \text{,} \qquad X_1, \ldots, X_n \in \ul{\mf{X}}\paren{M} \text{.} \]
Furthermore, if $A$ is scalar, then $\Pi A = A$.

For simplicity, we will adopt an informal schematic language to describe tensorial expressions in an inexact form.
We describe the conventions here merely by example; for a more exact description of the schematic notation, see \cite{shao:bdc_nv}.
Consider the Gauss equations \eqref{eq.gauss_tf}, which yields the informal schematic relation
\begin{equation}\label{eq.sch_gauss} \mc{R} \approx \Pi R + k \cdot k \text{.} \end{equation}
Here, the symbol $\cdot$ represents a tensor product followed by a finite number of contractions and metric contractions.
Moreover, the multiple terms of the form $k \cdot k$ in \eqref{eq.gauss_tf} are compacted into a single term in the informal equation \eqref{eq.sch_gauss}.

We take advantage of such schematic relations by using the following lemma:

\begin{lemma}\label{thm.sch_D_tf}
If $\Phi \in \Gamma \mc{T} M$ is fully covariant, then we have the schematic relation
\begin{equation}\label{eq.sch_D_tf} \nabla \Pi \Phi \approx \Pi D \Phi + k \cdot \Pi \paren{\Phi \cdot T} \text{.} \end{equation}
\end{lemma}

\begin{proof}
Suppose $\Phi \in \Gamma T^0_r M$, $X \in \ul{\mf{X}}(M)$, and $Z \in \Gamma \ul{T}^r M$.
Then, by definition,
\[ \nabla_X \paren{\Pi \Phi}\paren{Z} = X \paren{\Phi\paren{Z}} - \Phi\paren{\nabla_X Z} = D_X \Phi\paren{Z} + \Phi\paren{D_X Z - \nabla_X Z} \text{.} \]
The first term on the right is simply $\Pi D \Phi$ applied to $X \otimes Z$.
Moreover, the quantity $D_X Z - \nabla_X Z$ can be expressed as a sum of terms of the form $k \cdot X \cdot Z \cdot T$.
\footnote{For example, if $Z \in \ul{\mf{X}}(M)$, then $D_X Z - \nabla_X Z = -g(D_X Z, T) T = -k(X, Z) T$.}
The desired equation \eqref{eq.sch_D_tf} follows as a result of the above observations.
\end{proof}

For instance, considering \eqref{eq.sch_gauss}, if we wish to control $\nabla \mc{R}$, then applying Lemma \ref{thm.sch_D_tf} with $\Phi \approx R$, we need only control terms of types $D R$, $k \cdot R \cdot T$, and $\nabla k \cdot k$.
\footnote{Note that $|\Pi \Phi| \leq |\Phi|$ for any $\Phi \in \Gamma \mc{T} M$.}
As a result of the above, we obtain the derivative bound
\begin{equation}\label{eq.sch_gauss_est} \abs{\nabla \mc{R}} \lesssim \abs{D R} + \abs{R} \abs{k} + \abs{\nabla k} \abs{k} \text{.} \end{equation}
With a little more care, we can iterate the above argument using \eqref{eq.sch_gauss} in order to estimate higher derivatives of $\mc{R}$.
In fact, such estimates for $\nabla^2 \mc{R}$ will be necessary later when we establish higher-order elliptic estimates.

In addition to Lemma \ref{thm.sch_D_tf}, we will also make use of the following general elliptic and B\"ochner estimates.
For any fully symmetric $\xi \in \Gamma \ul{T}^s M$, $s > 0$, we define
\begin{align*}
\nabla \cdot \xi \in \Gamma \ul{T}^{s-1} M \text{,} &\qquad \paren{\nabla \cdot \xi}_I = \nabla^a \xi_{aI} \text{,} \\
\nabla \times \xi \in \Gamma \ul{T}^{s+1} M \text{,} &\qquad \paren{\nabla \times \xi}_{abI} = \nabla_a \xi_{bI} - \nabla_b \xi_{aI} \text{,} \\
\trace \xi \in \Gamma \ul{T}^{s-2} M \text{,} &\qquad \begin{cases} \paren{\trace \xi}_I = \gamma^{ab} \xi_{abI} & s > 1 \text{,} \\ 0 & s = 1 \text{,} \end{cases}
\end{align*}
i.e., the divergence, curl, and trace of $\xi$, respectively.

\begin{lemma}\label{thm.ell_estimate}
If $\xi \in \Gamma \ul{\mc{T}} M$ is fully symmetric and nonscalar, then
\begin{equation}\label{eq.ell_first_est} \int_{\Sigma_\tau} \abs{\nabla \xi}^2 \lesssim \int_{\Sigma_\tau} \paren{\abs{\nabla \cdot \xi}^2 + \frac{1}{2} \abs{\nabla \times \xi}^2 + \abs{\mc{R}} \abs{\xi}^2} \text{.} \end{equation}
Moreover, for any $\Psi \in \Gamma \ul{\mc{T}} M$ and $\phi \in C^\infty(M)$,
\begin{align}
\label{eq.ell_bochner_est} \int_{\Sigma_\tau} \abs{\nabla^2 \Psi}^2 &\lesssim \int_{\Sigma_\tau} \paren{\abs{\lapl \Psi}^2 + \abs{\mc{R}} \abs{\nabla \Psi}^2 + \frac{1}{2} \abs{\mc{R}}^2 \abs{\Psi}^2} \text{,} \\
\notag \int_{\Sigma_\tau} \abs{\nabla^2 \phi}^2 &\lesssim \int_{\Sigma_\tau} \paren{\abs{\lapl \phi}^2 + \abs{\mc{R}} \abs{\nabla \phi}^2} \text{.}
\end{align}
\end{lemma}

\begin{proof}
See \cite[Lemma 2.3, Cor. 2.5]{shao:bdc_nv}; see also \cite[Ch. 2]{chr_kl:stb_mink}.
\end{proof}

In particular, Lemma \ref{thm.ell_estimate} will be useful for elliptic estimates on $k$ and $n$.
Variants of Lemma \ref{thm.ell_estimate} will also be essential for controlling the local null geometry.

\subsection{Elliptic Estimates}\label{sec.ell}

Recall that our primary task in proving Theorem \ref{thm.bdc} is to control the horizontal energy quantities $\mf{K}(\tau)$, $\mf{R}(\tau)$, and $\mf{f}(\tau)$ defined in \eqref{eq.bdc_energy_KR}-\eqref{eq.bdc_energy_fem}.
To accomplish this, we must apply various elliptic estimates in conjunction with energy estimates derived using EMTs, for example, Propositions \ref{thm.e1_est} and \ref{thm.e1_lest}.
In addition, we must control $L^2$-norms of derivatives of the lapse $n$, as well as obtain an $L^\infty$-bound for $\mc{L}_t n$.
These will be necessary in order to control the geometry of null cones and derive higher-order elliptic estimates.

We begin here by deriving lower-order a priori elliptic estimates for $k$ and $\mc{R}$, as well as the horizontal formulation $\mf{f}$ of the matter field.
These estimates are the E-S and E-M analogues for the elliptic estimates of \cite[Sec. 8]{kl_rod:bdc} in the vacuum case.
The main difference here is that we must also bound the matter field.

\begin{remark}
The proofs of \cite[Sec. 8]{kl_rod:bdc} were adapted for maximal foliations.
Here, we simplify the process by taking advantage of the volume bounds of Proposition \ref{thm.cmc_bound}.
\end{remark}

Recall that the ``horizontal matter field quantities" $\mf{f}$ are given as follows:
\begin{itemize}
\item In the E-S setting, this consists of the scalar field $\phi$ itself and $\mc{L}_t \phi$.

\item In the E-M setting, this consists of the pair $E, H \in \ul{\mf{X}}(M)$ representing the electromagnetic decomposition of $F$.
\end{itemize}

\begin{remark}
Recall that given a set orientation of $M$ (or equivalently, of $\Sigma_{\tau_0}$), the electromagnetic decomposition of $F$ is given by the formulas
\begin{equation}\label{eq.electromagnetic} E_i = F_{\alpha i} T^\alpha \text{,} \qquad H_i = {}^\star F_{\alpha i} T^\alpha = \frac{1}{2} \epsilon_{\alpha i}{}^{\beta\gamma} F_{\beta\gamma} T^\alpha \text{.} \end{equation}
\end{remark}

First, by Proposition \ref{thm.cmc_bound} and \eqref{eq.bdc}, we have the following trivial bounds:

\begin{lemma}\label{thm.e0_est_ell}
For any $1 \leq p \leq \infty$ and $\tau_0 \leq \tau < t_1$, we have the estimates
\begin{align*}
\nabs{k}_{L^p\paren{\Sigma_\tau}} + \nabs{\nabla \phi}_{L^p\paren{\Sigma_\tau}} + \nabs{\mc{L}_t \phi}_{L^p\paren{\Sigma_\tau}} &\lesssim 1 \text{,} \\
\nabs{k}_{L^p\paren{\Sigma_\tau}} + \nabs{E}_{L^p\paren{\Sigma_\tau}} + \nabs{H}_{L^p\paren{\Sigma_\tau}} &\lesssim 1
\end{align*}
in the E-S and E-M cases, respectively.
\end{lemma}

The next batch of elliptic estimates are consequences of various schematic relations between horizontal and spacetime objects.

\begin{proposition}\label{thm.e1_est_ell}
For any $\tau_0 \leq \tau < t_1$, we have the elliptic estimates
\begin{equation}\label{eq.ell_est_1} \nabs{\mc{R}}_{L^2\paren{\Sigma_\tau}} + \nabs{\nabla k}_{L^2\paren{\Sigma_\tau}} \lesssim 1 \text{.} \end{equation}
Moreover, for any $\tau_0 \leq \tau < t_1$, the estimates
\begin{align}
\label{eq.ell_est_f1} \nabs{\nabla^2 \phi}_{L^2\paren{\Sigma_\tau}} + \nabs{\nabla \paren{\mc{L}_t \phi}}_{L^2\paren{\Sigma_\tau}} &\lesssim 1 \text{,} \\
\notag \nabs{\nabla E}_{L^2\paren{\Sigma_\tau}} + \nabs{\nabla H}_{L^2\paren{\Sigma_\tau}} &\lesssim 1
\end{align}
hold in the E-S and E-M cases, respectively.
\end{proposition}

\begin{proof}
From \eqref{eq.gauss_tf}, we immediately obtain
\[ \nabs{\mc{R}}_{L^2\paren{\Sigma_\tau}} \lesssim \nabs{R}_{L^2\paren{\Sigma_\tau}} + \nabs{k}_{L^\infty\paren{\Sigma}} \nabs{k}_{L^2\paren{\Sigma_\tau}} \lesssim 1 + \nabs{R}_{L^2\paren{\Sigma_\tau}} \text{.} \]
To bound $\nabla k$, we appeal to \eqref{eq.ell_first_est}.
From \eqref{eq.codazzi_tf} and the CMC gauge condition, we obtain $|\nabla \cdot k| + |\nabla \times k| \lesssim |R|$, hence by \eqref{eq.ell_first_est}, Lemma \ref{thm.e0_est_ell}, and the above,
\[ \nabs{\nabla k}_{L^2\paren{\Sigma_\tau}} \lesssim \nabs{R}_{L^2\paren{\Sigma_\tau}} + \paren{\int_{\Sigma_\tau} \abs{\mc{R}} \abs{k}^2}^\frac{1}{2} \lesssim 1 + \nabs{R}_{L^2\paren{\Sigma_\tau}} \text{.} \]
Combining the above and applying Proposition \ref{thm.e1_est} yields \eqref{eq.ell_est_1}.

For the matter field estimates of \eqref{eq.ell_est_f1}, we appeal to the informal schematic language.
In the E-S setting, we have the schematic relations
\begin{equation}\label{eq.sch_sf} \nabla \phi \approx \Pi D \phi \text{,} \qquad n^{-1} \mc{L}_t \phi \approx \Pi \paren{D \phi \cdot T} \text{.} \end{equation}
Similarly, in the E-M setting,
\begin{equation}\label{eq.sch_mw} E \approx \Pi \paren{F \cdot T} \text{,} \qquad H \approx \Pi \paren{V \cdot F \cdot T} \text{,} \end{equation}
where $V$ is the volume form for $M$ (given an orientation of $M$).
The estimates \eqref{eq.ell_est_f1} then follow immediately after applying Lemma \ref{thm.sch_D_tf} and Proposition \ref{thm.e1_est} to the schematic equations \eqref{eq.sch_sf} and \eqref{eq.sch_mw}.
For details, see \cite[Lemma 6.5]{shao:bdc_nv}.
\end{proof}

We can use the lapse equation \eqref{eq.lapse} in order to derive some higher-order a priori bounds for $n$.
These are described in the next two propositions:

\begin{proposition}\label{thm.lapse_est_ell}
For any $\tau_0 \leq \tau < t_1$, we have the estimates
\begin{align}
\label{eq.lapse0_est_ell} \nabs{\nabla^2 n}_{L^2\paren{\Sigma_\tau}} &\lesssim 1 \text{,} \\
\label{eq.lapse1_est_lemma} \nabs{-\lapl \paren{\mc{L}_t n} + \brak{\abs{k}^2 + \ric\paren{T, T}} \mc{L}_t n}_{L^2\paren{\Sigma_\tau}} &\lesssim 1 \text{,} \\
\label{eq.lapse1_est_ell} \nabs{\nabla \paren{\mc{L}_t n}}_{L^2\paren{\Sigma_\tau}} + \nabs{\mc{L}_t n}_{L^2\paren{\Sigma_\tau}} &\lesssim 1 \text{.}
\end{align}
\end{proposition}

\begin{proof}
Equations \eqref{eq.bdc}, \eqref{eq.lapse}, and Proposition \ref{thm.cmc_bound} imply $\|\lapl n\|_{L^\infty(\Sigma_\tau)} \lesssim 1$.
Then,
\[ \int_{\Sigma_\tau} \abs{\nabla^2 n}^2 \lesssim \int_{\Sigma_\tau} \paren{\abs{\lapl n}^2 + \abs{\mc{R}} \abs{\nabla n}^2} \lesssim 1 + \int_{\Sigma_\tau} \abs{\mc{R}}^2 \text{,} \]
by \eqref{eq.ell_bochner_est} and Lemma \ref{thm.e0_est_ell}; applying \eqref{eq.ell_est_1} results in \eqref{eq.lapse0_est_ell}.

The proofs of \eqref{eq.lapse1_est_lemma} and \eqref{eq.lapse1_est_ell} are a bit more involved.
Recall that $\mc{L}_t$ applied to a scalar $\phi$ is the same as $n T \phi$.
We begin by commuting $\mc{L}_t$ with $\lapl$ to obtain
\[ \abs{\lapl \mc{L}_t n - \mc{L}_t \lapl n} \lesssim \abs{n} \abs{k} \abs{\nabla^2 n} + \paren{\abs{\nabla n \cdot k} + \abs{n \cdot \nabla k}} \abs{\nabla n} \text{,} \]
where we also used \eqref{eq.dt_met}.
Applying \eqref{eq.bdc}, Proposition \ref{thm.e1_est_ell}, and \eqref{eq.lapse0_est_ell} yields
\begin{equation}\label{eql.lapse1_est_ell_1} \nabs{\lapl \mc{L}_t n - \mc{L}_t \lapl n}_{L^2\paren{\Sigma_\tau}} \lesssim 1 \text{.} \end{equation}

Next, differentiating \eqref{eq.lapse}, we have
\[ \mc{L}_t \lapl n = \brak{\abs{k}^2 + \ric\paren{T, T}} \mc{L}_t n + n \cdot \mc{L}_t \abs{k}^2 + n \cdot \mc{L}_t \brak{\ric\paren{T, T}} \text{.} \]
Moreover, applying \eqref{eq.dt_sff} and \eqref{eq.cond_T0}, we obtain
\[ \abs{n \cdot \mc{L}_t \abs{k}^2} \lesssim 1 + \abs{R} + \abs{\nabla^2 n} \text{.} \]
If we let $\mf{F}$ denote $D \phi$ in the E-S case and $F$ in the E-M case, then by \eqref{eq.einstein_ric} and \eqref{eq.bdc}, we see that $|D \ric| \lesssim |D \mf{F}|$.
As a result,
\[ \abs{n \cdot \mc{L}_t \brak{\ric\paren{T, T}}} \lesssim \abs{D \ric} + \abs{D_T T} \abs{\ric} \lesssim 1 + \abs{D \mf{F}} \text{.} \]
Combining Proposition \ref{thm.e1_est}, \eqref{eq.lapse0_est_ell}, \eqref{eql.lapse1_est_ell_1}, and the above, we obtain \eqref{eq.lapse1_est_lemma}.

Define now the scalar
\[ P = -\lapl \mc{L}_t n + \brak{\abs{k}^2 + \ric\paren{T, T}} \mc{L}_t n \text{.} \]
Multiplying $P$ by $\mc{L}_t n$, integrating by parts, and decomposing $|k|^2$, then
\begin{align*}
\int_{\Sigma_\tau} \paren{\abs{\nabla \mc{L}_t n}^2 + \frac{1}{3} t_1^2 \abs{\mc{L}_t n}^2} &\leq \int_{\Sigma_\tau} \brak{\abs{\nabla \mc{L}_t n}^2 + \paren{\frac{1}{3} \tau^2 + \abs{\hat{k}}^2 + \ric\paren{T, T}} \abs{\mc{L}_t n}^2} \\
&= \int_{\Sigma_\tau} \abs{P} \abs{\mc{L}_t n} \text{,}
\end{align*}
where we have also used the strong energy condition satisfied by both the E-S and E-M models.
Applying a weighted Cauchy inequality to the right-hand side along with \eqref{eq.lapse1_est_lemma} yields \eqref{eq.lapse1_est_ell}, as desired.
\end{proof}

\begin{proposition}\label{thm.lapse_unif_est_ell}
For any $\tau_0 \leq \tau < t_1$,
\[ \nabs{\nabla^2 \mc{L}_t n}_{L^2\paren{\Sigma_\tau}} + \nabs{\mc{L}_t n}_{L^\infty\paren{\Sigma_\tau}} \lesssim 1 \text{.} \]
\end{proposition}

\begin{proof}
By \eqref{eq.bdc}, \eqref{eq.ell_bochner_est}, and \eqref{eq.ell_est_1},
\begin{align*}
\nabs{\nabla^2 \mc{L}_t n}_{L^2\paren{\Sigma_\tau}} &\lesssim \nabs{\lapl \mc{L}_t n}_{L^2\paren{\Sigma_\tau}} + \nabs{\mc{R}}_{L^2\paren{\Sigma_\tau}}^\frac{1}{2} \nabs{\nabla \mc{L}_t n}_{L^4\paren{\Sigma_\tau}} \\
&\lesssim \nabs{\mc{L}_t n}_{L^2\paren{\Sigma_\tau}} + \nabs{-\lapl \mc{L}_t n + \brak{\abs{k}^2 + \ric\paren{T, T}} \mc{L}_t n}_{L^2\paren{\Sigma_\tau}} \\
&\qquad + \nabs{\nabla \mc{L}_t n}_{L^4\paren{\Sigma_\tau}} \text{.}
\end{align*}
Applying Proposition \ref{thm.lapse_est_ell}, then
\[ \nabs{\nabla^2 \mc{L}_t n}_{L^2\paren{\Sigma_\tau}} \lesssim 1 + \nabs{\nabla \mc{L}_t n}_{L^4\paren{\Sigma_\tau}} \text{.} \]
Next, by the Sobolev estimate \cite[Cor 2.7]{kl_rod:bdc}, along with \eqref{eq.lapse1_est_ell}, then
\begin{align*}
\nabs{\nabla^2 \mc{L}_t n}_{L^2\paren{\Sigma_\tau}} &\lesssim 1 + \nabs{\nabla^2 \mc{L}_t n}_{L^2\paren{\Sigma_\tau}}^\frac{3}{4} \nabs{\nabla \mc{L}_t n}_{L^2\paren{\Sigma_\tau}}^\frac{1}{4} + \nabs{\nabla \mc{L}_t n}_{L^2\paren{\Sigma}} \\
&\lesssim \paren{1 + \epsilon^{-4}} + \epsilon^\frac{4}{3} \nabs{\nabla^2 \mc{L}_t n}_{L^2\paren{\Sigma_\tau}}
\end{align*}
for any $\epsilon > 0$.
Taking $\epsilon$ sufficiently small, we obtain the desired bound for $\nabla^2 \mc{L}_t n$.
Lastly, we can bound $\mc{L}_t n$ in $L^\infty$ by the Sobolev bound \eqref{eq.sob_2h} and the above.
\end{proof}

The uniform bound for $\mc{L}_t n$ in Proposition \ref{thm.lapse_unif_est_ell}, omitted in the earliest versions of the works in the E-V setting, is essential for controlling null injectivity radii.

\section{Local Regularity of Null Cones}

In this section, we will summarize what is the most technically demanding step in the proof of Theorem \ref{thm.bdc}: control of the geometry of regular past null cones.
This step will be crucial for applying the generalized representation formula for tensor wave equations, which will be needed in order to obtain the higher-order energy bounds required to complete the proof of Theorem \ref{thm.bdc}.

Throughout this section, we fix an arbitrary point $p \in M_+$, and we normalize and foliate the regular past null cone $\mc{N}^-(p)$ by $T|_p$ and $t_p = t(p) - t$, as usual.
In addition, we fix a constant $0 < \delta_0 \leq \min(\mf{i}(p), 1)$, and we define the segment
\[ \mc{N} = \mc{N}^-\paren{p; \delta_0} = \brac{z \in \mc{N}^-\paren{p} \mid t_p\paren{z} < \delta_0} \text{.} \]

The main result of this section states roughly the following:
\begin{itemize}
\item The past null injectivity radius $\mf{i}(p)$ of $p$ is bounded below by some constant $\delta > 0$ depending only on the fundamental constants.

\item For small enough $\delta_0$, depending only on the fundamental constants, the Ricci coefficients $\chi$, $\ul{\chi}$, $\zeta$, $\ul{\eta}$ of $\mc{N}^-(p)$ (see Section \ref{sec.rc}) can be controlled on $\mc{N}$ in various norms by the fundamental constants.
\end{itemize}
The explicit result is stated in Theorem \ref{thm.nc}, and a detailed proof of the theorem is available in \cite[Ch. 7-8]{shao:bdc_nv}.
\footnote{See also \cite{kl_rod:cg, kl_rod:stt, kl_rod:rin, parl:bdc, wang:cg, wang:cgp} for earlier results in Einstein-vacuum settings.}
Due to the prohibitive length and the level of technical detail of the argument, we omit the proof of this theorem in this paper and refer the reader to \cite{shao:bdc_nv} for details.
\footnote{In fact, a complete accounting of the details of the proof of Theorem \ref{thm.nc} would more than double the length of this paper!}
On the other hand, we do provide a brief outline of the main components of the proof at the end of this section.
Our focus, however, will be on the consequences of Theorem \ref{thm.nc} in relation to the breakdown problem.

\subsection{Motivation and Past Results}

In Riemannian geometry, one generally requires $L^\infty$-bounds on the curvature in order to derive uniform lower bounds for the injectivity radius.
Similarly, for the Lorentzian case, given an $L^\infty$-bound for $R$, we can control the Ricci coefficients $\chi$, $\ul{\chi}$, $\zeta$, and $\ul{\eta}$ without too much effort.
Uniform bounds for the null injectivity radius would also follow.

Unfortunately, we will only have the local flux bounds of Proposition \ref{thm.e1_lest} here, which makes our task tremendously more difficult.
Indeed, bounding the Ricci coefficients using $L^2$ rather than $L^\infty$-estimates will necessitate the use of sharp trace estimates along the generators of null cones, which then requires the geometric Littlewood-Paley theory of \cite{kl_rod:glp} and the resulting Besov estimates.
This process is responsible for much of the technical difficulties in the breakdown problem.

Such $L^2$-curvature results were first obtained by S. Klainerman and I. Rodnianski in \cite{kl_rod:cg, kl_rod:stt, kl_rod:rin}.
The results of \cite{kl_rod:cg, kl_rod:stt}, however, applied only to truncated null cones and hence were not directly applicable to the breakdown problem.
To address this issue, Q. Wang, in \cite{wang:cg, wang:cgp}, extended the estimates in \cite{kl_rod:cg, kl_rod:stt} to past null cones with vertex initial data by keeping track of scaling factors at every step.
\footnote{The works \cite{wang:cg, wang:cgp} also addressed an error in \cite{kl_rod:cg}, which assumed a Besov bound that does not hold.  In fact, a significant amount of effort in \cite{wang:cg, wang:cgp} is dedicated to addressing this issue.}
Later, D. Parlongue, in \cite{parl:bdc}, revisited this argument for time-foliated null cones.

In all the above, an essential assumption is that the spacetime is vacuum.
The primary contribution of this section is the extension of this family of results to the E-S and E-M settings.
We offer now a more detailed comparison between the current results and its predecessors:
\begin{itemize}
\item The papers \cite{kl_rod:cg, kl_rod:stt, parl:bdc} all considered the case of truncated regular null cones with prescribed spherical initial data.
In contrast, we adopt the setting of \cite{wang:cg, wang:cgp} and consider the case of regular null cones with initial data given by a vertex point.
As mentioned before, this is the type of result needed by the breakdown problems in this text as well as in \cite{kl_rod:bdc, parl:bdc}.

\item Unlike \cite{kl_rod:cg, kl_rod:stt, wang:cg, wang:cgp}, which bounded the Ricci coefficients with respect to the geodesic foliation, we obtain these bounds in terms of the time foliation, like in \cite{parl:bdc}.
This has the advantage of being able to interface directly with the breakdown problem.
In this ``gauge", we will need the regularity properties of the time foliation in order to bound the Ricci coefficients.
However, this also results in improved estimates for both $\ul{\eta}$ and $\ul{\chi}$.

\item While \cite{kl_rod:cg, kl_rod:stt, kl_rod:rin, parl:bdc, wang:cg, wang:cgp} dealt exclusively with the Einstein-vacuum setting, here we extend the results to E-S and E-M spacetimes.
In particular, the ``curvature flux" defined in \cite{kl_rod:cg, kl_rod:rin, parl:bdc, wang:cg, wang:cgp} is replaced by an analogous flux quantity involving both the curvature and the matter field; see Proposition \ref{thm.e1_lest}.
Furthermore, several of the structure equations governing the Ricci coefficients now contain additional terms reflecting the contributions from the nontrivial matter field.

\item The works \cite{kl_rod:cg, kl_rod:stt, parl:bdc, wang:cg, wang:cgp} all dealt only with the special case of very small curvature flux and initial values, over a unit interval along (possibly truncated) null cones.
This is, however, not directly applicable to the breakdown problem, since in this setting we can only stipulate bounded, not small, curvature flux.
In our case, Propositions \ref{thm.e1_est} and \ref{thm.e1_lest} only established the existence of some possibly large a priori energy and flux bounds.
On the other hand, we can in the breakdown problem work in only a very small interval along null cones.
Heuristically, one can see that this ``large flux, small interval" case relates to the ``small flux, unit interval" case via a rescaling argument.
However, we handle the ``large flux, small interval" setting directly in the statement and proof of Theorem \ref{thm.nc}.

\item The only paper in the existing literature to address the null injectivity radius is \cite{kl_rod:rin}, which was separated from the remaining components \cite{kl_rod:cg, kl_rod:stt} in the overall proof.
In the time foliation case, however, the interplay between the ``Ricci coefficients" and the ``null injectivity radius" portions of the proof is more subtle and must be addressed in tandem.
\end{itemize}

\subsection{Integral Norms}

In order to state the main result, we must first define the relevant integral norms used within.
We begin with natural integral norms on $\mc{N}$ (or $\mc{N}^-(p)$): for any $1 \leq q < \infty$ and $\Psi \in \Gamma \ul{\mc{T}} \ol{\mc{T}} \mc{N}^-(p)$, we define
\begin{equation}\label{eq.pnc_norm_Lp} \nabs{\Psi}_{L^q\paren{\mc{N}}} = \paren{\int_{\mc{N}} \abs{\Psi}^q}^\frac{1}{q} \text{,} \qquad \nabs{\Psi}_{L^\infty\paren{\mc{N}}} = \sup_{z \in \mc{N}} \valat{\abs{\Psi}}_z \text{,} \end{equation}
where mixed tensors are normed with respect to $h$ and $\lambda$, and where the above integral over $\mc{N}$ is defined using the formula \eqref{eq.pnc_int}.
In addition, we will make use of the following ``null Sobolev" norm on $\mc{N}$:
\begin{equation}\label{eq.pnc_norm_sob} \nabs{\Psi}_{\mc{H}^1\paren{\mc{N}}} = \nabs{\onasla_{t_p} \Psi}_{L^2\paren{\mc{N}}} + \nabs{\onasla \Psi}_{L^2\paren{\mc{N}}} + \nabs{t_p^{-1} \Psi}_{L^2\paren{\mc{N}}} \text{.} \end{equation}

Given $1 \leq q < \infty$ and $1 \leq r \leq \infty$, we also define the following iterated norms:
\begin{align}
\label{eq.pnc_norm_iter} \nabs{\Psi}_{L^q_t L^r \paren{\mc{N}}} &= \paren{\int_0^{\delta_0} \nabs{\Psi}_{L^r\paren{\mc{S}_v}}^q dv}^\frac{1}{q} \text{,} \\
\notag \nabs{\Psi}_{L^\infty_t L^r \paren{\mc{N}}} &= \sup_{0 < v < \delta_0} \nabs{\Psi}_{L^r\paren{\mc{S}_v}} \text{.}
\end{align}
As before, we use $\mc{S}_v$ to denote the level set $\{z \in \mc{N} \mid t_p(z) = v\}$, while $L^r(\mc{S}_v)$ refers to the tensorial $L^r$-norm on $\mc{S}_v$ with respect to the induced metric $\lambda$ on $\mc{S}_v$.
Since $\vartheta \simeq 1$ on $\mc{N}$ due to Corollary \ref{thm.nl_bound}, then by \eqref{eq.pnc_int},
\[ \nabs{\Psi}_{L^q_t L^q \paren{\mc{N}}} \simeq \nabs{\Psi}_{L^q\paren{\mc{N}}} \text{,} \qquad 1 \leq q \leq \infty \text{,} \]
i.e., the $L^q_t L^q$ and $L^q$-norms on $\mc{N}$ are equivalent.

Recall that $\mc{N}$ can be parametrized by the $t_p$-value $0 < v < \delta_0$ and a spherical parameter $\omega \in \Sph^2$.
We can then define additional iterated norms with respect to this parametrization: for any $1 \leq q, r < \infty$, we define
\begin{align}
\label{eq.pnc_norm_param} \nabs{\Psi}_{L^q_t L^r_\omega \paren{\mc{N}}} &= \brak{\int_0^{\delta_0} \paren{\int_{\Sph^2} \valat{\abs{\Psi}^r}_{\paren{v, \omega}} d\omega}^\frac{q}{r} dv}^\frac{1}{q} \text{,} \\
\notag \nabs{\Psi}_{L^\infty_t L^r_\omega \paren{\mc{N}}} &= \sup_{0 < v < \delta_0} \paren{\int_{\Sph^2} \valat{\abs{\Psi}^{r}}_{\paren{v, \omega}} d\omega}^\frac{1}{r} \text{,} \\
\notag \nabs{\Psi}_{L^q_t L^\infty_\omega \paren{\mc{N}}} &= \brak{\int_0^{\delta_0} \paren{\sup_{\omega \in \Sph^2} \valat{\abs{\Psi}}_{\paren{v, \omega}}}^q dv}^\frac{1}{q} \text{.}
\end{align}
We can reverse the order of integration to obtain additional useful norms:
\begin{align}
\label{eq.pnc_norm_rparam} \nabs{\Psi}_{L^r_\omega L^q_t \paren{\mc{N}}} &= \brak{\int_{\Sph^2} \paren{\int_0^{\delta_0} \valat{\abs{\Psi}^q}_{\paren{v, \omega}} dv}^\frac{r}{q} d\omega}^\frac{1}{r} \text{,} \\
\notag \nabs{\Psi}_{L^r_\omega L^\infty_t \paren{\mc{N}}} &= \brak{\int_{\Sph^2} \paren{\sup_{0 < v < \delta_0} \valat{\abs{\Psi}}_{\paren{v, \omega}}}^r d\omega}^\frac{1}{r} \text{,} \\
\notag \nabs{\Psi}_{L^\infty_\omega L^q_t \paren{\mc{N}}} &= \sup_{\omega \in \Sph^2} \paren{\int_0^{\delta_0} \valat{\abs{\Psi}}_{\paren{v, \omega}}^q d\omega}^\frac{1}{q} \text{.}
\end{align}
Lastly, the $L^\infty_t L^\infty_\omega (\mc{N})$ and $L^\infty_\omega L^\infty_t (\mc{N})$-norms coincide with the $L^\infty(\mc{N})$-norm.

Later, we will see that the parameter $L^q_t L^r_\omega$-norms are comparable to rescalings of the natural iterated $L^q_t L^r$-norms.
Of the reverse parametrized norms, the most important will be the ``trace" $L^\infty_\omega L^2_t$-norms and $L^q_\omega L^\infty_t$-norms.

\subsection{The Main Result}

With all of the appropriate norms defined, we can now state the main result of this section.

\begin{theorem}\label{thm.nc}
For any $q \in M_+$, the following hold:
\begin{itemize}
\item There exists a constant $\delta_0 > 0$, depending only on the fundamental constants, such that $\mf{i}(q) > \min(\delta_0, t(q) - \tau_0)$.

\item Letting $\mc{N}_q = \mc{N}^-(q; \min(\delta_0, t(q) - \tau_0))$, then the following estimates hold:
\begin{align}
\label{eq.pnc_main_est} \nabs{\vartheta \paren{\trace \chi} - \frac{2}{t_q}}_{L^\infty_\omega L^2_t \paren{\mc{N}_q}} + \nabs{\hat{\chi}}_{L^\infty_\omega L^2_t \paren{\mc{N}_q}} + \nabs{\zeta}_{L^\infty_\omega L^2_t \paren{\mc{N}_q}} &\lesssim 1 \text{,} \\
\notag \nabs{\vartheta \paren{\trace \chi} - \frac{2}{t_q}}_{\mc{H}^1\paren{\mc{N}_q}} + \nabs{\hat{\chi}}_{\mc{H}^1\paren{\mc{N}_q}} + \nabs{\zeta}_{\mc{H}^1\paren{\mc{N}_q}} &\lesssim 1 \text{,}
\end{align}
where $\chi$ and $\zeta$ refer to the corresponding Ricci coefficients of $\mc{N}^-(q)$.

\item In addition, for $\mc{N}_q$ as above, we have
\begin{align}
\label{eq.pnc_main_est_alt} \nabs{\vartheta \paren{\trace \chi} - \frac{2}{t_q}}_{L^\infty\paren{\mc{N}_q}} + \nabs{t_q^\frac{3}{2} \nasla \paren{\trace \chi}}_{L^2_\omega L^\infty_t \paren{\mc{N}_q}} &\lesssim 1 \text{,} \\
\notag \nabs{t_q^\frac{3}{2} \mu}_{L^2_\omega L^\infty_t \paren{\mc{N}_q}} + \nabs{\mu}_{L^2\paren{\mc{N}_q}} &\lesssim 1 \text{,}
\end{align}
where $\mu$ is the mass aspect function on $\mc{N}_q$ defined in \eqref{eq.maf}.
\end{itemize}
\end{theorem}

\begin{remark}
An analogue of Theorem \ref{thm.nc} holds for future null cones $\mc{N}^+(p)$.
\end{remark}

With greater care throughout the proof of Theorem \ref{thm.nc}, we could in fact be more precise about the bounds for the Ricci coefficients in \eqref{eq.pnc_main_est} and \eqref{eq.pnc_main_est_alt}, and how they are affected by the fundamental constants.
Such a task is significantly simpler in the geodesically foliated E-V settings examined in \cite{kl_rod:cg, wang:cg, wang:cgp}, since in these cases, the only external parameter in the problems is the curvature flux.

Again, the details of the proof of Theorem \ref{thm.nc} are left to \cite[Ch. 7-8]{shao:bdc_nv} due to its length and the amount of technical background involved.
A brief summary of the proof is given at the end of this section, after first discussing the consequences of Theorem \ref{thm.nc} pertaining to the current breakdown problem.

\subsection{Some Basic Consequences}\label{sec.pnc_cor}

We now list some basic consequences of Theorem \ref{thm.nc} which will be useful in upcoming analyses involving null cones.
Recall that the induced metrics on the $\mc{S}_v$'s are denoted by $\lambda$.
We define the ``rescaled" metrics on these level sets by $\bar{\lambda} = (t_p)^{-2} \lambda$.
This rescaling is essential, since in the case of null cones with vertex initial data, the uniformities occur with respect to $\bar{\lambda}$ rather than $\lambda$.
For example, from Proposition \ref{thm.met_pnc_init}, we observe that $\bar{\lambda}$ tends toward the standard Euclidean metric on $\Sph^2$ (rescaled by a factor of $n(p)^2$) as $t_p \searrow 0$.

One consequence of Proposition \ref{thm.met_pnc_init} and Theorem \ref{thm.nc}, then, is that as long as the timespan $\delta_0$ of $\mc{N}$ is sufficiently small, then $\bar{\lambda}$ does not differ much from the standard Euclidean metric on $\Sph^2$.
Another related result is that for similarly small $\delta_0$, then the volume forms $\bar{V}$ on the $\mc{S}_v$'s with respect to $\bar{\lambda}$ do not differ much from that of the standard Euclidean metric on $\Sph^2$.
Precise statements, in terms of transported coordinate systems, can be found in \cite[Lemma 7.1, Lemma 7.3]{shao:bdc_nv}.

This latter property immediately implies the following integral comparisons:

\begin{proposition}\label{thm.pnc_vol}
Suppose $\delta_0 \leq \mf{i}(p)$ is sufficiently small with respect to the fundamental constants.
Then, for every $0 < v < \delta_0$ and $\Psi \in \Gamma \ul{\mc{T}} \ol{\mc{T}} \mc{N}$, we have
\[ \int_{\mc{S}_v} \abs{\Psi} \simeq v^2 \int_{\Sph^2} \valat{\abs{\Psi}}_{\paren{v, \omega}} d\omega \text{.} \]
In particular, for any $1 \leq q, r \leq \infty$, we also have
\[ \nabs{\Psi}_{L^q_t L^r \paren{\mc{N}}} \simeq \nabs{t_p^\frac{2}{r} \Psi}_{L^q_t L^r_\omega \paren{\mc{N}}} \text{,} \qquad \nabs{\Psi}_{L^q\paren{\mc{N}}} \simeq \nabs{\Psi}_{L^q_t L^q \paren{\mc{N}}} \simeq \nabs{t_p^\frac{2}{q} \Psi}_{L^q_t L^q_\omega \paren{\mc{N}}} \text{.} \]
\end{proposition}

In other words, the natural iterated $L^q_t L^r$-norms are equivalent to the $(v, \omega)$-parametrized norms.
The latter representation is especially important, since it is expressed as an integral over the product $(0, \delta_0) \times \Sph^2$, and hence the order of integration can be easily reversed.
These reversed $L^r_\omega L^p_t$-norms will be essential in numerous estimates, especially within the proof of Theorem \ref{thm.nc}.

The regularity of the metric $\bar{\lambda}$, as described in \cite[Lemma 7.3]{shao:bdc_nv}, can be used to derive uniform first and second-order Sobolev estimates in a process analogous to that of Proposition \ref{thm.sobolev_tf}.
The results are given below:

\begin{proposition}\label{thm.pnc_sob}
Suppose $\delta_0 \leq \mf{i}(p)$ is sufficiently small with respect to the fundamental constants.
Then, for every $0 < v < \delta_0$, $2 < r < \infty$, and $\Psi \in \Gamma \ul{\mc{T}} \mc{N}$,
\begin{align*}
\nabs{\Psi}_{L^\infty\paren{\mc{S}_v}} &\lesssim v^{1 - \frac{2}{r}} \nabs{\nasla \Psi}_{L^r\paren{\mc{S}_v}} + v^{-\frac{2}{r}} \nabs{\Psi}_{L^r\paren{\mc{S}_v}} \text{,} \\
\nabs{\Psi}_{L^r\paren{\mc{S}_v}} &\lesssim \nabs{\nasla \Psi}_{L^2\paren{\mc{S}_v}}^{1 - \frac{2}{r}} \nabs{\Psi}_{L^2\paren{\mc{S}_v}}^\frac{2}{r} + v^{\frac{2}{r} - 1} \nabs{\Psi}_{L^2\paren{\mc{S}_v}} \\
&\lesssim v^\frac{2}{r} \nabs{\nasla \Psi}_{L^2\paren{\mc{S}_v}} + v^{\frac{2}{r} - 1} \nabs{\Psi}_{L^2\paren{\mc{S}_v}} \text{,} \\
\nabs{\Psi}_{L^\infty\paren{\mc{S}_v}} &\lesssim \nabs{\nasla^2 \Psi}_{L^2\paren{\mc{S}_v}}^\frac{1}{2} \nabs{\Psi}_{L^2\paren{\mc{S}_v}}^\frac{1}{2} + v^{-1} \nabs{\Psi}_{L^2\paren{\mc{S}_v}} \\
&\lesssim v \nabs{\nasla^2 \Psi}_{L^2\paren{\mc{S}_v}} + v^{-1} \nabs{\Psi}_{L^2\paren{\mc{S}_v}} \text{.}
\end{align*}
In the first two estimates, the constants of the inequalities also depend on $r$.
\end{proposition}

\begin{proof}
See \cite[Sec. 2.1, Sec. 7.2]{shao:bdc_nv}.
\end{proof}

\begin{remark}
The powers of $v$ present throughout the estimates of Proposition \ref{thm.pnc_sob} are consequences of the fact that the uniformities of this setting occur with respect to the rescaled metrics $\bar{\lambda}$ rather than the induced metric $\lambda$.
\end{remark}

In addition to the Sobolev inequalities of Proposition \ref{thm.pnc_sob}, which are estimates on the individual $\mc{S}_v$'s, we have the following first-order ``null Sobolev" inequality, which is in contrast a Sobolev-type estimate on all of $\mc{N}$.

\begin{proposition}\label{thm.pnc_nsob}
Suppose $\delta_0 \leq \mf{i}(p)$ is sufficiently small with respect to the fundamental constants.
Then, for any $\Psi \in \Gamma \ul{\mc{T}} \mc{N}$, the following estimate holds:
\[ \nabs{t_p^\frac{1}{2} \Psi}_{L^4_\omega L^\infty_t \paren{\mc{N}}} + \nabs{\Psi}_{L^6\paren{\mc{N}}} + \nabs{t_p^\frac{1}{2} \Psi}_{L^2_\omega L^\infty_t \paren{\mc{N}}} \lesssim \nabs{\Psi}_{\mc{H}^1\paren{\mc{N}}} \text{.} \]
As a consequence, the following estimates hold for any $2 \leq d \leq \infty$:
\[ \nabs{\Psi}_{L^\infty_t L^4 \paren{\mc{N}}} \lesssim \nabs{\Psi}_{\mc{H}^1\paren{\mc{N}}} \text{,} \qquad \nabs{t_p^{-\frac{1}{2} - \frac{1}{d}} \Psi}_{L^d_{t_p} L^2 \paren{\mc{N}}} \lesssim \nabs{\Psi}_{\mc{H}^1\paren{\mc{N}}} \text{.} \]
\end{proposition}

\begin{proof}
See \cite[Lemma 7.8]{shao:bdc_nv}.
\end{proof}

In Theorem \ref{thm.nc}, we have sufficiently controlled the Ricci coefficients $\chi$ and $\zeta$.
The remaining coefficients $\ul{\chi}$ and $\ul{\eta}$ can also be controlled by observing their relationships with $\chi$, $\zeta$, and the quantities $k$ and $n$ derived from the time foliation.

\begin{proposition}\label{thm.pnc_tf}
Suppose $\delta_0 \leq \mf{i}(p)$ is sufficiently small with respect to the fundamental constants.
Then, the following bounds hold on $\mc{N}$:
\[ \nabs{\ul{\eta}}_{L^\infty\paren{\mc{N}}} + \nabs{\nasla \vartheta}_{L^\infty_\omega L^2_t \paren{\mc{N}}} + \nabs{\vartheta \paren{\trace \ul{\chi}} + \frac{2}{t_p}}_{L^\infty\paren{\mc{N}}} + \nabs{\hat{\ul{\chi}}}_{L^\infty_\omega L^2_t \paren{\mc{N}}} \lesssim 1 \text{.} \]
\end{proposition}

\begin{proof}
The $L^\infty$-bound for $\ul{\eta}$ follows immediately from \eqref{eq.bdc} and \eqref{eq.eta_tfol}, and the trace bound for $\nasla \vartheta$ follows from this due to \eqref{eq.null_torsion} and \eqref{eq.pnc_main_est}.
To handle the estimates for $\ul{\chi}$, we resort to the identity \eqref{eq.chib_tfol}, from which one can derive the bounds
\[ \abs{\vartheta \paren{\trace \ul{\chi}} + \frac{2}{t_p}} \lesssim 1 + \abs{\vartheta \paren{\trace \chi} - \frac{2}{t_p}} \text{,} \qquad \abs{\hat{\ul{\chi}}} \lesssim 1 + \abs{\hat{\chi}} \]
on $\mc{N}$.
\footnote{See \cite[Cor. 4.3]{shao:bdc_nv}.}
The desired estimates for $\ul{\chi}$ now follow from \eqref{eq.pnc_main_est} and \eqref{eq.pnc_main_est_alt}.
\end{proof}

\begin{remark}
It is also possible to obtain estimates for $\ul{\eta}$, $\nasla \vartheta$, and $\ul{\chi}$ in the $\mc{H}^1$-norm analogous to those for $\chi$ and $\zeta$.
Such estimates are in fact important within the proof of Theorem \ref{thm.nc} itself.
However, we will not need these results in this paper.
\end{remark}

\begin{remark}
If $\mc{K}$ is the Gauss curvatures of the $\mc{S}_v$'s, then the quantity $\mc{K} - \vartheta^{-2} t_p^{-2}$ also satisfies estimates in both the $L^2(\mc{N})$-norm and the $L^\infty_t H^{-1/2} (\mc{N})$-norm.
The latter bound is particularly essential to the proof of Theorem \ref{thm.nc}, since it validates numerous elliptic estimates on the $\mc{S}_v$'s that play a fundamental role in the proof; see \cite[Sec. 2.3, Sec. 7.2]{shao:bdc_nv}.
Since these bounds play no role outside of the proof of Theorem \ref{thm.nc}, however, we omit serious discussions of this topic within this paper.
\end{remark}

In the remainder of this section, we will sketch the proof of Theorem \ref{thm.nc}.
For the sake of clarity, this will be divided into multiple steps.
Again, the reader is referred to \cite{shao:bdc_nv} for a detailed account of the proof.

\subsection{Proof Outline I: Bootstrap Assumptions}

The main argument is a ``bootstrap", in which we assume much of what we are trying to prove and proceed to derive even better estimates.
We can then establish via a standard continuity argument that these estimates hold even without the assumptions.

To be more specific, we define the following two conditions, which correspond to the conclusions of Theorem \ref{thm.nc}.
Let $p \in M_+$, and let $0 < \delta_0 \leq 1$ and $\Delta_0 \geq 1$ be fixed constants, whose values will be determined later.
\begin{itemize}
\item \ass{N0}{p, \delta_0, \Delta_0}: If $\delta = \min(\delta_0, t(p) - \tau_0, \mf{i}(p))$ and $\mc{N}_0 = \mc{N}^-(p; \delta)$, then
\begin{align}
\label{eq.bs_N} \Delta_0 &\geq \nabs{\vartheta \paren{\trace \chi} - \frac{2}{t_p}}_{L^\infty_\omega L^2_t \paren{\mc{N}_0}} + \nabs{\vartheta \paren{\trace \chi} - \frac{2}{t_p}}_{\mc{H}^1\paren{\mc{N}_0}} \\
\notag &\qquad + \nabs{\hat{\chi}}_{L^\infty_\omega L^2_t \paren{\mc{N}_0}} + \nabs{\hat{\chi}}_{\mc{H}^1\paren{\mc{N}_0}} + \nabs{\zeta}_{L^\infty_\omega L^2_t \paren{\mc{N}_0}} + \nabs{\zeta}_{\mc{H}^1\paren{\mc{N}_0}} \\
\notag &\qquad + \nabs{\mc{K} - \frac{1}{\vartheta^2 t_p^2}}_{L^\infty_t H^{-\frac{1}{2}} \paren{\mc{N}_0}} + \nabs{\mc{K} - \frac{1}{\vartheta^2 t_p^2}}_{L^2\paren{\mc{N}_0}} \text{.}
\end{align}

\item \ass{N1}{p, \delta_0, \Delta_0}: If $\delta = \min(\delta_0, t(p) - \tau_0)$, then $\mf{i}(p) > \delta$, and the inequality \eqref{eq.bs_N} holds for the null cone segment $\mc{N}_0 = \mc{N}^-(p; \delta)$.
\end{itemize}
Note that the \ass{N1}{p, \delta_0, \Delta_0} condition implies the \ass{N0}{p, \delta_0, \Delta_0} condition.
Furthermore, the conclusions of Theorem \ref{thm.nc} imply that the \ass{N1}{q, \delta_0, \Delta_0} condition holds over all $q \in M_+$ for some $0 < \delta_0 \leq 1$ and $\Delta_0 \geq 1$ depending only on the fundamental constants.
The \ass{N0}{p, \delta_0, \Delta_0} condition, on the other hand, corresponds to the conclusions of Theorem \ref{thm.nc} except the null injectivity radius bound.

We begin the proof of Theorem \ref{thm.nc} by assuming the following:
\begin{itemize}
\item[] {\it The conditions \ass{N0}{q, \delta_0, \Delta_0} holds for every $q \in M_+$, with the associated parameters $0 < \delta_0 \leq 1$ ``sufficiently small" and $\Delta_0 \geq 1$ ``sufficiently large", both depending only on the fundamental constants.}
\end{itemize}
This is the explicit bootstrap assumption we use for our proof.
The precise requirements for $\delta_0$ and $\Delta_0$ in our bootstrap assumption is determined within the proof.
Our goal is then to show the following estimates:
\begin{itemize}
\item[] {\it The conditions \ass{N0}{q, \delta_0, \Delta_0/2} hold for every $q \in M_+$.}
\end{itemize}
Once we establish this, we can conclude that the conditions \ass{N0}{q, \delta_0, \Delta_0} hold even without the bootstrap assumptions.
\footnote{We of course also need the initial value results of Proposition \ref{thm.pnc_init}.}
This is essentially the bootstrap argument.

To carry out this process, we require numerous \emph{auxiliary estimates}, by which we mean estimates that are consequences of the \ass{N0}{p, \delta_0, \Delta_0} condition.
This is fundamental to the proof, as many of these auxiliary estimates imply sufficient regularity on the $\mc{S}_v$'s so that the relevant analysis tools can be applied.

The primary examples of auxiliary estimates are the properties described in Section \ref{sec.pnc_cor}.
Note that these properties, in particular the bounds and comparisons of Propositions \ref{thm.pnc_vol}-\ref{thm.pnc_tf}, depend only on the \ass{N0}{p, \delta_0, \Delta_0} condition and not on the validity of Theorem \ref{thm.nc} itself.
For example, the Sobolev estimates of Propositions \ref{thm.pnc_sob} and \ref{thm.pnc_nsob} are vitally important for various basic geometric Littlewood-Paley estimates.

Another important class of auxiliary estimates are the symmetric Hodge-elliptic estimates on the $\mc{S}_v$'s.
These arguments are essentially variations of those originally introduced in \cite[Ch. 2]{chr_kl:stb_mink}.
First, we define the following operators:
\begin{itemize}
\item Given a horizontal $1$-form $\xi \in \ul{\mf{X}} (\mc{N})$, we define
\[ \mc{D}_1 \xi \in C^\infty \paren{\mc{N}} \times C^\infty\paren{\mc{N}} \text{,} \qquad \mc{D}_1 \xi = \paren{ \lambda^{ab} \nasla_a \xi_b, \epsilon^{ab} \nasla_a \xi_b } \text{,} \]
where $\epsilon$ is the associated volume forms on the $\mc{S}_v$'s.

\item Given a symmetric traceless $\xi \in \ul{T}^2 \mc{N}$, we define
\[ \mc{D}_2 \xi \in \ul{\mf{X}} \paren{\mc{N}} \text{,} \qquad \mc{D}_2 \xi_a = \lambda^{bc} \nasla_b \xi_{ca} \text{.} \]

\item Given $(\xi_1, \xi_2) \in C^\infty (\mc{N}) \times C^\infty (\mc{N})$, we define
\[ \mc{D}_1^\ast \xi \in \ul{\mf{X}} \paren{\mc{N}} \text{,} \qquad \mc{D}_1^\ast \paren{\xi_1, \xi_2}_a = - \nasla_a \xi_1 - \epsilon_a{}^b \nasla_b \xi_2 \text{.} \]
\end{itemize}
The notation $\mc{D}_1^\ast$ is justified by the fact that $\mc{D}_1^\ast$ is the $L^2$-adjoint of $\mc{D}_1$.

Here is where the $L^\infty_t H^{-1/2}$-estimate for $\mc{K} - \vartheta^{-2} t_p^{-2}$ becomes essential.
From this assumption, we can prove that $\mc{D}_1$ and $\mc{D}_2$ are one-to-one, and hence we can define via projections their $L^2$-bounded ``inverses" $\mc{D}_1^{-1}$ and $\mc{D}_2^{-1}$.
Although $\mc{D}_1^\ast$ clearly fails to be injective, we can still define a viable ``inverse" $(\mc{D}_1^\ast)^{-1}$ by ``modding out" constant functions and mapping to the canonical element with zero mean.
For detailed explanations of these procedures, consult \cite[Sec. 2.3]{shao:bdc_nv}.
The upshot of this development is the following set of auxiliary Hodge-elliptic estimates:

\begin{lemma}\label{thm.pnc_hodge}
If the \ass{N0}{p, \delta_0, \Delta_0} condition holds for sufficiently small $\delta_0$ (with respect to $\Delta_0$ and the fundamental constants), then the following estimates hold:
\begin{align*}
\nabs{\nasla \mc{D}^{-1} \xi}_{L^2 (\mc{S}_v)} + v^{-1} \nabs{\mc{D}^{-1} \xi}_{L^2 (\mc{S}_v)} \lesssim \nabs{\xi}_{L^2 (\mc{S}_v)} \text{,} \qquad \mc{D} \in \brac{\mc{D}_1, \mc{D}_2, \mc{D}_1^\ast} \text{.}
\end{align*}
\end{lemma}

\begin{proof}
See \cite[Prop. 2.28, (7.29)]{shao:bdc_nv}.
\end{proof}

From Lemma \ref{thm.pnc_hodge} and Proposition \ref{thm.pnc_sob}, we can obtain additional estimates for the inverse Hodge operators; see \cite[Prop. 2.29, (7.30)]{shao:bdc_nv} for details.

\subsection{Proof Outline II: The Sharp Trace Estimate}

The next batch of auxiliary estimates are the basic geometric Littlewood-Paley estimates and the corresponding Besov estimates on both $\mc{N}$ and the $\mc{S}_v$'s.
These were listed in \cite[Sec. 7.3]{shao:bdc_nv}, while much of the background material is within \cite[Ch. 2]{shao:bdc_nv}.

The main issue is that we are required to make use of Besov norms within the proof of Theorem \ref{thm.nc}.
Recall that in $\R^2$, although the Sobolev embedding $H^1 \hookrightarrow L^\infty$ fails, the corresponding Besov space $B^1_{2, 1}$, on the other hand, does embed into $L^\infty$.
In fact, we will need an analogous Besov embedding property on the $\mc{S}_v$'s.
This will be necessary in particular for the $L^\infty_\omega L^2_t$-type estimates for $\hat{\chi}$ and $\zeta$.

To begin with, one must make sense of Besov norms.
This task was accomplished in \cite{kl_rod:glp} via the construction of a geometric tensorial Littlewood-Paley (abbreviated \emph{L-P}) theory.
\footnote{See also \cite[Sec. 2.2]{shao:bdc_nv} for additional remarks and clarifications on this construction.}
In terms of our current setting, we can systematically define ``L-P projections" $P_k$, $k \in \mathbb{Z}$, which are used to dyadically decompose horizontal tensor fields on $\mc{N}$.
These $P_k$'s are defined explicitly using the heat flow operators intrinsic to the $\mc{S}_v$'s; see \cite[Sec. 5]{kl_rod:glp} and \cite[Sec. 2.2]{shao:bdc_nv} for the exact definitions.
\footnote{The idea of constructing Littlewood-Paley theories from the heat flow, or more generally from diffusion semigroups or martingales, originated from \cite{st:ha_lp}.}

One can proceed to show that these geometric L-P operators satisfy many of the same properties as the classical L-P projections on Euclidean spaces.

\begin{lemma}\label{thm.pnc_glp}
Let $\Psi \in \Gamma \ul{\mc{T}} \mc{N}$, $k \geq 0$, and $0 < v < \delta_0$.
\begin{itemize}
\item The following estimate holds for each $1 \leq p \leq \infty$:
\begin{equation}\label{eq.glp_bdd} \nabs{P_k \Psi}_{L^2 (\mc{S}_v)} \lesssim \nabs{\Psi}_{L^2 (\mc{S}_v)} \text{.} \end{equation}

\item The following ``finite band" estimates hold for each $1 \leq p \leq \infty$:
\begin{align}
\label{eq.glp_fbl} \nabs{\lasl P_k \Psi}_{L^p (\mc{S}_v)} = \nabs{P_k \lasl \Psi}_{L^p (\mc{S}_v)} &\lesssim 2^{2k} \nabs{\Psi}_{L^p (\mc{S}_v)} \text{,} \\
\notag \nabs{P_k \Psi}_{L^p (\mc{S}_v)} &\lesssim 2^{-2k} \nabs{\lasl \Psi}_{L^p (\mc{S}_v)} \text{.}
\end{align}

\item The following ``finite band" estimates hold:
\begin{align}
\label{eq.glp_fb} \nabs{\nasla P_k \Psi}_{L^2 (\mc{S}_v)} + \nabs{P_k \nasla \Psi}_{L^2 (\mc{S}_v)} &\lesssim 2^k \nabs{\Psi}_{L^2 (\mc{S}_v)} \text{,} \\
\notag \nabs{P_k \Psi}_{L^2 (\mc{S}_v)} &\lesssim 2^{-k} \nabs{\nasla \Psi}_{L^2 (\mc{S}_v)} \text{.}
\end{align}

\item If the \ass{N0}{p, \delta_0, \Delta_0} condition holds for sufficiently small $\delta_0$ (with respect to $\Delta_0$ and the fundamental constants), then the following ``weak Bernstein inequalities" hold for any $2 \leq p < \infty$ and $1 < q \leq 2$:
\begin{equation}\label{eq.glp_bernstein} \nabs{P_k \Psi}_{L^p (\mc{S}_v)} \lesssim 2^{ \paren{1 - \frac{2}{p}} k} \nabs{\Psi}_{L^2 (\mc{S}_v)} \text{,} \qquad \nabs{P_k \Psi}_{L^2 (\mc{S}_v)} \lesssim 2^{ \paren{\frac{2}{q} - 1} k} \nabs{\Psi}_{L^q (\mc{S}_v)} \text{.} \end{equation}
\end{itemize}
\end{lemma}

\begin{proof}
See \cite{kl_rod:glp}.
\end{proof}

In particular, similar to the classical L-P projections, the $P_k$'s convert covariant derivatives into multiplication by a scaling factor (in the weaker sense of $L^p$-norms).
Several other basic properties for the $P_k$'s can be found in \cite{kl_rod:glp, shao:bdc_nv}.

We can now define (geometric tensorial) Besov norms using the operators $P_k$.
The Besov-type norms we will need for Theorem \ref{thm.nc} are the following:
\begin{align*}
\nabs{\Psi}_{\mc{B}^0 (\mc{N})} &= \sum_{k \geq 0} \nabs{P_k \Psi}_{L^\infty_t L^2 (\mc{N})} + \nabs{\Psi}_{L^\infty_t L^2 (\mc{N})} \text{,} \\
\nabs{\Psi}_{\mc{P}^0 (\mc{N})} &= \sum_{k \geq 0} \nabs{P_k \Psi}_{L^2 (\mc{N})} + \nabs{\Psi}_{L^2 (\mc{N})} \text{.}
\end{align*}
For example, one can prove the following sharp Besov embedding estimate:
\begin{equation}\label{eq.pnc_besov_embed} \nabs{\phi}_{L^\infty (\mc{N})} \lesssim \nabs{\nasla \phi}_{\mc{B}^0 (\mc{N})} + \nabs{t_p^{-1} \phi}_{L^\infty_t L^2 (\mc{N})} \text{,} \qquad \phi \in C^\infty \paren{\mc{N}} \text{.} \end{equation}

\begin{remark}
In fact, one of the main reasons why this line of reasoning is so highly technical is because the inequality \eqref{eq.pnc_besov_embed} and other sharp $L^\infty$-$L^2$ and $L^2$-$L^1$ estimates cannot be established for nonscalar tensorial quantities.
This is due to the lack of a satisfactory tensorial ``B\"ochner estimate" controlling the second covariant derivative $\nasla^2$ in $L^2$-norms by $\lasl$.
The reason for this is that we only have a very weak $L^\infty_t H^{-1/2}$-type control on the Gauss curvatures of the $\mc{S}_v$'s.
\end{remark}

The primary tool we will require from our Besov norms is the following auxiliary estimate, which we refer to as the ``sharp trace theorem".

\begin{lemma}\label{thm.pnc_stt}
Assume the \ass{N0}{p, \delta_0, \Delta_0} condition holds for sufficiently small $\delta_0$ and sufficiently large $\Delta_0$ (with respect to the fundamental constants).
If $\Psi, P, E \in \Gamma \ul{\mc{T}} \mc{N}$, and if $\nasla \Psi = \nasla_{t_p} P + E$, then the following ``sharp trace" estimate holds:
\[ \nabs{\Psi}_{L^\infty_\omega L^2_t \paren{\mc{N}}} \lesssim \nabs{\Psi}_{\mc{H}^1\paren{\mc{N}}} + \nabs{P}_{\mc{H}^1\paren{\mc{N}}} + \nabs{E}_{\mc{P}^0\paren{\mc{N}}} \text{.} \]
\end{lemma}

\begin{proof}
See \cite[Lemma 7.22]{shao:bdc_nv} and \cite{kl_rod:stt, wang:cg}.
\footnote{In fact, the first step is an application of \eqref{eq.pnc_besov_embed}.}
\end{proof}

The goal will be to apply Lemma \ref{thm.pnc_stt} to $\hat{\chi}$ and $\zeta$.
The proof of Lemma \ref{thm.pnc_stt} requires an extensive number of Littlewood-Paley decompositions, integrations by parts, and Besov estimates.
This is a significant technical undertaking in its own right; in fact, all of \cite{kl_rod:stt} is dedicated to establishing these estimates.
The main ingredient of this proof is the ``bilinear trace theorems", in which one bounds integrals along null generators of bilinear tensorial quantities; these were stated and proved for the vacuum case in \cite[Thm. 4.9]{kl_rod:stt} and \cite[Prop. 5.4]{wang:cg}.
For the E-S and E-M cases, the proofs of these estimates remain essentially unchanged; see \cite[Lemma 7.20]{shao:bdc_nv}.

\subsection{Proof Outline III: The Null Injectivity Radius}

Another one of the most important auxiliary estimates is that of bounding the null injectivity radii from below.
Such an estimate would state roughly the following:

\begin{lemma}\label{thm.pnc_rin}
Assume the bootstrap assumptions, i.e., that \ass{N0}{q, \delta_0, \Delta_0} holds for any $q \in M_+$.
If $\delta_0$ is sufficiently small with respect to the fundamental constants, then the \ass{N1}{q, \delta_0, \Delta_0} condition also holds for every $q \in M_+$.
In particular, this implies a uniform lower bound for the null injectivity radii $\mf{i}(q)$, $q \in M_+$.
\end{lemma}

The proof of Lemma \ref{thm.pnc_rin} (in the E-S and E-M cases) is essentially the same as the vacuum case, originally presented in \cite{kl_rod:rin}.
Only a slight adaptation of this was required for our settings; this was outlined in \cite[Sec. 7.4]{shao:bdc_nv}.

For completeness, we provide a short summary of the argument behind Lemma \ref{thm.pnc_rin}.
We must control both the null conjugacy radius $\mf{s}(q)$ and the null cut locus radius $\mf{l}(q)$ of each point $q \in M_+$.
\footnote{Recall that $\mf{s}(p)$ is the timespan before which there are no past null conjugate points of $p$, while $\mf{l}(p)$ is the timespan before which no two distinct null generators of $\mc{N}^-(p)$ can intersect.}
First, we can control the null conjugacy radii by appealing to the following ``breakdown principle" for regular null cones:
\begin{itemize}
\item[] {\it A null conjugate point cannot occur along $\mc{N}^-(p)$ if $\trace \chi$ remains finite.}
\end{itemize}
This is true because uniform control on $\trace \chi$ implies control on the volume forms of the $\mc{S}_v$'s, and hence the $(v, \omega)$-parametrization of $\mc{N}^-(p)$ can be shown to be nonsingular.
As a result of this, we obtain the following:

\begin{lemma}\label{thm.pnc_rcn}
If \ass{N0}{p, \delta_0, \Delta_0} holds, then $\mf{s}(p) > \min(\mf{l}(p), \delta_0)$.
\end{lemma}

\begin{proof}
See \cite[Lemma 7.23]{shao:bdc_nv}.
\end{proof}

It remains to control the null cut locus radii $\mf{l}(q)$, $q \in M_+$.
The argument for this is identical to that found in \cite{kl_rod:rin}.
The main ingredients are the following:
\begin{itemize}
\item The regularity properties of the time foliation, established from the a priori estimates of Section \ref{sec.a_priori}.
In particular, this includes \eqref{eq.bdc} and Proposition \ref{thm.e1_est}.

\item The following null cone comparison property:
  \begin{itemize}
  \item[] {\it About any $q \in M$, there exists a sufficiently large ``almost Minkowski" coordinate system about $q$, within which the regular past null cone is contained within two Minkowski coordinate cones.}
  \end{itemize}
For details, consult \cite[Prop. 4.11, Prop. 4.12]{shao:bdc_nv}.
\end{itemize}
As a result of the above, we can make the following observations:
\begin{itemize}
\item The $\mc{S}_v$'s must be ``comparable" to rescaled Euclidean spheres.

\item If the first point of intersection between two null generators occurs before null conjugate points, then the generators must intersect with angle $\pi$.
For a proof of this property, see \cite[Lemma 3.1]{kl_rod:rin}.

\item If two null generators of $\mc{N}^-(p)$ are opposite at $p$, then they cannot intersect until after a fixed timespan $\delta_\ast$, which depends on the regularity of the time foliation, and hence on the fundamental constants; see \cite[Prop. 3.4]{kl_rod:rin}.
\end{itemize}

We can show that the above observations suffice to imply a lower bound on the $\mf{l}(q)$'s.
The idea is roughly the following.
Using \cite[Prop. 3.5]{kl_rod:rin}, we can generate a pair of null geodesic segments that intersect at two points $p_0, q_0$ such that:
\begin{itemize}
\item The timespan between $p_0$ and $q_0$ is at most that of any other pair of null geodesic segments that intersect twice.

\item The two segments connecting $p_0$ and $q_0$ are opposite at both $p_0$ and $q_0$.
\end{itemize}
Then, the preceding observations, with the help of Lemma \ref{thm.pnc_rcn}, will suffice to control this minimal timespan between $p_0$ and $q_0$.
As a result of this argument, then Lemma \ref{thm.pnc_rin} is properly established.
For the details, see \cite[Sec. 3]{kl_rod:rin}.

As a result, we obtain a lower bound on the null injectivity radii as well as the \ass{N1}{q, \delta_0, \Delta_0} conditions.
This will be important in the remainder of the proof of Theorem \ref{thm.nc}, since it is used to derive trace estimates, which bound quantities on the $\mc{S}_v$'s by quantities on the timeslices $\Sigma_\tau$; see \cite[Prop. 7.3, Cor. 7.4, Cor. 7.5]{shao:bdc_nv}.

\begin{remark}
We remark that as stated in Lemma \ref{thm.pnc_rin}, the lower bound for the null injectivity radius is contingent on the bootstrap assumptions.
However, after completing the bootstrap argument and hence proving that the \ass{N0}{q, \delta_0, \Delta_0} properties hold without preconditions, then we can apply Lemma \ref{thm.pnc_rin} again.
This controls the null injectivity radii without requiring the bootstrap assumptions a priori.
\end{remark}

\subsection{Proof Outline IV: The Improved Estimates}

With the necessary auxiliary estimates established, we now turn to deriving improved estimates for the Ricci coefficients and the Gauss curvatures $\mc{K}$.
This would complete our main bootstrap argument.
The details sketched here are performed in detail in \cite[Ch. 8]{shao:bdc_nv}.

The Ricci coefficients are related to each other and to the spacetime curvature $R$ via a family of evolutionary (along null generators) and elliptic identities, known as the \emph{structure equations}.
For the full list of structure equations, see \cite[Sec. 3.2]{shao:bdc_nv}.
These are adaptations of the analogous equations for vacuum spacetimes found in \cite{chr_kl:stb_mink, kl_rod:cg, parl:bdc, wang:cg, wang:cgp}.
In particular, in the nonvacuum settings discussed here, there are numerous additional terms in the structure equations corresponding to the Ricci curvature; of course, these terms vanish in the vacuum setting.

Our desired improved estimates will be a consequence of the forms of these structure equations.
Let $\Gamma$ denote one of the quantities bounded by the \ass{N0}{p, \delta_0, \Delta_0} condition, i.e., $\Gamma \leq \Delta_0$.
Our goal, in general, will be to integrate the various structure equations in order to derive estimates roughly of the form
\[ \Gamma \lesssim \delta_0^\frac{1}{2} \Delta_0^2 + 1 \text{,} \]
where the constant of the inequality depends only on the fundamental constants.
In general, the bound $\delta_0^{1/2} \Delta_0^2$ is a consequence of terms in the structure equations that are quadratic (and occasionally cubic) in the Ricci coefficients, while ``$1$" reflects those terms which can be controlled using a priori estimates.

If we choose $\Delta_0$ to be sufficiently large with respect to the fundamental constants, and we then choose $\delta_0$ to be sufficiently small with respect to $\Delta_0$ and the fundamental constants, then we obtain the schematic bound
\[ \Gamma \leq C \paren{\delta_0^\frac{1}{2} \Delta_0^2 + 1} \leq \frac{\Delta_0}{2} \text{.} \]
In particular, this means we have proved the \ass{N0}{p, \delta_0, \Delta_0/2} condition.
As a result, the bootstrap argument would be complete.
We will now briefly describe how the above schematic bound can be obtained for the possible values of $\Gamma$.

For example, consider the following structure equation:
\begin{equation}\label{eq.raych} \nasla_{t_p} \paren{\trace \chi} = -\frac{1}{2} \vartheta \paren{\trace \chi}^2 - \vartheta \abs{\hat{\chi}}^2 - \vartheta R_{44} \text{.} \end{equation}
This is the \emph{Raychaudhuri equation}, which is perhaps the most well-known of the structure equations.
By integrating a slight variant of the Raychaudhuri equation and taking an $L^\infty_\omega L^2_t$-norm of the result, we derive
\[ \nabs{\vartheta \paren{\trace \chi} - \frac{2}{t_p}}_{L^\infty_\omega L^2_t (\mc{N})} \lesssim \delta_0^\frac{1}{2} \Delta_0^2 + 1 \text{.} \]
The quantity $\Delta_0^2$ follows from the bootstrap assumption $\Gamma \leq \Delta_0$, while the small factor $\delta_0^{1/2}$ appears because of the integrations involved.
The Ricci curvature term in \eqref{eq.raych} can be controlled using a priori estimates from the previous section.

We can also control the quantity $\nasla_{t_p} [\vartheta (\trace \chi) - 2 t_p^{-1}]$ in the $L^2(\mc{N})$-norm using \eqref{eq.raych}.
Similar bounds for the $L^2$-norms of $\nasla_{t_p} \hat{\chi}$ and $\nasla_{t_p} \zeta$ can be derived using the corresponding evolution equations for $\hat{\chi}$ and $\zeta$.

At this point, we can also prove analogous estimates for the quantities in the left-hand side of \eqref{eq.pnc_main_est_alt} by integrating the appropriate evolutionary structure equations.
These estimates will in fact function like the auxiliary estimates:
\begin{itemize}
\item These bounds will be essential for completing the bootstrap argument.

\item Once the bootstrap argument is completed, and \eqref{eq.pnc_main_est} is established, then \eqref{eq.pnc_main_est_alt} will immediately follow from \eqref{eq.pnc_main_est} and these estimates.
\end{itemize}

\begin{remark}
The evolutionary structure equations are listed in \cite[Prop. 3.8]{shao:bdc_nv}.
\end{remark}

Next, we turn to the elliptic structure equations, cf. \cite[Prop. 3.9]{shao:bdc_nv}, from which we can control $\nasla \hat{\chi}$ and $\nasla \zeta$ in the $L^2$-norms.
These equations are of the forms
\[ \mc{D}_2 \hat{\chi} = \check{R} + \nasla (\trace \chi) + B \text{,} \qquad \mc{D}_1 \zeta = \check{R} + \mu + B \text{,} \]
where $\check{R}$ represents certain components of the spacetime Riemann and Ricci curvatures, and $B$ represents lower-order terms.
Since the right-hand sides of these elliptic equations can be easily controlled in the $L^2(\mc{N})$-norm, then direct applications of Lemma \ref{thm.pnc_hodge} yield the desired improved estimates for $\nasla \hat{\chi}$ and $\nasla \zeta$.
These are once again of the familiar form $\Gamma \lesssim \delta_0^{1/2} \Delta_0^2 + 1$.

Furthermore, using the structure equation
\[ \mc{K} = -\frac{1}{4} \paren{\trace \chi} \paren{\trace \ul{\chi}} + \frac{1}{2} \hat{\chi}^{ab} \hat{\ul{\chi}}_{ab} - \frac{1}{4} R_{4343} + \frac{1}{2} R_a{}^a + \frac{1}{2} R_{43} \text{,} \]
we can proceed to derive the following improved estimates:
\[ \nabs{\mc{K} - \frac{1}{\vartheta^2 t_p^2}}_{L^\infty_t H^{-\frac{1}{2}} (\mc{N})} \lesssim \delta_0^\frac{1}{2} \Delta_0^2 + 1 \text{,} \qquad \nabs{\mc{K} - \frac{1}{\vartheta^2 t_p^2}}_{L^2 (\mc{N})} \lesssim \delta_0^\frac{1}{2} \Delta_0^2 + 1 \text{.} \]
The latter bound is straightforward, but the former requires much more effort due to the presence of the fractional Sobolev $H^{-1/2}$-norm.
\footnote{See \cite{kl_rod:glp} and \cite[Sec. 2.2]{shao:bdc_nv} for the definition of the $H^s$-norms.}
Indeed, for the $L^\infty_t H^{-1/2}$-bound, we need a technical estimate for the commutator $[\Lambda^{-1/2}, \nasla_{t_p}]$, where $\Lambda^{-1/2}$ is the fractional Sobolev operator on the $\mc{S}_v$'s, defined in the standard fashion in terms of $\lasl$.
For details, see \cite[Lemma 8.8]{shao:bdc_nv}, as well as \cite{kl_rod:cg, wang:cg}.

It remains only to establish improved estimates for the $L^\infty_\omega L^2_t$-norms of $\hat{\chi}$ and $\zeta$.
As mentioned earlier, the strategy is to utilize Lemma \ref{thm.pnc_stt}.
In particular, we must show that $\nasla \hat{\chi}$ and $\nasla \zeta$ have decompositions of the form $\nasla_{t_p} P + E$, where the quantities $P$ and $E$ have sufficient bounds, as specified in Lemma \ref{thm.pnc_stt}.
This is a priori not obvious; the argument in fact spans nearly all of \cite[Sec. 8.2-8.4]{shao:bdc_nv}.

The main ideas and steps for deriving these decompositions are the following:
\begin{itemize}
\item First, since the structure equations only provide relations for $\mc{D}_2 \hat{\chi}$ and $\mc{D}_1 \zeta$, not for $\nasla \hat{\chi}$ and $\nasla \zeta$, we must apply the inverse Hodge operators, e.g.,
\[ \nasla \hat{\chi} = \nasla \mc{D}_2^{-1} \mc{D}_2 \hat{\chi} = \nasla \mc{D}_2^{-1} \check{R} + \nasla \mc{D}_2^{-1} \nasla \paren{\trace \chi} + \nasla \mc{D}_2^{-1} B \text{.} \]

\item In order to generate the components $\nasla_{t_p} P$ in the desired decomposition, we must take advantage of additional relations satisfied by $\check{R}$.
For instance, for $\zeta$, this is a ``null Bianchi equation" of the form
\[ \mc{D}_1^\ast \check{R} = \nasla_{t_p} \check{R} + D \text{,} \]
where $D$ denotes lower-order terms.
From this, we see that we will also need to work with the ``inverse" operator $(\mc{D}_1^\ast)^{-1}$.

\item Next, combining the previous two steps, we see that we must commute $\nasla_{t_p}$ with $\nasla$ and the inverse Hodge operators.
In particular, this requires multiple commutator estimates in both Besov and lower-order norms.

\item Unfortunately, there exist ``bad" commutator terms which cannot be adequately controlled.
In fact, these terms must themselves be further decomposed by adopting arguments analogous to the above, which yields an additional level of ``good" and ``bad" terms.
\footnote{The ``good" terms will be contribute to both $\nasla_{t_p} P$ and $E$.}
This ultimately results in an infinite sequence of decompositions, which can be shown to converge.
Only in this fashion can the ``bad" commutator terms be treated.

\item After the infinite sequence of renormalizations, we finally obtain the desired decompositions for $\nasla \hat{\chi}$ and $\nasla \zeta$.
We must also show that components of these decompositions can be controlled as required by Lemma \ref{thm.pnc_stt}.
\end{itemize}
Once all of the above is complete, then Lemma \ref{thm.pnc_stt} yields the desired improved estimates for both $\hat{\chi}$ and $\zeta$ in the $L^\infty_\omega L^2_t$-norm.

\begin{remark}
For the most part, the Ricci curvature terms in the structure equations (which are unique to the nonvacuum settings) are harmless, as they generally have better bounds than the Riemann curvature terms due to \eqref{eq.bdc_lemma} and Proposition \ref{thm.e1_lest}.
The main exception to this is for the above decomposition argument, for which the higher-order Ricci curvature terms require their own unique decompositions.
In particular, for the E-M case, we must make use of analogous null Bianchi equations for components of $F$.
For details, see \cite[Sec. 8.3]{shao:bdc_nv}.
\end{remark}

As a result of all of the above, we obtain the improved estimates $\Gamma \leq \Delta_0/2$, which proves the \ass{N0}{q, \delta_0, \Delta_0/2} condition for all $q \in M_+$ and hence completes the bootstrap process.
By a standard continuity argument, we obtain that the \ass{N0}{q, \delta_0, \Delta_0} condition holds for all $q \in M_+$ without the a priori bootstrap assumption.

Finally, by Lemma \ref{thm.pnc_rin}, the conditions \ass{N1}{q, \delta_0, \Delta_0} also hold for all $q \in M_+$ without prior assumptions.
This implies the desired past null injectivity radius bounds, along with \eqref{eq.pnc_main_est}.
The remaining estimates \eqref{eq.pnc_main_est_alt} now follow from \eqref{eq.pnc_main_est}, as discussed before.
This completes the proof of Theorem \ref{thm.nc}.

\section{Higher-Order Energy Estimates}

Now that we have established both a priori estimates and local null geometry estimates, we can proceed to derive higher-order energy estimates and then complete the proof of Theorem \ref{thm.bdc}.
Our main goal is higher-order $L^2$-estimates on $k$, $\mc{R}$, and the horizontal formulation of the matter field; these are the objects represented by the quantities $\mf{K}(\tau)$, $\mf{R}(\tau)$, $\mf{f}(\tau)$ in \eqref{eq.bdc_energy_KR}-\eqref{eq.bdc_energy_fem}.
Like for the a priori estimates, we obtain these by first deriving higher-order energy inequalities for the spacetime quantities $R$ and the matter field $\Phi$, and then deriving elliptic estimates.

Unlike for the a priori estimates, we will also require $L^\infty$-bounds for $R$ and the matter field ($D^2 \phi$ in the E-S case, $D F$ in the E-M case) in order to complete the higher-order energy estimates.
For this task, we will make use of the generalized representation formula for covariant tensorial wave equations discussed in \cite{shao:bdc_nv, shao:ksp}.
This is an extension of the formula of S. Klainerman and I. Rodnianski in \cite{kl_rod:ksp}, which was applied in an analogous fashion in the E-V case.
However, the formula of \cite{kl_rod:ksp} fails in the E-M case due to additional first-order terms.

\subsection{Covariant Wave Relations}

We have previously made use of the fact that the matter field ($D \phi$ in the E-S case, and $F$ in the E-M case) satisfies a covariant wave equation.
We now list similar wave equations for $R$, $D R$, and higher derivatives of the matter field.
We can then derive energy inequalities from these relations using the same EMT techniques as in Propositions \ref{thm.e1_est} and \ref{thm.e1_lest}.

To begin with, by taking the divergence of the Bianchi identities \eqref{eq.bianchi} for $R$ and then commuting derivatives, we obtain the covariant wave equation
\[ \square_g R_{\alpha\beta\gamma\delta} = - D_\alpha{}^\mu R_{\beta\mu\gamma\delta} - D_\beta{}^\mu R_{\mu\alpha\gamma\delta} + R_\alpha{}^\mu \brak{R}_{\beta\mu\gamma\delta} + R_\beta{}^\mu \brak{R}_{\mu\alpha\gamma\delta} \text{.} \]
Applying either \eqref{eq.curv_divg_es} or \eqref{eq.curv_divg_em}, depending on the matter model at hand, and expanding the terms quadratic in $R$, we obtain the nonlinear wave equation
\begin{align}
\label{eq.curv_wave} \square_g R &= \tilde{d}^2 \ric + R \ast R \text{,} \\
\notag \paren{R \ast R}_{\alpha\beta\gamma\delta} &= \mf{A}_{\alpha\beta} \brac{R_\alpha{}^\lambda R_{\lambda\beta\gamma\delta}} + 2 R_\alpha{}^\lambda{}_\beta{}^\mu R_{\mu\lambda\gamma\delta} - 2 \mf{A}_{\alpha\beta} \brac{R_\alpha{}^\mu{}_\gamma{}^\lambda R_{\beta\mu\delta\lambda}} \text{.}
\end{align}
The term $\tilde{d}^2 \ric$ has the following expansions in the E-S and E-M cases, respectively:
\begin{align*}
\paren{\tilde{d}^2 \ric}_{\alpha\beta\gamma\delta} &= -2 \mf{A}_{\alpha\beta} \brac{D_{\alpha\gamma} \phi D_{\beta\delta} \phi} - \mf{A}_{\alpha\beta} \mf{A}_{\gamma\delta} \brac{D_\lambda \phi D_\gamma \phi R_{\alpha\delta\beta}{}^\lambda} \text{,} \\
\paren{\tilde{d}^2 \ric}_{\alpha\beta\gamma\delta} &= \mf{A}_{\alpha\beta} \brac{F_\alpha{}^\mu D_{\beta\mu} F_{\gamma\delta}} + \mf{A}_{\gamma\delta} \brac{F_\gamma{}^\mu D_{\delta\mu} F_{\alpha\beta}} - \mf{A}_{\alpha\beta} \mf{A}_{\gamma\delta} \brac{D_\alpha F_{\gamma\mu} D_\delta F_\beta{}^\mu} \\
\notag &\qquad + D^\mu F_{\alpha\beta} D_\mu F_{\gamma\delta} + \mf{A}_{\alpha\beta} \mf{A}_{\gamma\delta} \brac{F^{\lambda\mu} F_{\gamma\mu} R_{\alpha\lambda\beta\delta} + F_{\gamma \mu} F_{\alpha\lambda} R_{\beta\delta}{}^{\mu\lambda}} \\
\notag &\qquad - \frac{1}{2} \mf{A}_{\alpha\beta} \mf{A}_{\gamma\delta} \brac{g_{\alpha\gamma} \paren{F^{\mu\lambda} D_{\beta\delta} F_{\mu\lambda} + D_\beta F_{\mu\lambda} D_\delta F^{\mu\lambda}}} \text{,}
\end{align*}
Recall that $\mf{A}$, defined in \eqref{eq.index_antisym}, denotes anti-symmetrizations of the indices.

In the E-S case, by differentiating \eqref{eq.es_wave0} and commuting derivatives, we can derive a wave equation for $D^2 \phi$.
More specifically, we can compute
\begin{align*}
\square_g D_{\alpha\beta} \phi &= D^\mu{}_{\alpha\mu\beta} \phi + D^\mu \paren{R_\beta{}^\lambda{}_{\mu\alpha} D_\lambda \phi} \\
&= D_\alpha \square_g D_\beta \phi + R_\mu{}^{\lambda\mu}{}_\alpha D_{\lambda\beta} \phi + R_\beta{}^{\lambda\mu}{}_\alpha D_{\mu\lambda} \phi \\
&\qquad + D^\mu R_{\mu\alpha\beta}{}^\lambda D_\lambda \phi + R_{\mu\alpha\beta}{}^\lambda D^\mu{}_\lambda \phi \text{.}
\end{align*}
By \eqref{eq.curv_divg_es} and \eqref{eq.es_wave0}, then
\begin{align}
\label{eq.es_wave} \square_g D_{\alpha\beta} \phi &= 2 D^\lambda \phi D_\lambda \phi D_{\alpha\beta} \phi + D_\alpha \phi D^\lambda \phi D_{\beta\lambda} \phi \\
\notag &\qquad + D_\beta \phi D^\lambda \phi D_{\alpha\lambda} \phi - 2 R_\alpha{}^\mu{}_\beta{}^\lambda D_{\mu\lambda} \phi \text{.}
\end{align}
By a similar process, in the E-M case, we have the wave relation
\begin{align}
\label{eq.em_wave} \square_g D_\gamma F_{\alpha\beta} &= 2 F^{\lambda\mu} D_\gamma R_{\alpha\lambda\mu\beta} + \mf{A}_{\alpha\beta} \brac{R_\alpha{}^\lambda{}_\gamma{}^\mu D_\mu F_{\beta\lambda}} + 2 R_\alpha{}^{\lambda\mu}{}_\beta D_\gamma F_{\lambda\mu} \\
\notag &\qquad + 2 F_\alpha{}^\mu F_\beta{}^\lambda D_\gamma F_{\mu\lambda} + 2 \mf{A}_{\alpha\beta} \brac{F^{\mu\lambda} F_{\alpha\mu} D_\gamma F_{\beta\lambda}} \\
\notag &\qquad - F^{\mu\lambda} F_{\alpha\beta} D_\gamma F_{\mu\lambda} - \frac{3}{4} F^{\mu\lambda} F_{\mu\lambda} D_\gamma F_{\alpha\beta} \\
\notag &\qquad - F^{\lambda\mu} F_{\gamma\mu} D_\lambda F_{\alpha\beta} - \mf{A}_{\alpha\beta} \brac{ F_\alpha{}^\lambda F_\gamma{}^\mu D_\mu F_{\beta\lambda} } \\
\notag &\qquad - \mf{A}_{\alpha\beta} \brak{F_\alpha{}^\lambda \mf{A}_{\beta\lambda} \paren{F_\beta{}^\mu D_\lambda F_{\gamma\mu} - \frac{1}{2} g_{\beta\gamma} F^{\mu\nu} D_\lambda F_{\mu\nu}}} \text{.}
\end{align}

In addition to the above, we will need covariant wave equations for one higher derivative of $R$, $\phi$, and $F$.
Fortunately, we will require only the schematic forms, not the exact equations.
In the E-S case, we have
\begin{align}
\label{eq.curv_wave1_es} \square_g D R &\cong D^2 \phi \cdot D^3 \phi + D \phi \cdot D^2 \phi \cdot R + \paren{D \phi}^2 \cdot D R + R \cdot D R \text{,} \\
\label{eq.es_wave1} \square_g D^3 \phi &\cong \paren{D \phi}^2 \cdot D^3 \phi + D \phi \cdot \paren{D^2 \phi}^2 + R \cdot D^3 \phi + DR \cdot D^2 \phi \text{.}
\end{align}
Similarly, in the E-M case,
\begin{align}
\label{eq.curv_wave1_em} \square_g D R &\cong F \cdot D^3 F + D F \cdot D^2 F + F \cdot D F \cdot R + F^2 \cdot D R + R \cdot D R \text{,} \\
\label{eq.em_wave1} \square_g D^2 F &\cong F \cdot D^2 R + D F \cdot D R + R \cdot D^2 F + F \cdot \paren{D F}^2 + F^2 \cdot D^2 F \text{.}
\end{align}

\subsection{Energy Inequalities}

With the wave relations \eqref{eq.curv_wave}-\eqref{eq.em_wave1} in hand, we can now apply general EMT methods to derive preliminary energy inequalities for the curvature and the matter field.
Like for the a priori estimates of Proposition \ref{thm.e1_est}, we must bound the curvature and the matter field concurrently.

For convenience, in the E-S case, we define the quantities
\[ \mc{E}^2\paren{\tau} = \nabs{D R}_{L^2\paren{\Sigma_\tau}}^2 + \nabs{D^3 \phi}_{L^2\paren{\Sigma_\tau}}^2 \text{,} \qquad \mc{E}^3\paren{\tau} = \nabs{D^2 R}_{L^2\paren{\Sigma_\tau}}^2 + \nabs{D^4 \phi}_{L^2\paren{\Sigma_\tau}}^2 \text{,} \]
while in the E-M case, we define the analogous quantities
\[ \mc{E}^2\paren{\tau} = \nabs{D R}_{L^2\paren{\Sigma_\tau}}^2 + \nabs{D^2 F}_{L^2\paren{\Sigma_\tau}}^2 \text{,} \qquad \mc{E}^3\paren{\tau} = \nabs{D^2 R}_{L^2\paren{\Sigma_\tau}}^2 + \nabs{D^3 F}_{L^2\paren{\Sigma_\tau}}^2 \text{.} \]

\begin{proposition}\label{thm.e2_est}
For any $\tau_0 \leq \tau_1 < \tau < t_1$, the energy inequalities
\begin{align*}
\mc{E}^2\paren{\tau} &\lesssim \mc{E}^2\paren{\tau_1} + \int_{\tau_1}^\tau \paren{1 + \nabs{R}_{L^\infty\paren{\Sigma_{\tau^\prime}}}^2 + \nabs{D^2 \phi}_{L^\infty\paren{\Sigma_{\tau^\prime}}}^2} d\tau^\prime \text{,} \\
\mc{E}^2\paren{\tau} &\lesssim \mc{E}^2\paren{\tau_1} + \int_{\tau_1}^\tau \paren{1 + \nabs{R}_{L^\infty\paren{\Sigma_{\tau^\prime}}}^2 + \nabs{D F}_{L^\infty\paren{\Sigma_{\tau^\prime}}}^2} d\tau^\prime
\end{align*}
hold in the E-S and E-M cases, respectively.
\end{proposition}

\begin{proof}
In the E-S case, applying \eqref{eq.emt_wave_est_pre} with $\Phi = R$ along with \eqref{eq.curv_wave} yields
\[ \nabs{D R}_{L^2\paren{\Sigma_\tau}}^2 \lesssim \nabs{D R}_{L^2\paren{\Sigma_{\tau_1}}}^2 + \int_{\tau_1}^\tau \int_{\Sigma_{\tau^\prime}} \paren{\abs{D R}^2 + \abs{R}^4 + \abs{D^2 \phi}^4 + \abs{R}^2} d\tau^\prime \text{.} \]
Similarly, applying \eqref{eq.emt_wave_est_pre} with $\Phi = D^2 \phi$ along with \eqref{eq.es_wave} yields
\[ \nabs{D^3 \phi}_{L^2\paren{\Sigma_\tau}}^2 \lesssim \nabs{D^3 \phi}_{L^2\paren{\Sigma_{\tau_1}}}^2 + \int_{\tau_1}^\tau \int_{\Sigma_{\tau^\prime}} \paren{\abs{D^3 \phi}^2 + \abs{R}^2 \abs{D^2 \phi}^2 + \abs{D^2 \phi}^2} d\tau^\prime \text{.} \]
In these inequalities, we have used the a priori uniform bounds on $k$, $n$, and $D \phi$.
Summing the above and applying Proposition \ref{thm.e1_est} yields
\[ \mc{E}^2\paren{\tau} \lesssim \mc{E}^2\paren{\tau_1} + \int_{\tau_1}^\tau \paren{1 + \nabs{R}_{L^\infty\paren{\Sigma_{\tau^\prime}}}^2 + \nabs{D^2 \phi}_{L^\infty\paren{\Sigma_{\tau^\prime}}}^2} d\tau^\prime + \int_{\tau_1}^\tau \mc{E}^2\paren{\tau^\prime} d\tau^\prime \text{.} \]
Applying Gr\"onwall's inequality results in the desired bound.

The E-M case proceeds analogously; apply \eqref{eq.emt_wave_est_pre} with $\Phi = R$ and $\Phi = D F$.
\end{proof}

The process behind the proof of Proposition \ref{thm.e2_est} can be repeated for the wave equations \eqref{eq.curv_wave1_es}-\eqref{eq.em_wave1} in order to obtain the following:

\begin{proposition}\label{thm.e3_est}
For any $\tau_0 \leq \tau_1 < \tau < t_1$, the energy inequalities
\begin{align*}
\mc{E}^3\paren{\tau} &\lesssim \mc{E}^3\paren{\tau_1} + \int_{\tau_1}^\tau \paren{1 + \nabs{R}_{L^\infty\paren{\Sigma_{\tau^\prime}}}^2 + \nabs{D^2 \phi}_{L^\infty\paren{\Sigma_{\tau^\prime}}}^2} \brak{1 + \mc{E}^2\paren{\tau^\prime}} d\tau^\prime \text{,} \\
\mc{E}^3\paren{\tau} &\lesssim \mc{E}^3\paren{\tau_1} + \int_{\tau_1}^\tau \paren{1 + \nabs{R}_{L^\infty\paren{\Sigma_{\tau^\prime}}}^2 + \nabs{D F}_{L^\infty\paren{\Sigma_{\tau^\prime}}}^2} \brak{1 + \mc{E}^2\paren{\tau^\prime}} d\tau^\prime
\end{align*}
hold in the E-S and E-M cases, respectively.
\end{proposition}

\begin{proof}
In the E-S case, apply \eqref{eq.emt_wave_est_pre}, \eqref{eq.curv_wave1_es}, and \eqref{eq.es_wave1} with $\Phi = D R$ and $\Phi = D^3 \phi$.
In the E-M case, apply \eqref{eq.emt_wave_est_pre}, \eqref{eq.curv_wave1_em}, and \eqref{eq.em_wave1} with $\Phi = D R$ and $\Phi = D^2 F$.
\end{proof}

\begin{remark}
We can obtain identical energy bounds for $D R$ and $D^2 R$, as well as $D F$, $D^2 F$, and $D^3 F$ in the E-M case, using the generalized Maxwell EMTs and \eqref{eq.emt_maxwell_est_pre}.
\end{remark}

Lastly, we will need a local variant of Proposition \ref{thm.e2_est} involving flux bounds on regular past null cones.
This can be derived in the same manner that Proposition \ref{thm.e1_lest} was obtained from Proposition \ref{thm.e1_est}.
Fix $p \in M_+$ and a time value $\tau_0 \leq \tau < t(p)$, with $t(p) - \tau < \mf{i}(p)$; we normalize and foliate $\mc{N}^-(p)$ using $T|_p$ and $t_p$.

Define the sets $\mc{N}_\tau$, $\mc{I}_\tau$, and $\Sigma^p_\tau$ as in Section \ref{sec.lest}, and define the flux densities
\begin{align*}
\rho^2\paren{p, \tau} &= - Q_w\brak{R}\paren{T, L} - Q_w\brak{D^2 \phi}\paren{T, L} \text{,} \\
\rho^2\paren{p, \tau} &= - Q_w\brak{R}\paren{T, L} - Q_w\brak{D F}\paren{T, L} \text{,}
\end{align*}
in the E-S and E-M settings, respectively.
We also define the corresponding flux
\[ \mc{F}^2\paren{p; \tau} = \int_{\mc{N}_\tau} \rho^2\paren{p, \tau} \text{.} \]
Recall that the notation $Q_w[U]$ was defined in Section \ref{sec.emt}.

\begin{proposition}\label{thm.e2_lest}
The following hold in the E-S and E-M settings, respectively:
\begin{align*}
\mc{F}^2\paren{p; \tau} &\lesssim \mc{E}^2\paren{\tau} + \int_\tau^{t\paren{p}} \paren{1 + \nabs{R}_{L^\infty\paren{\Sigma_{\tau^\prime}}}^2 + \nabs{D^2 \phi}_{L^\infty\paren{\Sigma_{\tau^\prime}}}^2} d\tau^\prime \text{,} \\
\mc{F}^2\paren{p; \tau} &\lesssim \mc{E}^2\paren{\tau} + \int_\tau^{t\paren{p}} \paren{1 + \nabs{R}_{L^\infty\paren{\Sigma_{\tau^\prime}}}^2 + \nabs{D F}_{L^\infty\paren{\Sigma_{\tau^\prime}}}^2} d\tau^\prime \text{.}
\end{align*}
\end{proposition}

\begin{proof}
For the E-S case, we apply \eqref{eq.emt_wave_lest_pre} with $\Phi = R$ and $\Phi = D^2 \phi$ to obtain
\begin{align*}
- \int_{\mc{N}_\tau} Q_w\brak{R}\paren{T, L} &\lesssim \nabs{D R}_{L^2\paren{\Sigma_\tau}}^2 \\
&\qquad + \int_\tau^{t\paren{p}} \int_{\Sigma_{\tau^\prime}} \paren{\abs{D R}^2 + \abs{R}^4 + \abs{D^2 \phi}^4 + \abs{R}^2} d\tau^\prime \text{,} \\
- \int_{\mc{N}_\tau} Q_w\brak{D^2 \phi}\paren{T, L} &\lesssim \nabs{D^3 \phi}_{L^2\paren{\Sigma_\tau}}^2 \\
&\qquad + \int_\tau^{t\paren{p}} \int_{\Sigma_{\tau^\prime}} \paren{\abs{D^3 \phi}^2 + \abs{R}^2 \abs{D^2 \phi}^2 + \abs{D^2 \phi}^2} d\tau^\prime \text{.}
\end{align*}
Summing the above bounds and applying Proposition \ref{thm.e1_est}, we obtain
\[ \mc{F}^2\paren{p; \tau} \lesssim \mc{E}^2\paren{\tau} + \int_\tau^{t\paren{p}} \brak{1 + \nabs{R}_{L^\infty\paren{\Sigma_{\tau^\prime}}}^2 + \nabs{D^2 \phi}_{L^\infty\paren{\Sigma_{\tau^\prime}}}^2 + \mc{E}^2\paren{\tau^\prime}} d\tau^\prime \text{.} \]
The proof is completed by applying Proposition \ref{thm.e2_est}.

The E-M case is analogous, except we apply \eqref{eq.emt_wave_lest_pre} with $\Phi = R$ and $\Phi = D F$.
\end{proof}

\subsection{The Generalized Kirchhoff-Sobolev Parametrix} 

From Propositions \ref{thm.e2_est} and \ref{thm.e3_est}, we see that in order to control the higher-order energy quantities $\mc{E}^2(\tau)$ and $\mc{E}^3(\tau)$, uniformly for $\tau_0 \leq \tau < t_1$, we must first obtain a bound for the $L^\infty$-norms of $R$ and the matter field: $D^2 \phi$ in the E-S case, and $D F$ in the E-M case.
As mentioned before, we will accomplish this by using the wave equations \eqref{eq.curv_wave}, \eqref{eq.es_wave}, and \eqref{eq.em_wave}, along with the generalized Kirchhoff-Sobolev parametrix of \cite{shao:ksp}.

In the vacuum analogue of \cite{kl_rod:bdc}, one needed an $L^\infty$-bound for $R$, which was obtained using the Kirchhoff-Sobolev parametrix of \cite{kl_rod:ksp}.
Although this parametrix also suffices in the E-S setting, it fails in the E-M case due to first-order terms present in \eqref{eq.curv_wave} and \eqref{eq.em_wave}.
Such terms must be handled differently than allowed in \cite{kl_rod:ksp}; this is the primary motivation behind the generalized formula of \cite{shao:ksp}.

Consider the scalar wave equation in the Minkowski spacetime $\R^{1+3}$:
\begin{equation}\label{eq.wave_eq_mink} \square \phi = \psi \text{,} \qquad \valat{\phi}_{t = 0} = \alpha_0 \text{,} \qquad \valat{\partial_t \phi}_{t = 0} = \alpha_1 \text{.} \end{equation}
From standard theory, cf. \cite[Sec. 2.4]{ev:pde}, at a point $(t, x) \in (0, \infty) \times \R^3$, we can express $\phi(t, x)$ explicitly in terms of $\psi$, $\alpha_0$, and $\alpha_1$.
Moreover, the above can be written as an integral along the past null cone segment $N$ from the vertex $(t, x)$ to the timeslice $t = 0$, along with an ``initial data" integral on the set $N \cap \{t = 0\}$.

The goals of both \cite{shao:ksp} and its predecessor \cite{kl_rod:ksp} are to provide a local first-order extension of the standard formula to arbitrary curved spacetimes.
In particular, the parametrices of \cite{kl_rod:ksp, shao:ksp} enjoy the following features:
\begin{itemize}
\item These formulas treat tensorial wave equations in a completely covariant fashion, without first expressing them as scalar equations.

\item The parametrices are supported entirely on past null cones.

\item The parametrices are valid only on regular (i.e., smooth) portions of past null cones.
In the setting of this paper, this means that the formulas are only applicable up to the null injectivity radius.

\item The parametrices contain ``error terms", expressed as integrals along regular past null cones.
These are a result of the nontrivial geometry.

\item The parametrices depend only on quantities defined on the past null cones, i.e., it is independent of the extensions of such quantities off the cones.

\item The parametrices can be systematically generalized to covariant wave equations on arbitrary vector bundles over an arbitrary curved spacetime.
\footnote{In this extended setting, one also needs a bundle metric and a compatible connection.}
\end{itemize}

The parametrix of \cite{shao:ksp} differs from that of \cite{kl_rod:ksp} in the following respects:
\begin{itemize}
\item The formula of \cite{shao:ksp} directly handles systems of tensor wave equations with additional first-order terms.
The first-order terms are treated by altering the transport equation associated with the parametrix.

\item While the proofs in \cite{kl_rod:ksp} make heavy use of distributions on manifolds in an informal fashion, the proofs in \cite{shao:ksp} instead remain entirely at the level of rigorous calculus operations on null cones.

\item The computations in \cite{kl_rod:ksp} generated terms not supported on the null cone, which must all be meticulously cancelled in later steps.
In \cite{shao:ksp}, on the other hand, the entire derivation was performed on the null cone.
More specifically, the extra terms in \cite{kl_rod:ksp} could be avoided altogether if one integrates by parts only the derivatives tangential to the null cone.

\item As a consequence of the above modifications, one can slightly weaken the assumptions listed in \cite{kl_rod:ksp} for the parametrix to be valid.
\footnote{In fact, this weakening of the assumptions applies to both the formulas of \cite{kl_rod:ksp} and \cite{shao:ksp}, since the results in \cite{kl_rod:ksp} are strictly special cases of those in \cite{shao:ksp}.}
\end{itemize}
For more details on the preceding discussions, see \cite{shao:ksp}, or \cite[Ch. 5]{shao:bdc_nv}.

We now state a special case of the main result of \cite{shao:ksp}, adapted to the setting of this paper.
More specifically, we state the representation formula only in the case of time foliations of the null cone, while \cite{shao:ksp} also considered arbitrary foliating functions.
For the general version, see \cite[Thm. 7]{shao:ksp}.

\begin{theorem}\label{thm.ksp_gen}
Assume the following:
\begin{itemize}
\item Let $n$ be a positive integer, let $\sss{r}{1}, \ldots, \sss{r}{n}$ be nonnegative integers, and suppose for each $1 \leq m, c \leq n$, we have defined tensor fields
\[ \sss{\Phi}{m} \in \Gamma T^{\sss{r}{m}} M \text{,} \qquad \sss{\Psi}{m} \in \Gamma T^{\sss{r}{m}} M \text{,} \qquad \sss{P}{mc} \in \Gamma T^{1 + \sss{r}{m} + \sss{r}{c}} M \text{.} \]

\item Suppose the $\sss{\Phi}{m}$'s, $\sss{\Psi}{m}$'s, and $\sss{P}{mc}$'s satisfy the system
\begin{equation}\label{eq.tensor_wave_system} \Box_g \sss{\Phi}{m}_I + \sum_{c = 1}^n \sss{P}{mc}_{\mu I}{}^J D^\mu \sss{\Phi}{c}_J = \sss{\Psi}{m}_I \text{,} \qquad 1 \leq m \leq n \end{equation}
of tensor wave equations, where $I$ and $J$ in \eqref{eq.tensor_wave_system} are collections of $\sss{r}{m}$ and $\sss{r}{c}$ spacetime indices, respectively.

\item Fix $p \in M$, and suppose $\mc{N}^-(p)$ is normalized and foliated by $T|_p$ and $t_p$.

\item Let $v_0$ be a constant such that $0 < v_0 \leq \mf{i}(p)$.

\item For each $1 \leq m \leq n$, we define the extrinsic tensor fields
\[ \sss{B}{m} \in \Gamma \ol{T}^{\sss{r}{m}} \mc{N}^-\paren{p} \text{,} \qquad \sss{A}{m} = t_p^{-1} \sss{B}{m} \in \Gamma \ol{T}^{\sss{r}{m}} \mc{N}^-\paren{p} \text{,} \]
along with a tensor $\sss{J}{m}$ of rank $\sss{r}{m}$ at $p$, such that the system
\begin{equation}\label{eq.transport_pre} \onasla_{t_p} \sss{B}{m}^I = -\frac{1}{2} \brak{\vartheta \paren{\trace \chi} - \frac{2}{t_p}} \sss{B}{m}^I + \frac{\vartheta}{2} \sum_{c = 1}^n \sss{P}{cm}_{4J}{}^I \sss{B}{c}^J \text{,} \qquad 1 \leq m \leq n \end{equation}
of transport equations is satisfied, along with the initial conditions
\begin{equation}\label{eq.transport_init} \valat{\sss{B}{m}}_p = \sss{J}{m} \text{,} \qquad 1 \leq m \leq n \text{.} \end{equation}

\item Define the ``error coefficients"
\[ \sss{\nu}{cm} \in \Gamma \ol{T}^{\sss{r}{c} + \sss{r}{m}} \mc{N}^-\paren{p} \text{,} \qquad 1 \leq m, c \leq n \]
by the formulas
\begin{align}
\label{eq.mass_asp_P} \sss{\nu}{cm}_J{}^I &= - \onasla^a \sss{P}{cm}_{aJ}{}^I + \frac{1}{2} \onasla_4 \sss{P}{cm}_{3J}{}^I + \zeta^a \sss{P}{cm}_{aJ}{}^I + \frac{1}{4} \paren{\trace \ul{\chi}} \sss{P}{cm}_{4J}{}^I \\
\notag &\qquad + \frac{1}{4} \paren{\trace \chi} \sss{P}{cm}_{3J}{}^I + \frac{1}{2} \sum_{d = 1}^n \sss{P}{cd}_{4J}{}^K \sss{P}{dm}_{3K}{}^I \text{.}
\end{align}
\end{itemize}
Then, we have the representation formula
\begin{equation}\label{eq.ksp_gen} 4 \pi \cdot n\paren{p} \cdot \sum_{m = 1}^n \sss{J}{m}^I \valat{\sss{\Phi}{m}_I}_p = \mf{F}\paren{p; v_0} + \mf{E}^1\paren{p; v_0} + \mf{E}^2\paren{p; v_0} + \mf{I}\paren{p; v_0} \text{,} \end{equation}
where:
\begin{itemize}
\item The ``fundamental solution term" $\mf{F}(p; v_0)$ is given by
\begin{equation}\label{eq.ksp_gen_F} \mf{F}\paren{p; v_0} = - \sum_{m = 1}^n \int_{\mc{N}^-\paren{p; v_0}} \sss{A}{m}^I \sss{\Psi}{m}_I \text{.} \end{equation}

\item The ``principal error terms" $\mf{E}^1(p; v_0)$ are given by
\begin{align}
\label{eq.ksp_gen_E1_app} \mf{E}^1\paren{p; v_0} &= - \sum_{m = 1}^n \int_{\mc{N}^-\paren{p; v_0}} \onasla^a \sss{A}{m}^I \onasla_a \sss{\Phi}{m}_I \\
\notag &\qquad + \sum_{m = 1}^n \int_{\mc{N}^-\paren{p; v_0}} \paren{\zeta^a - \ul{\eta}^a} \onasla_a \sss{A}{m}^I \sss{\Phi}{m}_I \text{.}
\end{align}

\item The remaining ``error terms" $\mf{E}^2(p; v_0)$ are given by
\begin{align}
\label{eq.ksp_gen_E2} \mf{E}^2\paren{p; v_0} &= \sum_{m = 1}^n \int_{\mc{N}^-\paren{p; v_0}} \mu \cdot \sss{A}{m}^I \sss{\Phi}{m}_I \\
\notag &\qquad + \frac{1}{2} \sum_{m = 1}^n \int_{\mc{N}^-\paren{p; v_0}} \sss{A}{m}^I R_{43}\brak{\sss{\Phi}{m}}_I \\
\notag &\qquad - \sum_{m, c = 1}^n \int_{\mc{N}^-\paren{p; v_0}} \sss{P}{cm}_{aJ}{}^I \onasla^a \sss{A}{c}^I \sss{\Phi}{m}_I \\
\notag &\qquad + \sum_{m, c = 1}^n \int_{\mc{N}^-\paren{p; v_0}} \sss{\nu}{cm}_J{}^I \cdot \sss{A}{c}^J \sss{\Phi}{m}_I \text{.}
\end{align}

\item The ``initial value terms" $\mf{I}(p; v_0)$ are given by
\begin{align}
\label{eq.ksp_gen_I} \mf{I}\paren{p; v_0} &= -\frac{1}{2} \sum_{m = 1}^n \int_{\mc{S}_{v_0}} \paren{\trace \ul{\chi}} \sss{A}{m}^I \sss{\Phi}{m}_I - \sum_{m = 1}^n \int_{\mc{S}_{v_0}} \sss{A}{m}^I D_3 \sss{\Phi}{m}_I \\
\notag &\qquad - \frac{1}{2} \sum_{m, c = 1}^n \int_{\mc{S}_{v_0}} \sss{P}{cm}_{3J}{}^I \sss{A}{c}^J \sss{\Phi}{m}_I \text{.}
\end{align}
\end{itemize}
Here, we have indexed with respect to arbitrary null frames $L, \ul{L}, e_1, e_2$ adapted to the $t_p$-foliation.
The capital letters $I, J$ refer to collections of extrinsic indices.
The symbols $\chi$, $\ul{\chi}$, $\zeta$, and $\ul{\eta}$ refer to the Ricci coefficients of $\mc{N}^-(p)$.
\end{theorem}

\begin{remark}
It is easy to see that Theorem \ref{thm.ksp_gen} is a special case of \cite[Thm. 7]{shao:ksp}.
The function $t_p$ is clearly a foliating function in the sense of \cite[Sec. 2.2]{shao:ksp}, and the associated null lapse $\vartheta$ has initial value $n(p)$ at $p$.
Moreover, both $\mc{N}^-(p)$ and $\mc{N}^-(p; v_0)$, as given in Theorem \ref{thm.ksp_gen}, are regular portions of the full null cone $N^-(p)$.
\end{remark}

\begin{remark}
Although the representation formula was stated in \eqref{eq.ksp_gen}-\eqref{eq.ksp_gen_I} in index notation, this was done only as a matter of convenience.
It is easy to see that these expressions can in fact be described invariantly.
\end{remark}

\begin{remark}
In particular, we can use \eqref{eq.ksp_gen} to examine the value of any $\sss{\Phi}{m}|_p$ individually by setting $\sss{J}{c} = 0$ for all $c \neq m$.
\end{remark}

Lastly, we note that the $\sss{A}{m}$'s in Theorem \ref{thm.ksp_gen} satisfy the transport equations
\begin{equation}\label{eq.transport} \onasla_L \sss{A}{m}^I = -\frac{1}{2} \paren{\trace \chi} \sss{A}{m}^I + \frac{1}{2} \sum_{c = 1}^n \sss{P}{cm}_{4J}{}^I \sss{A}{c}^J \text{,} \qquad 1 \leq m \leq n \text{.} \end{equation}

\subsection{Applying the Parametrix}\label{sec.ksp_app}

We now describe how Theorem \ref{thm.ksp_gen} is applied.
Let $0 < \delta_0 \leq 1$ be sufficiently small such that Theorem \ref{thm.nc} is satisfied.
\footnote{More explicitly, we assume that given any $q \in M_+$, then $\mf{i}(q) \geq \delta_1 = \min(\delta_0, t(q) - \tau_0)$, and the estimates \eqref{eq.pnc_main_est} and \eqref{eq.pnc_main_est_alt} hold on $\mc{N}_q = \mc{N}^-(q; \delta_1)$.}
Fix a point $p \in M_+$, fix another constant $0 < \delta \leq \min(\delta_0, t(p) - \tau_0)$, and define $\mc{N} = \mc{N}^-(p; \delta)$, normalized and foliated by $T|_p$ and $t_p$, as usual.

In the E-S case, we wish to find $L^\infty$-bounds for both $R$ and $D^2 \phi$ at $p$.
As a result, appropriating the notations of Theorem \ref{thm.ksp_gen}, we set $n = 2$, $\sss{\Phi}{1} = R$, and $\sss{\Phi}{2} = D^2 \phi$.
For the corresponding system of wave equations comprising \eqref{eq.tensor_wave_system}, we take \eqref{eq.curv_wave} and \eqref{eq.es_wave}.
Moreover, the right-hand sides of \eqref{eq.curv_wave} and \eqref{eq.es_wave} determine the $\sss{\Psi}{m}$'s and $\sss{P}{cm}$'s.
In particular, in the E-S case, the wave equations contain no first-order terms, that is, the $\sss{P}{cm}$'s vanish entirely.
\footnote{Consequently, the parametrix of \cite{kl_rod:ksp} suffices in the E-S setting.}

The E-M case is analogous, except that we must obtain $L^\infty$-bounds for $R$ and $D F$.
We apply Theorem \ref{thm.ksp_gen} by setting $n = 2$, $\sss{\Phi}{1} = R$, and $\sss{\Phi}{2} = D F$, and we adopt \eqref{eq.curv_wave} and \eqref{eq.em_wave} as the system of wave equations.
In contrast to the E-S setting, the first-order coefficients $\sss{P}{cm}$ are no longer trivial.
In particular, we see from the right-hand sides of \eqref{eq.curv_wave} and \eqref{eq.em_wave} that $\sss{P}{11}$ and $\sss{P}{22}$ vanish, while $\sss{P}{12}$ and $\sss{P}{21}$ are sums of terms, each of which can be expressed as tensor products and contractions of $F$ with instances of $g$.
More specifically, the $\sss{P}{12}$ and $\sss{P}{21}$ terms arise from the first, second, and seventh terms of \eqref{eq.curv_wave} in the E-M expansion of $\tilde{d}^2 \ric$, and from the first term of \eqref{eq.em_wave}.
The remaining terms on the right-hand sides of \eqref{eq.curv_wave} and \eqref{eq.em_wave} comprise the $\sss{\Psi}{m}$'s.

Now that we have determined the $\sss{P}{cm}$'s, we can determine their a priori bounds on $\mc{N}$. 
By the breakdown criterion \eqref{eq.bdc}, we have the uniform bounds
\begin{equation}\label{eq.bdc_bound_P} \nabs{\sss{P}{cm}}_{L^\infty\paren{\mc{N}}} \lesssim 1 \text{,} \qquad 1 \leq m, c \leq 2 \end{equation}
in both the E-S and E-M settings.
Furthermore, it follows from Propositions \ref{thm.e1_lest} and \ref{thm.null_flux} that in both the E-S and E-M cases,
\begin{equation}\label{eq.bdc_bound_DP} \nabs{\onasla \sss{P}{cm}}_{L^2\paren{\mc{N}}} + \nabs{\onasla_{t_p} \sss{P}{cm}}_{L^2\paren{\mc{N}}} \lesssim 1 \text{,} \qquad 1 \leq m, c \leq 2 \text{.} \end{equation}
Of course, in the E-S case, the bounds \eqref{eq.bdc_bound_P} and \eqref{eq.bdc_bound_DP} hold trivially.

Finally, we choose arbitrary tensors $\sss{J}{1}$ and $\sss{J}{2}$ at $p$, of the same ranks as $\sss{\Phi}{1}$ and $\sss{\Phi}{2}$, respectively.
For convenience, we adopt the abbreviation
\[ \abs{J} = \abs{\sss{J}{1}} + \abs{\sss{J}{2}} \text{.} \]
Applying Theorem \ref{thm.ksp_gen}, we obtain the expansion \eqref{eq.ksp_gen} for the quantity
\[ 4 \pi \cdot n\paren{p} \cdot \paren{\sss{J}{1}{}^I \valat{\sss{\Phi}{1}{}_I}_p + \sss{J}{2}^I \valat{\sss{\Phi}{2}{}_I}_p} \text{.} \]
Each term of this expansion is either an integral over $\mc{N}$ (i.e., the terms $\mf{F}(p; \delta)$, $\mf{E}^1(p; \delta)$, and $\mf{E}^2(p; \delta)$) or an ``initial value" integral over the $t_p$-level set $\mc{S}_\delta$ (i.e., the terms $\mf{I}(p; \delta)$).
Thus, in order to bound $\sss{\Phi}{1}$ and $\sss{\Phi}{2}$ at $p$, as desired, we must bound each of the integral terms mentioned above.

Before we can accomplish this, though, we must first control the fields $\sss{B}{m}$ and $\sss{A}{m}$ associated with the system of transport equations \eqref{eq.transport_pre}, \eqref{eq.transport_init}, and \eqref{eq.transport}.
This is the content of the subsequent proposition, which is the E-S and E-M analogue of the estimates \cite[Prop. 7.1, Prop. 7.3]{kl_rod:bdc} in the vacuum case.

\begin{proposition}\label{thm.transport_bound_bdc}
In both the E-S and E-M settings, we have
\[ \sum_{m = 1}^2 \nabs{\sss{B}{m}}_{L^\infty\paren{\mc{N}}} \lesssim \abs{J} \text{,} \qquad \sum_{m = 1}^2 \nabs{\onasla \sss{A}{m}}_{L^2\paren{\mc{N}}} \lesssim \abs{J} \text{,} \]
as long as $\delta_0$ is sufficiently small with respect to the fundamental constants.
\end{proposition}

\begin{proof}
By standard calculus computations analogous to those of Proposition \ref{thm.norm_power_rule_tf},
\begin{align*}
\onasla_{t_p} \abs{\sss{B}{m}}^2 &\lesssim \abs{\sss{B}{m}} \abs{\onasla_{t_p} \sss{B}{m}} + \abs{\onasla_{t_p} h} \abs{\sss{B}{m}}^2 \\
&\lesssim \abs{\sss{B}{m}} \brak{\abs{I} \abs{\sss{B}{m}} + \sum_{c = 1}^2 \abs{\sss{P}{cm}} \abs{\sss{B}{c}}} + \abs{\sss{B}{m}}^2 \text{,}
\end{align*} 
where $I = \vartheta (\trace \chi) - 2 t_p^{-1}$ and $h$ is the induced Riemannian metric on $M$.
Here, we have applied \eqref{eq.dft_bound} and \eqref{eq.transport_pre}.
Summing over $m$ and applying \eqref{eq.pnc_main_est_alt} and \eqref{eq.bdc_bound_P}, then
\[ \onasla_{t_p} \paren{\sum_{m = 1}^2 \abs{\sss{B}{m}}^2} \lesssim \paren{1 + \abs{I} + \sum_{c, m = 1}^2 \abs{\sss{P}{cm}}} \sum_{m = 1}^2 \abs{\sss{B}{m}}^2 \lesssim \sum_{m = 1}^2 \abs{\sss{B}{m}}^2 \text{.} \]
Applying Gr\"onwall's inequality to the above while taking into account the initial conditions \eqref{eq.transport_init}, we obtain the first desired estimate.

Next, we define
\[ \sss{U}{m} = t_p^2 \cdot \onasla \sss{A}{m} \in \Gamma \ul{T}^1 \ol{T}^{\sss{r}{m}} \mc{N}^-\paren{p} \text{,} \qquad \abs{V} = \max\paren{\abs{\onasla \sss{A}{1}}, \abs{\onasla \sss{A}{2}}} \text{.} \]
Differentating the transport equations \eqref{eq.transport} and commuting derivatives (see the commutation formula \cite[Prop. 3.12]{shao:bdc_nv}), we obtain the transport equations
\begin{align}
\label{eql.transport_D_bound_0} \onasla_{t_p} \sss{U}{m}_{aK} &= -I \cdot \sss{U}{m}_{aK} - \vartheta \hat{\chi}_a{}^b \sss{U}{m}_{bK} + \frac{1}{2} \vartheta \sum_{c = 1}^2 \sss{P}{cm}_4{}^C{}_K \sss{U}{c}_{aC} \\
\notag &\qquad - \frac{1}{2} t_p \nasla_a I \cdot \sss{B}{m}_K + t_p \vartheta \sum_{i=1}^{\sss{r}{m}} R_{\gamma_i}{}^\mu{}_{4a} \sss{B}{m}_{K^i_\mu} \\
\notag &\qquad + \frac{1}{2} t_p \sum_{c = 1}^2 \paren{\nasla_a \vartheta \sss{P}{cm}_4{}^C{}_K + \vartheta \onasla_a \sss{P}{cm}_4{}^C{}_K} \sss{B}{c}_C \text{,}
\end{align}
where $K = (\gamma_1, \ldots, \gamma_{\sss{r}{m}})$ and $C$ denote collections of $\sss{r}{m}$ and $\sss{r}{c}$ extrinsic indices, respectively, and where $K^i_\mu$ denotes $K$ except with the $i$-th index replaced by $\mu$.
In addition, we index both the horizontal and the extrinsic components of the $\sss{U}{m}$'s using local Fermi transported null frames.
\footnote{By this, we mean null frames $L$, $\ul{L}$, $e_1$, $e_2$ satisfying the conditions $\nasla_L e_a \equiv 0$.}
If we also (locally) define
\[ \mc{C} = \abs{\ol{D}_L e_1} + \abs{\ol{D}_L e_2} + \abs{\ol{D}_L L} + \abs{\ol{D}_L \ul{L}} \text{,} \]
then $\mc{C} \lesssim |\ul{\eta}| \lesssim 1$ by Proposition \ref{thm.pnc_tf}, and
\begin{equation}\label{eql.transport_D_bound_1} \abs{\onasla_{t_p} \paren{ \sss{U}{m}_{aK} }} \lesssim \abs{\onasla_{t_p} \sss{U}{m}} + \mc{C} \abs{\sss{U}{m}} \lesssim \abs{\onasla_{t_p} \sss{U}{m}} + \abs{\sss{U}{m}} \text{,} \end{equation}
where the left-hand side denotes $\onasla_{t_p}$ applied to the scalar quantity $\sss{U}{m}_{aK}$.

With consideration of the coefficients present in \eqref{eql.transport_D_bound_0}, we define
\begin{align*}
\mc{Q}_1 &= \abs{I} + \abs{\hat{\chi}} + \sum_{c, m = 1}^2 \abs{\sss{P}{cm}} \text{,} \\
\mc{Q}_2 &= \abs{\nasla I} + \abs{\nasla \vartheta} \sum_{c, m = 1}^2 \abs{\sss{P}{cm}} - Q_m\brak{R}\paren{T, L} + \sum_{c, m = 1}^2 \abs{\onasla \sss{P}{cm}} \text{.}
\end{align*}
In addition, by Proposition \ref{thm.e1_lest}, \eqref{eq.pnc_main_est}, Proposition \ref{thm.pnc_tf}, \eqref{eq.bdc_bound_P}, and \eqref{eq.bdc_bound_DP}, we obtain
\begin{equation}\label{eql.transport_D_bound_2} \nabs{\mc{Q}_1}_{L^\infty_\omega L^2_{t_p} \paren{\mc{N}}} \lesssim 1 \text{,} \qquad \nabs{\mc{Q}_2}_{L^2\paren{\mc{N}}} \lesssim 1 \text{.} \end{equation}

We now integrate \eqref{eql.transport_D_bound_0} along each null generator of $\mc{N}$.
By also noting Corollary \ref{thm.nl_bound} and \eqref{eql.transport_D_bound_1}, then we obtain the inequality
\begin{align*}
v^2 \valat{\abs{V}}_{\paren{v, \omega}} &\lesssim \int_0^v w^2 \valat{\paren{1 + \mc{Q}_1} \abs{V}}_{\paren{w, \omega}} dw + \abs{J} \int_0^v w \valat{\mc{Q}_2}_{\paren{w, \omega}} dw \\
&\lesssim \paren{\int_0^v w^4 \valat{\abs{V}^2}_{\paren{w, \omega}} dw}^\frac{1}{2} + \abs{J} \int_0^v w \valat{\mc{Q}_2}_{\paren{w, \omega}} dw \text{,}
\end{align*}
for each $0 < v < \delta$ and $\omega \in \Sph^2$, where we also applied the preceding bound $|\sss{B}{m}| \lesssim |J|$ and the estimate \eqref{eql.transport_D_bound_2} for $\mc{Q}_1$.
Dividing both sides of the above by $v$, taking $L^2_t L^2_\omega$-norms of the resulting inequality (cf. \eqref{eq.pnc_norm_param}), and then recalling the norm comparisons of Proposition \ref{thm.pnc_vol}, we obtain
\begin{align*}
\nabs{\abs{V}}_{L^2\paren{\mc{N}}}^2 &\lesssim \int_0^\delta \int_{\Sph^2} \paren{\int_0^v w^2 \valat{\abs{V}^2}_{\paren{w, \omega}} dw} d\omega dv \\
&\qquad + \abs{J}^2 \int_{\Sph^2} \int_0^\delta \paren{v^{-1} \int_0^v w \valat{\mc{Q}_2}_{\paren{w, \omega}} dw}^2 dv d\omega \\
&\lesssim \delta_0 \nabs{\abs{V}}_{L^2\paren{\mc{N}}}^2 + \abs{J}^2 \int_{\Sph^2} \int_0^\delta v^2 \valat{\mc{Q}_2}_{\paren{v, \omega}} dv d\omega \\
&\lesssim \delta_0 \nabs{\abs{V}}_{L^2\paren{\mc{N}}}^2 + \abs{J}^2 \text{.}
\end{align*}
where we applied Proposition \ref{thm.pnc_vol} to the $V$-term, and where we applied Hardy's inequality and \eqref{eql.transport_D_bound_2} to the $\mc{Q}_2$-term.
Shrinking $\delta_0$ if necessary (still depending only on the fundamental constants), then $\||V|\|_{L^2(\mc{N})} \lesssim |J|$, completing the proof.
\end{proof}

\subsection{The Uniform Bounds}

We described in Section \ref{sec.ksp_app} how Theorem \ref{thm.ksp_gen} is applied to the breakdown problem.
With the preliminary work in place, we must now bound each of the resulting terms $\mf{F}(p, \delta)$, $\mf{E}^1(p, \delta)$, $\mf{E}^2(p, \delta)$, and $\mf{I}(p, \delta)$.
For convenience, let $\Theta$ denote $D^2 \phi$ in the E-S case and $D F$ in the E-M case.
In addition, we define for every $\tau_0 \leq \tau < t_1$ the uniform norm
\[ \mc{S}\paren{\tau} = \nabs{R}_{L^\infty\paren{\Sigma_\tau}} + \nabs{\Theta}_{L^\infty\paren{\Sigma_\tau}} \text{.} \]

We start with the principal ``fundamental solution" term $\mf{F}(p; \delta)$, given by \eqref{eq.ksp_gen_F}.
First of all, we decompose the integrands $\sss{A}{m}^I \sss{\Psi}{m}_I$ using local null frames.
Since $t_p \sss{A}{m} = \sss{B}{m}$ is uniformly bounded by Proposition \ref{thm.transport_bound_bdc}, then the main challenge is to bound the integral of every component $t_p^{-1} \sss{\Psi}{m}_I$ along $\mc{N}$.
The primary observation is the following, which is also the cornerstone of \cite{ea_mo:g_ymh, kl_rod:bdc}.
From \eqref{eq.curv_wave}-\eqref{eq.em_wave}, we see that each term of any component $\sss{\Psi}{m}_I$ is of the schematic form $W^2 \cdot Q$ or $Q^2$, where $\|W\|_{L^\infty(\mc{N})} \lesssim 1$, and where $Q$ denotes a tensor field of the form $R$ or $\Theta$.
\footnote{Recall that in the E-M case, the first-order terms (corresponding to the $\sss{P}{cm}$'s) are omitted from the $\sss{\Psi}{m}$'s.  In fact, the reason that the original parametrix of \cite{kl_rod:ksp} fails here is precisely because these first-order terms are not of the above forms $W^2 \cdot Q$ or $Q^2$.}

The first form $W^2 \cdot Q$ can be sufficiently controlled as follows:
\[ \int_{\mc{N}} \abs{\sss{A}{m}} \abs{W}^2 \abs{Q} \lesssim \abs{J} \int_{\mc{N}} t_p^{-1} \abs{Q} \lesssim \delta^2 \abs{J} \sup_{t\paren{p} - \delta \leq \tau \leq t\paren{p}} \mc{S}\paren{\tau} \text{.} \]
Note in particular that we have applied Proposition \ref{thm.transport_bound_bdc} to control the $\sss{A}{m}$'s.
The terms of type $Q^2$ require a more detailed null frame decomposition.
The main observation is the following: for each component of every $Q^2$-term, one of the $Q$-factors is bounded by the flux density $[\rho^1(p, t(p) - \delta)]^{1/2}$, which can be controlled using Proposition \ref{thm.e1_lest}.
We shall further elaborate on this point later in this section.

Assuming for now the above observation, then
\[ \int_{\mc{N}} \sss{A}{m} \cdot Q^2 \lesssim \abs{J} \brak{\mc{F}^1\paren{p; t\paren{p} - \delta}}^\frac{1}{2} \paren{\int_{\mc{N}} t_p^{-2} \abs{Q}^2}^\frac{1}{2} \lesssim \delta^\frac{1}{2} \abs{J} \sup_{t\paren{p} - \delta \leq \tau \leq t\paren{p}} \mc{S}\paren{\tau} \]
by Proposition \ref{thm.e1_lest}.
Combining the above, we obtain the estimate
\begin{equation}\label{eq.bdc_ksp_F} \mf{F}\paren{p; \delta} \lesssim \delta^\frac{1}{2} \abs{J} \sup_{t\paren{p} - \delta \leq \tau \leq t\paren{p}} \mc{S}\paren{\tau} \lesssim \delta_0^\frac{1}{2} \abs{J} \sup_{t\paren{p} - \delta \leq \tau \leq t\paren{p}} \mc{S}\paren{\tau} \text{.} \end{equation}

\begin{remark}
Recall that $\delta_0 > 0$ is the constant, depending only on the fundamental constants, for which Theorem \ref{thm.nc} holds, while $0 < \delta \leq \min(\delta_0, t(p) - \tau_0)$.
\end{remark}

We now discuss handling the $Q^2$-terms in further detail.
The process is analogous to the E-V case of \cite{kl_rod:bdc}, except we must consider several additional terms involving the matter field.
Recall that the exact null frame components controlled by the flux density $\rho^1(p, t(p) - \delta)$ were determined in detail in Proposition \ref{thm.null_flux} and in the subsequent discussions.
More specifically, the only components of $R$ and $\Theta$ that are not controlled are those of the forms $R_{3a3b}$, $D_{33} \phi$ (in the E-S case), and $D_3 F_{3a}$ (in the E-M case).
As a result, we must determine that for every $Q^2$-term in the $\sss{\Psi}{m}$'s, at least one of the $Q$'s is not an ``invalid" component.

We give a few examples here demonstrating the reasoning described above:
\begin{itemize}
\item Consider the term $I_{\alpha\beta\gamma\delta} = \mf{A}_{\alpha\beta} [D_{\alpha\gamma} \phi D_{\beta\delta} \phi]$ in the expansion of \eqref{eq.curv_wave} in the E-S case.
By the above characterizations, the only invalid component, for which neither factor is controlled, is $D_{33} \phi D_{33} \phi$.
To obtain this component, we require $\alpha = \beta = \gamma = \delta = 3$; but, by the antisymmetry between $\alpha$ and $\beta$, the component $I_{3333}$ vanishes.
Consequently, any component of $I_{\alpha\beta\gamma\delta}$ can be bounded by $-Q_w[D \phi](T, L) \cdot |D^2 \phi| \lesssim [\rho^1(p, t(p) - \delta)]^{1/2} \cdot |D^2 \phi|$.

\item The term $I_{\alpha\beta\gamma\delta} = D^\mu F_{\alpha\beta} D_\mu F_{\gamma\delta}$ in the E-M expansion of \eqref{eq.curv_wave} can be handled similarly.
Because of the contraction involving $\mu$ and the properties of null frames, for each term of the summation, one of the $\mu$'s must not be ``$3$".
As a result, one $D F$ factor can always be bounded by $-Q_w[F](T, L)$.

\item The terms $R_\alpha{}^\lambda{}_\beta{}^\mu R_{\mu\lambda\gamma\delta}$ and $\mf{A}_{\alpha\beta} [R_\alpha{}^\mu{}_\gamma{}^\lambda R_{\beta\mu\delta\lambda}]$ in \eqref{eq.curv_wave} correspond with the $R^2$-terms present in the E-V setting in \cite{kl_rod:bdc}.
They can be handled using signature considerations, as in \cite[Sec. 5.4]{kl_rod:bdc}, or directly using the same reasoning as for the previous two examples.
More explicitly, we see that at least one of the $R$'s is controlled by $-Q_m[R](T, L)$.

\item For the term $R_\alpha{}^\mu{}_\beta{}^\lambda D_{\mu\lambda} \phi$ in \eqref{eq.es_wave}, due to the contractions, for each term in this summation, the ``$D^2 \phi$" factor can be bounded by $-Q_w[D \phi](T, L)$ unless $\mu = \lambda = 3$ in $D^2 \phi$.
If this is true, however, then $\mu = \lambda = 4$ in the ``$R$" factor, which can hence be bounded by $-Q_m[R](T, L)$.
\end{itemize}
The remaining $Q^2$-terms in \eqref{eq.curv_wave}-\eqref{eq.em_wave} can be handled using similar reasoning.

The next task is to bound the ``error terms" $\mf{E}^1(p; \delta)$ and $\mf{E}^2(p; \delta)$, defined in \eqref{eq.ksp_gen_E1_app} and \eqref{eq.ksp_gen_E2}.
First, by Propositions \ref{thm.e2_lest} and \ref{thm.transport_bound_bdc}, we have
\begin{align*}
\int_{\mc{N}} \abs{\onasla \sss{A}{m}} \abs{\onasla \sss{\Phi}{m}} &\lesssim \nabs{\onasla \sss{A}{m}}_{L^2\paren{\mc{N}}} \nabs{\onasla \sss{\Phi}{m}}_{L^2\paren{\mc{N}}} \\
&\lesssim \abs{J} \brak{\mc{E}^2\paren{t\paren{p} - \delta}^\frac{1}{2} + \delta^\frac{1}{2} \sup_{t\paren{p} - \delta \leq \tau \leq t\paren{p}} \mc{S}\paren{\tau}} \text{.}
\end{align*}
Similarly, by \eqref{eq.pnc_main_est} and Proposition \ref{thm.transport_bound_bdc},
\begin{align*}
\int_{\mc{N}} \abs{\zeta - \ul{\eta}} \abs{\onasla \sss{A}{m}} \abs{\sss{\Phi}{m}} &\lesssim \nabs{\zeta - \ul{\eta}}_{L^2\paren{\mc{N}}} \nabs{\onasla \sss{A}{m}}_{L^2\paren{\mc{N}}} \nabs{\sss{\Phi}{m}}_{L^\infty\paren{\mc{N}}} \\
&\lesssim \delta \abs{J} \sup_{t\paren{p} - \delta \leq \tau \leq t\paren{p}} \mc{S}\paren{\tau} \text{.}
\end{align*}
Therefore, we obtain the bound
\begin{equation}\label{eq.bdc_ksp_E1} \mf{E}^1\paren{p; \delta} \lesssim \brak{\mc{E}^2\paren{t\paren{p} - \delta}}^\frac{1}{2} + \delta_0^\frac{1}{2} \abs{J} \sup_{t\paren{p} - \delta \leq \tau \leq t\paren{p}} \mc{S}\paren{\tau} \text{.} \end{equation}

Next, for the first term of \eqref{eq.ksp_gen_E2},
\begin{align*}
\int_{\mc{N}} \abs{\mu} \abs{\sss{A}{m}} \abs{\sss{\Phi}{m}} &\lesssim \nabs{\mu}_{L^2\paren{\mc{N}}} \nabs{\sss{A}{m}}_{L^2\paren{\mc{N}}} \nabs{\sss{\Phi}{m}}_{L^\infty\paren{\mc{N}}} \\
&\lesssim \delta^\frac{1}{2} \abs{J} \sup_{t\paren{p} - \delta \leq \tau \leq t\paren{p}} \mc{S}\paren{\tau} \text{,}
\end{align*}
where we appealed to \eqref{eq.pnc_main_est_alt} to bound $\mu$.
For the second term, by expanding $R_{43}[\sss{\Phi}{m}]$ using null frames, we see that each term of this expansion is bounded by $-Q_m[R](T, L) \cdot |\sss{\Phi}{m}|$.
As a result,
\[ \int_{\mc{N}} \abs{\sss{A}{m}} \abs{R_{43}\brak{\sss{\Phi}{m}}} \lesssim \delta^\frac{1}{2} \abs{J} \sup_{t\paren{p} - \delta \leq \tau \leq t\paren{p}} \mc{S}\paren{\tau} \text{.} \]
By \eqref{eq.bdc_bound_P} and Proposition \ref{thm.transport_bound_bdc}, we can bound the third term:
\begin{align*}
\int_{\mc{N}} \abs{\sss{P}{cm}} \abs{\onasla \sss{A}{c}} \abs{\sss{\Phi}{m}} &\lesssim \nabs{\sss{P}{cm}}_{L^2\paren{\mc{N}}} \nabs{\onasla \sss{A}{c}}_{L^2\paren{\mc{N}}} \nabs{\sss{\Phi}{m}}_{L^\infty\paren{\mc{N}}} \\
&\lesssim \delta^\frac{3}{2} \abs{J} \sup_{t\paren{p} - \delta \leq \tau \leq t\paren{p}} \mc{S}\paren{\tau} \text{.}
\end{align*}

For the last term, we will need the following simple estimate:

\begin{lemma}\label{thm.maf_P_bdc}
The coefficients $\sss{\nu}{cm}$, as defined in \eqref{eq.mass_asp_P}, satisfy
\[ \nabs{\sss{\nu}{cm}}_{L^2\paren{\mc{N}}} \lesssim 1 \text{,} \qquad c, m \in \brac{1, 2} \text{.} \]
\end{lemma}

\begin{proof}
We bound each term on the right-hand side of \eqref{eq.mass_asp_P}.
The first two terms are trivially bounded using \eqref{eq.bdc_bound_DP}; by Theorem \ref{thm.nc}, Proposition \ref{thm.pnc_tf}, and \eqref{eq.bdc_bound_P}, the third, fourth, and fifth terms are bounded by $\delta_0^{1/2}$.
The final term on the right-hand side of \eqref{eq.mass_asp_P} is trivially bounded using \eqref{eq.bdc_bound_P}.
\end{proof}

Finally, applying Lemma \ref{thm.maf_P_bdc}, then
\begin{align*}
\int_{\mc{N}} \abs{\sss{\nu}{cm}} \abs{\sss{A}{c}} \abs{\sss{\Phi}{m}} &\lesssim \nabs{\sss{\nu}{cm}}_{L^2\paren{\mc{N}}} \nabs{\sss{A}{c}}_{L^2\paren{\mc{N}}} \nabs{\sss{\Phi}{m}}_{L^\infty\paren{\mc{N}}} \\
&\lesssim \delta^\frac{1}{2} \abs{J} \sup_{t\paren{p} - \delta \leq \tau \leq t\paren{p}} \mc{S}\paren{\tau} \text{.}
\end{align*}
As a result,
\begin{equation}\label{eq.bdc_ksp_E2} \mf{E}^2\paren{p; \delta} \lesssim \delta_0^\frac{1}{2} \abs{J} \sup_{t\paren{p} - \delta \leq \tau \leq t\paren{p}} \mc{S}\paren{\tau} \text{.} \end{equation}

It remains only to bound the initial value terms $\mf{I}(p; \delta)$ in \eqref{eq.ksp_gen_I}.
The first term on the right-hand side of \eqref{eq.ksp_gen_I} can be handled using Proposition \ref{thm.pnc_vol} and \ref{thm.pnc_tf}:
\[ \int_{\mc{S}_\delta} \paren{\trace \ul{\chi}} \abs{\sss{A}{m}} \abs{\sss{\Phi}{m}} \lesssim \delta^{-2} \abs{J} \int_{\mc{S}_\delta} \abs{\sss{\Phi}{m}} \lesssim \abs{J} \nabs{\sss{\Phi}{m}}_{L^\infty\paren{\Sigma_{t\paren{p} - \delta}}} \text{.} \]
Applying \eqref{eq.sob_2e} and Proposition \ref{thm.e1_est}, then the above is bounded by
\[ \abs{J} \brak{1 + \mc{E}^2\paren{t\paren{p} - \delta} + \mc{E}^3\paren{t\paren{p} - \delta}}^\frac{1}{2} \text{.} \]
The third term of \eqref{eq.ksp_gen_I} can be bounded similarly:
\[ \int_{\mc{S}_\delta} \abs{\sss{P}{cm}} \abs{\sss{A}{c}} \abs{\sss{\Phi}{m}} \lesssim \delta \abs{J} \brak{1 + \mc{E}^2\paren{t\paren{p} - \delta} + \mc{E}^3\paren{t\paren{p} - \delta}}^\frac{1}{2} \text{.} \]

Lastly, for the remaining term of \eqref{eq.ksp_gen_I}, we will need the following trace estimate:

\begin{lemma}\label{thm.trace_tf}
If $Z \in \Gamma \ol{\mc{T}} \mc{N}^-(p)$ and $W \in \Gamma \mc{T} M$ are of the same rank, then
\[ \abs{\int_{\mc{S}_\delta} Z^I W_I} \lesssim \delta^\frac{3}{2} \nabs{Z}_{L^\infty\paren{\mc{S}_\delta}} \paren{\nabs{\ol{\nabla} W}_{L^2\paren{\Sigma_{t\paren{p} - \delta}}} + \nabs{W}_{L^2\paren{\Sigma_{t\paren{p} - \delta}}}} \text{.} \]
\end{lemma}

\begin{proof}
See \cite[Prop. 7.3]{shao:bdc_nv}.
The proof involves constructing a radial foliation of $\Sigma_{t(p) - \delta}$ using the level sets of a one-parameter family of null cones.
This relies heavily on Theorem \ref{thm.nc}, in particular the null injectivity radius bounds.
\end{proof}

We now apply Lemma \ref{thm.trace_tf} with $Z = \ul{L} \otimes \sss{A}{m}$ and $W = D \sss{\Phi}{m}$.
Since $Z$, as defined above, satisfies $\|Z\|_{L^\infty(\mc{S}_\delta)} \lesssim \delta^{-1} |J|$, then
\begin{align*}
\int_{\mc{S}_\delta} \ul{L}^\alpha \sss{A}{m}^I D_\alpha \sss{\Phi}{m}_I &\lesssim \delta^\frac{1}{2} \abs{J} \paren{\nabs{D^2 \sss{\Phi}{m}}_{L^2\paren{\Sigma_{t\paren{p} - \delta}}} + \nabs{D \sss{\Phi}{m}}_{L^2\paren{\Sigma_{t\paren{p} - \delta}}}} \\
&\lesssim \delta^\frac{1}{2} \abs{J} \brak{\mc{E}^2\paren{t\paren{p} - \delta} + \mc{E}^3\paren{t\paren{p} - \delta}}^\frac{1}{2} \text{.}
\end{align*}
As a result, we obtain
\begin{equation}\label{eq.bdc_ksp_I} \mf{I}\paren{p; \delta} \lesssim \abs{J} \brak{1 + \mc{E}^2\paren{t\paren{p} - \delta} + \mc{E}^3\paren{t\paren{p} - \delta}}^\frac{1}{2} \text{.} \end{equation}

Combining \eqref{eq.ksp_gen} and \eqref{eq.bdc_ksp_F}-\eqref{eq.bdc_ksp_I}, we obtain
\begin{align*}
\abs{\sum_{1 \leq a \leq 2} \sss{J}{a}^I \valat{\sss{\Phi}{a}_I}_p} &\lesssim \delta_0^\frac{1}{2} \abs{J} \sup_{t\paren{p} - \delta \leq \tau \leq t\paren{p}} \mc{S}\paren{\tau} \\
&\qquad + \abs{J} \brak{1 + \mc{E}^2\paren{t\paren{p} - \delta} + \mc{E}^3\paren{t\paren{p} - \delta}}^\frac{1}{2} \text{.}
\end{align*}
If we consider all possible values of $\sss{J}{1}$ and $\sss{J}{2}$, then we have
\begin{equation}\label{eq.bdc_ksp_pre} \valat{\abs{R}}_p + \valat{\abs{\Theta}}_p \lesssim \delta_0^\frac{1}{2} \sup_{t\paren{p} - \delta \leq \tau \leq t\paren{p}} \mc{S}\paren{\tau} + \brak{1 + \mc{E}^2\paren{t\paren{p} - \delta} + \mc{E}^3\paren{t\paren{p} - \delta}}^\frac{1}{2} \text{.} \end{equation}

Finally, if we fix $\tau$ such that $\tau_0 < \tau < t_1$ and $\tau - \delta \geq \tau_0$, and we apply \eqref{eq.bdc_ksp_pre} to every $p \in M_+$ satisfying $\tau - \delta \leq t(p) \leq \tau$, then we obtain
\[ \sup_{\tau - \delta \leq \tau^\prime \leq \tau} \mc{S}\paren{\tau^\prime} \lesssim \delta_0^\frac{1}{2} \sup_{\tau - \delta \leq \tau^\prime \leq \tau} \mc{S}\paren{\tau^\prime} + \brak{1 + \mc{E}^2\paren{\tau - \delta} + \mc{E}^3\paren{\tau - \delta}}^\frac{1}{2} \text{.} \]
By taking $\delta_0$ sufficiently small, we have proven the following:

\begin{proposition}\label{thm.bdc_unif}
Fix $\tau_0 < \tau < t_1$.
Suppose $\delta_0 > 0$ is sufficiently small, depending only on the fundamental constants, and let $0 < \delta \leq \delta_0$ such that $\tau - \delta \geq \tau_0$.
Then,
\[ \sup_{\tau - \delta \leq \tau^\prime \leq \tau} \mc{S}\paren{\tau^\prime}^2 \lesssim 1 + \mc{E}^2\paren{\tau - \delta} + \mc{E}^3\paren{\tau - \delta} \text{.} \]
\end{proposition}

\subsection{Completion of the Proof}

With the uniform bounds of Proposition \ref{thm.bdc_unif} in place, we can now embark on the final stretch of the proof of the main theorem.
We begin by applying Proposition \ref{thm.bdc_unif} to the higher-order energy estimates of Propositions \ref{thm.e2_est} and \ref{thm.e3_est}.
This yields higher-order energy estimates.

\begin{proposition}\label{thm.eho_est}
For every $\tau_0 \leq \tau < t_1$, we have
\[ \mc{E}^2\paren{\tau} + \mc{E}^3\paren{\tau} \lesssim 1 \text{.} \]
\end{proposition}

\begin{proof}
Let $\tau_0 < \tau < t_1$ and $0 < \delta \leq \delta_0$ such that $\tau - \delta \geq \tau_0$, where $\delta_0$ was determined in Proposition \ref{thm.bdc_unif}.
Applying Proposition \ref{thm.bdc_unif} to Proposition \ref{thm.e2_est} yields
\begin{align*}
\mc{E}^2\paren{\tau} &\lesssim \mc{E}^2\paren{\tau - \delta} + \delta \brak{1 + \mc{E}^2\paren{\tau - \delta} + \mc{E}^3\paren{\tau - \delta}} \\
&\lesssim 1 + \mc{E}^2\paren{\tau - \delta} + \mc{E}^3\paren{\tau - \delta} \text{.}
\end{align*}
Combining the above with Propositions \ref{thm.e3_est} and \ref{thm.bdc_unif}, then
\begin{align*}
\label{eql.bdc_est_3} \mc{E}^3\paren{\tau} &\lesssim \mc{E}^3\paren{\tau - \delta} + \delta \brak{1 + \mc{E}^2\paren{\tau - \delta} + \mc{E}^3\paren{\tau - \delta}}^2 \\
\notag &\lesssim 1 + \brak{\mc{E}^2\paren{\tau - \delta} + \mc{E}^3\paren{\tau - \delta}}^2 \text{.}
\end{align*}
As a result, we have derived
\begin{equation}\label{eq.eho_est_iter} \mc{E}^2\paren{\tau} +  \mc{E}^3\paren{\tau} \lesssim 1 + \brak{\mc{E}^2\paren{\tau - \delta} + \mc{E}^3\paren{\tau - \delta}}^2 \text{.} \end{equation}

We can now iterate \eqref{eq.eho_est_iter}.
First, we let $\delta^\prime = \min(\delta_0, t_1 - \tau_0)$, and we apply \eqref{eq.eho_est_iter} with $0 < \delta < \delta^\prime$ and $\tau \leq \tau_0 + \delta^\prime$ to obtain
\[ \sup_{\tau_0 \leq \tau^\prime < \tau_0 + \delta^\prime} \brak{\mc{E}^2\paren{\tau^\prime} +  \mc{E}^3\paren{\tau^\prime}} \lesssim 1 \text{.} \]
If $\delta^\prime = t_1 - \tau_0$ (i.e., $t_1 - \tau_0 \leq \delta_0$), then the proof is complete.
On the other hand, if $\delta^\prime = \delta_0 < t_1 - \tau_0$, then \eqref{eq.eho_est_iter} also yields
\[ \mc{E}^2\paren{\tau_0 + \delta_0} + \mc{E}^3\paren{\tau_0 + \delta_0} \lesssim 1 \text{.} \]
We can now repeat the above procedure, but with $\tau_0$ replaced by $\tau_0 + \delta_0$.
By this process, we can reach the breakdown time $t_1$ within a finite number of iterations.
This implies the desired result and completes the proof.
\end{proof}

Now that the higher-order spacetime quantities are controlled, we must do the same for the corresponding horizontal quantities.
Similar to the a priori estimates, we accomplish this by relating the corresponding spacetime and horizontal quantities to each other and then applying elliptic estimates.
The process is analogous to that of Sections \ref{sec.ell_pre} and \ref{sec.ell}, except the computations are somewhat more involved.
Consequently, we only list the final results and omit the proofs.

\begin{proposition}\label{thm.ell_ho}
For any $\tau_0 \leq \tau < t_1$, the following inequalities hold:
\[ \mf{R}\paren{\tau} + \mf{K}\paren{\tau} + \mf{f}\paren{\tau} \lesssim 1 \text{.} \]
\end{proposition}

\begin{proof}
This is a consequence of \cite[Lemma 9.4]{shao:bdc_nv} and Proposition \ref{thm.eho_est}.
\end{proof}

\begin{remark}
The quantities $\mf{R}(\tau)$, $\mf{K}(\tau)$, and $\mf{f}(\tau)$ were defined in \eqref{eq.bdc_energy_KR}-\eqref{eq.bdc_energy_fem}.
\end{remark}

We have now successfully controlled the required horizontal energy quantities.
From Section \ref{sec.outline}, it remains only to control the diameters and the injectivity radii of the $\Sigma_\tau$'s.
First of all, Propositions \ref{thm.sobolev_tf} and \ref{thm.ell_ho} imply uniform bounds for
\[ \nabs{\mc{R}}_{L^\infty\paren{\Sigma_\tau}} \text{,} \qquad \tau_0 \leq \tau < t_1 \text{.} \]
As a result, combining the above with Proposition \ref{thm.vol_radius}, we obtain the desired uniform control for the injectivity radii of the $\Sigma_\tau$'s.
\footnote{More specifically, the uniform lower bound on the injectivity radii follows from uniform bounds on the curvature and on the volume radii of geodesic balls; see, e.g., \cite{ch:fin, ch_gr_tay:fps}.}

The diameters of the $\Sigma_\tau$'s can be controlled by even cruder means using Proposition \ref{thm.unif_ell_prop}.
For example, let $p, q \in \Sigma_\tau$, where $\tau_0 \leq \tau < t_1$.
Consider the normal transports $p_0, q_0$ of $p, q$ to $\Sigma_{\tau_0}$, and fix a curve $\alpha_0$ in $\Sigma_{\tau_0}$ from $p_0$ to $q_0$ whose length is controlled by the diameter of $\Sigma_{\tau_0}$.
By using the transported coordinate systems of Proposition \ref{thm.unif_ell_prop} and the uniform ellipticity property of \eqref{eq.unif_ell}, then we can bound the length of the normal transport $\alpha$ of $\alpha_0$ to $\Sigma_\tau$.
This bound depends only on the fundamental constants, including the diameter of $\Sigma_{\tau_0}$.
This controls the diameters of the $\Sigma_\tau$'s and finally completes the proof of Theorem \ref{thm.bdc}.

\bibliographystyle{amsplain}
\bibliography{bib}

\end{document}